\documentclass[10pt]{amsart}

\usepackage{cite,amsthm,amsfonts,amsmath,amscd,mathrsfs,enumitem,eucal}
\usepackage{accents, graphicx}
\usepackage[usenames]{color}
\usepackage[colorlinks=true, linkcolor=webgreen, citecolor=webgreen, urlcolor=webbrown]{hyperref}
\usepackage[initials]{amsrefs}
\usepackage[T1]{fontenc}
\usepackage{tabularx}
\usepackage{booktabs}
\usepackage{tikz}
\usetikzlibrary{positioning,graphs,fit}
\usepackage{latexsym}
\usepackage[psamsfonts]{amssymb}
\usepackage{dynkin-diagrams}
\pgfkeys{/Dynkin diagram, edge length=0.7cm}
\usepackage{multirow}
\usepackage{makecell}

\definecolor{webgreen}{rgb}{0,.4,0}
\definecolor{webbrown}{rgb}{.4,0,0}

\newtheorem{Thm}{Theorem}[section]

\newtheorem{statement}{Theorem}

\newtheorem{lemma}[Thm]{Lemma}
\newtheorem{Prop}[Thm]{Proposition}
\newtheorem{Cor}[Thm]{Corollary}

\theoremstyle{definition}

\theoremstyle{remark}

\newtheorem{remark}[Thm]{Remark}
\newtheorem{Exp}[Thm]{Example}
\newtheorem{convention}[Thm]{Convention}

\numberwithin{equation}{section}

\newcommand\mybar{\kern1pt\rule[-0.6\dp\strutbox]{0.5pt}{0.9\baselineskip}\kern1pt}

\newcommand{\comment}[1]{}

  \def\stab{\operatorname{Stab }}
  
 \def\Q{\mathbb{Q}}\def\Z{\mathbb{Z}}\def\RR{\mathbb{RR}}
\def\le{\leqslant} \def\ge{\geqslant}

 \def\CC{\mathfrak{C}}
  \def\RR{\mathbb{R}}
\def\F{\mathbb F}
  
\def\CC{\mathbb{C}}
 \def\a{\mathfrak a}
\def\e{\varepsilon} \def\DD{\Delta} \def\G{\Gamma}\def\om{\omega}

\def\a{\alpha}\def\b{\beta}\def\g{\gamma}
\def\+{\,+\,} \def\m{\,-\,} \def\={\;=\;}

\def\be{\begin{equation}}  \def\ee{\end{equation}}

\def\SS{\mathcal{S}}

\def\res{\operatorname*{Res}}

\def\aa{\alpha}\def\ss{\sigma}
\def\th{\theta}

\def\dd{\delta}
\def\la{\langle} \def\ra{\rangle}

\def\xx{ \mathrm{\bf x} }
\def\bs{ \mathrm{\bf s} }

\def\ddd{ {\underline{\delta}} }

\def\ZZ{\mathcal{Z}}
\def\NN{\mathcal{N}}

\def\lam{\lambda}

\def\um{\underline{m}}
\def\ddd{{\underline{\delta}}}
\def\supp{\operatorname{Supp}}

\def\new{\mathrm{new}}

\def\evu{{x_i=1/u}}

\def\ux{\underline{x}}
\def\us{\underline{s}}
\def\sss{\scriptscriptstyle}

\def\CG{\mybar^{\raisebox{3pt}{$\scriptscriptstyle \rm CG$}}\hspace{-0.7em}}

\def\qres#1#2{\left(\frac{#1}{#2}\right)}

\address{\parbox{\linewidth}{School of Mathematics, 127 Vincent Hall,
 206 Church St. \!SE, Minneapolis, MN 55455, USA;\\
 Institute of Mathematics ``Simion Stoilow" of the Romanian Academy\\
Calea Grivi\c tei 21, Bucharest 010702, Romania}}
\email{cad@umn.edu}
\address{Department of Mathematics, University of Pittsburgh, Pittsburgh, PA 15260}
\email{bion@pitt.edu}
\address{ \parbox{\linewidth}{Institute of Mathematics ``Simion Stoilow" of the Romanian Academy\\
Calea Grivi\c tei 21, Bucharest 010702, Romania}}
\email{vpasol@gmail.com}
\address{ \parbox{\linewidth}{Institute of Mathematics ``Simion Stoilow" of the Romanian Academy\\
Calea Grivi\c tei 21, Bucharest 010702,  Romania}}
\email{aapopa@gmail.com}
\begin{document}
\title[]{Residues of quadratic Weyl group multiple Dirichlet series}
\author{Adrian Diaconu, Bogdan Ion, Vicen\c tiu Pa\c sol, Alexandru A. Popa}
\date{\today}%
\subjclass[2010]{}
\begin{abstract}We give explicit formulas for the 
residue of the Chinta-Gunnells average attached to a finite irreducible root system, 
at the polar divisor corresponding to a simple short root.  The formula describes the 
residue in terms of the average attached to the root subsystem orthogonal to 
the relevant simple root. As a consequence, we obtain similar formulas for the residues 
of quadratic Weyl group multiple Dirichlet series over the rational function field
and over the Gaussian field. The residue formula also allows us to obtain a new expression 
for the Chinta-Gunnells average of a finite irreducible root system, as an 
average over a maximal parabolic subgroup of a rational function that has 
an explicit description reflecting the combinatorics of the root system.
\end{abstract}

\maketitle


\section{Introduction}

\subsection{} The genesis of the concept of Weyl group multiple Dirichlet series (WMDS) can be traced back to the work of Goldfeld and Hoffstein \cite{GH85}, where (using present terminology) a quadratic double Dirichlet series over $\mathbb{Q}$, of Cartan type $A_2$, was constructed as the Mellin transform of an Eisenstein series of  half-integral weight for the congruence subgroup $\G_0(4)$;  the study of the same object, in an equivalent form, was previously proposed by Siegel \cite{S}. Other examples, obtained as integral transforms of Eisenstein series  (and other automorphic objects) on the metaplectic double covers of ${\rm GL}_3$ and ${\rm GSp(4)}$,  were investigated; for the relevant results, see the survey \cites{BFH96} and the references therein.  The main application at the time was to obtain non-vanishing results for quadratic twists of central values of automorphic $L$-functions and their derivatives. In higher rank, such constructions, based on integral transforms of Eisenstein series on covers of reductive groups are difficult to obtain and analyze.  These initial investigations revealed the structural properties of such multiple Dirichlet series, and it has gradually emerged  \cites{DGH, BFH96, BFH04, FF} that these properties can be used to define a class of multiple Dirichlet series without making use of integral transforms of automorphic forms.
The general principles used to construct and analyze multiple Dirichlet series associated to finite reduced root systems were laid out in \cites{BBCFH, BBFH07, CG, CG1}.

\subsection{}\label{sec: Z-global} Following \cites{CG, CG1},  a coarse description of the class of \emph{finite} Weyl group \emph{quadratic} multiple Dirichlet series proceeds as follows. Let $\mathbb{K}$ denote a global field, let $\Phi$ be a finite (reduced) root system of rank $r$, and let $W$ denote its Weyl group.  Let $S$ be a finite set of places, which includes the set of \emph{infinite} places,
and in characteristic $0$,  the set of places dividing $2$, and large enough so  that the ring $\mathcal{O}_S$ of  $S$-integers has class number $1$. The quadratic Weyl group multiple Dirichlet series attached to the root system $\Phi$ is a series of $r$ complex variables of the form
$$
\mathcal{Z}_\Phi(s_1,\dots,s_r) = \sum_{} \frac{H(m_1,\dots,m_r)}{|m_1|^{s_1}\cdot\ldots\cdot |m_r|^{s_r}},
$$
the sum ranging over the set of $r$-tuples of non-zero integers in $\mathcal{O}_S$ modulo units. The coefficients $H(m_1,\dots, m_r)$ are required to satisfy a \emph{twisted multiplicativity} property involving the quadratic symbol,  which reflects the combinatorics of the root system $\Phi$. The twisted multiplicativity reduces the description of $H(m_1,\dots,m_r)$ to the case where all components are powers of the same prime $p$. The generating series
$$
\sum H(p^{n_1},\dots,p^{n_r}) |p|^{-n_1s_1}\cdot \ldots\cdot |p|^{-n_rs_r}
$$
are called the $p$-parts of $\mathcal{Z}_\Phi$. Chinta and Gunnells~\cite{CG} have constructed the $p$-parts through an averaging technique that uses an action of $W$ on the space of rational functions in $r$ variables.  Their construction leads to the series
$\mathcal{Z}_\Phi$ that has  meromorphic continuation to $\mathbb{C}^r$ and satisfies a group of functional equations isomorphic to $W$.

\subsection{}\label{sec: Z-local}
More precisely, let $Q$ denote the root lattice of $\Phi$, and denote by $V$ the $\mathbb{R}$-span of $\Phi$. Let $\F=\Q(u)$, with $u$ a formal parameter. The standard basis of $\F[Q]$, the $\F$-group ring of $Q$, is denoted by $\{\xx^\lambda\}_{\lambda\in Q}$; $\F(Q)$ denotes the field of fractions of $\F[Q]$, viewing the latter as the space of
Laurent polynomials in the monomials $\xx^\lam$. We fix a basis  $\Pi(\Phi)=\{\a_i\}_{1\leq i\leq r}$ of $\Phi$; denote $x_i=\xx^{\a_i}$ and treat $\xx=(x_1,\dots,x_r)$ as a multivariable. Chinta and Gunnells  defined a rational function   in $\xx=(x_1,\dots,x_r)$ which, as power series,  has the form
$$
Z_\Phi(\xx;u)=
\sum a_{\lambda}(u) \xx^\lambda,
$$
with polynomial coefficients in the extra parameter $u$. {The parameter $u$ is present in the Chinta-Gunnells action and formalizes the role played by the quadratic Gauss sum.}

The coefficients of the $p$-parts of $\mathcal{Z}_\Phi$ are (in our normalization, which is slightly different from the one in \cite{CG}) $$H(p^{n_1},\dots,p^{n_r})=a_{n_1\a_1+\cdots+n_r\a_r}(|p|^{-1/2}).$$
The function $Z_\Phi(\xx; u)$ is uniquely determined by its invariance under the Chinta-Gunnells action and the normalization $Z_\Phi(0; u)=1$. For $\mathbb{K}=\mathbb{F}_q(t)$, where $q\equiv 1\!\!\!\mod 4$ and $\mathbb{F}_q$ is the finite field with $q$ elements, we have (\cite{Fr1}*{Proposition 4.2}\footnote{In \cite{Fr1}, the WMDS $\mathcal{Z}_\Phi$ (denoted there by $\mathcal{Z}^*$) is constructed slightly differently, starting with a series $\mathcal{Z}$ with $p$-parts that correspond to the \emph{numerator} of $Z_\Phi(\xx; u)$, defined in
Convention~\ref{conv}. One can show that this construction  of  $\mathcal{Z}_\Phi$ is equivalent to the one described in \S\ref{sec: Z-global}-\ref{sec: Z-local}.}\!\!, generalizing prior observations in particular cases \cites{C,CM}),
\begin{equation}\label{eq-local-to-global}
\mathcal{Z}_\Phi(s_1,\dots,s_r)=Z_\Phi(q^{-s_1},\dots,q^{-s_r}; q^{1/2}).
\end{equation}
This is a manifestation of the same local-to-global phenomenon that classically connects the zeta function of the projective line and its Euler factors. In the case of the affine root system $D_4^{\sss (1)}$, the local-to-global theorem holds  \cite{DPP} with a correction factor that reflects the contribution of the imaginary roots. 

Aside from its role in the definition of WMDS, the Chinta-Gunnells average $Z_\Phi(\xx;u)$ has direct connections with spherical Whittaker functions on metaplectic covers of $p$-adic groups \cites{CO, McN16}, metaplectic Demazure-Lusztig operators \cite{CGP}, the combinatorial theory of crystal graphs \cites{BBF, McN11}, quantum groups and solvable lattice models in statistical mechanics \cite{BBCFG}. 
We also note that the Chinta-Gunnells action itself emerges canonically from the metaplectic representations of affine Hecke algebras \cite{SSV}.

\subsection{} Our main results give a precise description of the residues of the series $Z_\Phi(\xx,u)$, and of $\mathcal{Z}_\Phi(s_1,\dots,s_r)$ for $\mathbb{K}=\mathbb{F}_q(t)$, $q\equiv 1\!\!\!\mod 4$, and  for $\mathbb{K}=\Q(\sqrt{-1})$. Before stating them, let us introduce some notation; for the full details we refer to Section \ref{sec_CG}. The Chinta-Gunnells action  of $w\in W$ on $f\in \F(Q)$ will be denoted by $f{\CG}w$. Let $\la \cdot ,\cdot \ra$ denote the $W$-invariant inner product on $V$, normalized such that the short roots have square length $2$. If $\Phi$ is simply laced, all roots are considered to be short. The Chinta-Gunnells action depends on integers $m_\a$, $\a\in \Phi$, specified as follows: for short roots, $m_\a=2$; for long roots, $m_\a=2$ if $\Phi$ is of type $G_2$ and $m_\a=1$ otherwise. Because the long and short roots play similar roles for the root system of type $G_2$, and distinctively different roles for the other irreducible root systems, we treat the $G_2$  root system separately; we refer to  Appendix~\ref{sec_G2} for our results in this case. The statements included in this Introduction will assume that $\Phi$ is not of type $G_2$.

\subsection{}\label{sec: intro-notation}
Denote $\displaystyle\DD_{\Phi}(\xx)=\prod_{\a \in \Phi^+} (1-\xx^{m_\a \a})$ and  consider the related action $\displaystyle f\vert w = \frac{\DD_{\Phi}(\xx)}{\DD_{\Phi}(w\xx)} f\CG w $. By definition, the Chinta-Gunnells  zeta average is
\[
Z_{\Phi}(\xx;u)=\frac{\sum_{w\in W} 1|w }{\DD_{\Phi}(\xx)}.
\]
Clearly,  $Z_{\Phi}(\xx; u)=Z_\Phi(\xx; u) \CG w$, for all $w\in W$. For reducible root systems, we define  $Z_{\Psi\oplus\Psi^\prime}=Z_\Psi\cdot Z_{\Psi^\prime}$.

Fix $\a_i$ a short simple root. The function $Z_\Phi(\xx;u)$ has  a simple pole at  $x_i=1\slash u$, and we are interested
in the residue
\[
\res_{\evu}Z_\Phi(\xx; u) := \lim_{x_i\rightarrow 1\slash u} (1-u x_i) Z_\Phi(\xx;u).
\]
This can be described  in terms of the zeta average associated to the proper root sub-system $\Phi_0\subset \Phi$ that consists of the roots orthogonal on $\a_i$. The inclusion of root lattices $Q_0\subset Q$ induces canonical morphisms $\F(Q_0)\subset \F(Q)$. We regard $Z_{\Phi_0}(\xx; u)$ as an element of $\F(Q)$ in this fashion.

\begin{statement}\label{thmA}Let $\Phi$ be an irreducible root system not of type $G_2$, and let $\a_i$ be a short simple root. Then,
\begin{equation}\label{eqA}
\res_{x_i=1\slash u} Z_{\Phi}(\xx; u)=
Z_{\Phi_0}(\xx; u)\vert_{x_{i}=1\slash u} \cdot \prod_{{\substack{\a\in\Phi^+ \\ \la\a,\a_i \ra=1}}}  \frac{1}{ (1-u^2\xx^{2\a}){\vert_{x_{i}=1\slash u}}}.
\end{equation}
\end{statement}
We note that the zeta average of the root system of rank $1$ is $\displaystyle Z_{A_1}(x; u)=\frac{1}{1-ux}$. Therefore, each of the factors on the right-hand side of \eqref{eqA}  can be interpreted as the evaluation of the  rank $1$ zeta average $Z_{A_1}(\xx^{2\a}; u^2)$. The proof of this theorem is given in Section~\ref{sec_roots}.

We expect that Theorem~\ref{thmA} holds for \emph{affine} root systems as well, and we verified it in the case of root systems of type $D_4^{\sss (1)}$, and $A_r^{\sss (1)}$ of small rank.
A new feature in the affine case is that the two sides of the identity in Theorem~\ref{thmA}
are expected to agree up to a factor involving imaginary roots, whose determination is
problematic in general. This is a well-known feature of the affine case, already present in Macdonald's
affine generalization of the Weyl denominator formula~\cite{M1}.

\subsection{}
Let  $V_0\subset V$ be the span of $\Phi_0$ inside $V$. It is natural to consider an element
$\bs\in V_\CC:=V\otimes_\RR \CC$ as a  complex multivariable $\bs=(s_1,\dots,s_r)$, the components being the coordinates of $\bs$ with respect to the  basis $\Pi(\Phi)$.
We regard $\mathcal{Z}_\Phi(\bs)$ as a meromorphic function on $V_\CC$ in this fashion. For $\a=\sum n_i\a_i\in \Phi$, we denote $\bs_\a=\sum n_i s_i$.

The same convention can be adopted for $\mathcal{Z}_{\Phi_0}$, using the basis $\Pi_{\Phi_0}$ of $\Phi_0$ induced by the fixed basis $\Pi(\Phi)$.  We will use the inclusion $\Pi_{\Phi_0}\subset \Phi^+$ to express $\mathcal{Z}_{\Phi_0}$ using the multivariable $\bs$.
An immediate consequence of Theorem \ref{thmA} and the local-to-global principle \eqref{eq-local-to-global} is the following theorem, evaluating the corresponding residue of $\mathcal{Z}_\Phi(\bs)$.

\begin{statement}\label{thmB} Let $\mathbb{K}=\mathbb{F}_q(t)$, $q\equiv 1\!\!\!\mod 4$. Let $\Phi$ be an irreducible root system not of type $G_2$, and let $\a_i$ be a short simple root. Then
\begin{equation}\label{eqB}
\lim_{s_i\rightarrow 1\slash 2} (1-q^{1/2-s_i}) \mathcal{Z}_\Phi(\bs)=
\mathcal{Z}_{\Phi_0}(\bs)\vert_{s_{i}=1\slash 2} \cdot \prod_{{\substack{\a\in\Phi^+ \\ \la\a,\a_i \ra=1}}}  \frac{1}{ (1-q^{1-2{\bs_\a}}){\vert_{s_{i}=1\slash 2}}}.
\end{equation}
\end{statement}
Again, we remark that each of the factors on the right-hand side of \eqref{eqB}  can be interpreted in terms of an evaluation of the rank $1$ WMDS $\displaystyle \mathcal{Z}_{A_1}(s)=\zeta_{\mathbb{A}^1_{\F_q}}(s+1/2)=\frac{1}{1-q^{1/2-s}}$.

An analogue of Theorem~\ref{thmB} holds over number fields as well. To avoid 
technicalities, we illustrate it over the Gaussian field.

\begin{statement}\label{thmC} Let $\mathbb{K}=\mathbb{Q}(\sqrt{-1})$.
Let $\Phi$ be an irreducible root system not of type $G_2$, and let $\a_i$ be a short simple root. Then,
\begin{equation}\label{eqC}
\lim_{s_i\rightarrow 1\slash 2} (s_i-1\slash 2) \mathcal{Z}_{\Phi}(\bs)=\frac{\pi}{8}
\mathcal{Z}_{\Phi_0}(\bs)\vert_{s_{i}=1\slash 2} \cdot 
{ \displaystyle\prod_{{\substack{\a\in\Phi^+ \\ \la\a,\a_i \ra=1}}} \zeta_{\mathbb{K}}^{(2)}({2{\bs_\a}}){\vert_{s_{i}=1\slash 2}}},
\end{equation}
where $\zeta_{\mathbb{K}}^{(2)}(s)$ is the Dedekind zeta function of $\mathbb{K}$ with
the Euler factor at the prime dividing 2 removed.
\end{statement}
To put this result in context, note that when $\Phi$ is of type $A_1$,
then $Z_\Phi(s)=\zeta_{\mathbb{K}}^{(2)}(s+1/2)$, while $\Phi_0$ and the product on the right-hand side are trivial. The
theorem in this case reduces to the classical formula for the residue of the Dedekind zeta function at $s=1$,
which explains the constant $\pi\slash 8$.

The proof of this theorem is given in Appendix~\ref{appC}, and it makes essential use of Theorem~\ref{thmA}.
We do not strive for a more general result over number fields, as it is not the main focus of this paper. Rather,
our goal is to illustrate in a concrete case the general phenomenon that results over function fields have
number field counterparts.

The extent and precise formulation of this phenomenon, relating the residues of  higher order WMDS to similar objects associated to smaller rank root systems, is not yet clear. The only other examples known at this time relate the residues of  the \emph{cubic} WMDS of type $A_3$ over  $\F_q(t)$, $q\equiv 1\!\!\!\mod 4$, and over number fields \cites{BB, C}, to the Friedberg-Hoffstein-Lieman  cubic double Dirichlet series \cite{FHL}.   There are indications that, for quadratic WMDS associated to {affine} root systems, this phenomenon is still present. For $\mathbb{K}=\mathbb{F}_q(t)$, $q\equiv 1\!\!\!\mod 4$, we verified this for affine root systems of type $D_4^{\sss (1)}$, and $A_r^{\sss (1)}$ of small rank, and we are in the process of extending this result to all simply-laced affine root systems \cite{DIPP}.

\subsection{}\label{sec: thmD}
 For $i$ a node in the Dynkin diagram of $\Phi$, we denote by $\Phi^i\subset \Phi$ the maximal parabolic root subsystem obtained by excluding the node $i$, and we denote by $W^i\subset W$ the corresponding maximal parabolic subgroup.
For $\a\in\Phi$, let  $n_i(\a)\in \Z$ denote the coefficient of $\a_i$ in the expansion of $\a$ in the basis $\Pi(\Phi)$.

For each of the nodes $i$  specified in Table \ref{table1} (we use the standard labelling of nodes in the Dynkin diagram \cite{Bou}; see also \S\ref{sec_Phi0} and Appendix \ref{sec_exc}),
\renewcommand{\arraystretch}{1.2}
\begin{table}[!ht]
\vspace{0.5em}
\begin{center}
\begin{tabular}{c|cccccccc}
$\Phi$ & $A_r$        & $B_r$ &  $C_r$  & $D_r$ & $E_6$ & $E_7$ & $E_8$ & $F_4$  \\ \hline
$i$ & Any & $i=r$ &  $2i\le r$ &  \makecell{$2i\le r+1$\\ $i\in\{r-1,r\}$}  & $i\ne 4$ & $i\in\{1,2,7\}$ &$i\in\{1,8\}$ &  $i=4$
 \end{tabular}
\end{center}
\vspace{1em}
\caption{\label{table1}Admissible nodes.}
\end{table}
we give a new formula for $Z_\Phi(\xx; u)$ as  an average over $W^i$ of a rational function that is described in terms of the root system $\Phi$. More specifically,
we construct a  rooted tree $\mathcal{K}_\Phi(\a_i)$ whose vertices are positive roots, and such that the tree root is $\a_i$. This rooted tree  resembles the Kostant cascade construction \cites{J,K}. The roots that appear as vertices are  used to define an explicit rational function $K_{\Phi,\a_i}(\xx)$, which depends neither on $x_i$ nor on $u$. We refer to \S\ref{sec_cascade} for the details and we include a few examples below.

\begin{statement}\label{thmD}Let $\Phi$ be an irreducible root system not of type $G_2$, and let
$i$ one of the admissible nodes specified in Table \ref{table1}. We have,
\begin{equation}\label{eqD}
Z_{\Phi}(\xx; u)\cdot \prod_{\substack{\a\in\Phi^{+}\\ m_\a=2 \\n_i(\a)\ge 2}} (1-u^2 \xx^{2\a})=\frac{\displaystyle\sum_{w\in W^i}\left.\frac{1}{1-ux_i}K_{\Phi,\a_i}(\xx)\right|w }{\DD_{\Phi^i}(\xx)}.
\end{equation}
\end{statement}

We include here some examples of the rational function $K_{\Phi,\a_i}$. 
\begin{itemize}
\item If $\Phi$ is of type $A_r$, using the symmetry of the Dynkin diagram,
we can assume  that $2i\le r+1$. Then,
 \[K_{\Phi,\a_i}(\xx)=\prod_{j=1}^{i-1}\frac{1}{1-x_{i-j}x_{i+j}}.
 \]

\item If $\Phi$ is of type $C_r$ and $2i\le r$, the function $K_{\Phi,\a_i}$
is given by the same formula as for $A_r$.

\item If $\Phi$ is of type $B_r$, or $F_4$, and $i=r$ or, respectively, $i=4$, then $K_{\Phi,\a_i}=1$.

\item If $\Phi$ is of type $D_r$ and $r=2i$,  then
\[K_{\Phi,\a_i}(\xx)=\frac{1}{1-x_1x_{r-1}} \frac{1}{1-x_1x_{r}}\frac{1}{1-x_{r-1}x_{r}}
\prod_{j=1}^{i-2}\frac{1}{1-x_{i-j}x_{i+j}}
\frac{1}{1-x_{i+j}^2\cdot \ldots\cdot x_{r-2}^2 x_{r-1} x_r }.
\]
\end{itemize}
In principle one can obtain formulas for $Z_\Phi(\xx;u)$ as averages over a smaller subgroup,
by rewriting the definition using a system of coset representatives for the smaller subgroup.
However the kernel functions obtained in this way are tremendously more complicated than our kernel function.

\subsection{} It might be of some interest to comment on the origin of the results described above. We discovered formulas of type \eqref{eqD} as part of our investigations of $Z_\Phi(\xx;u)$ for affine root systems. For an affine root system, we adopt the notation set up in \S\ref{sec: thmD}  for its parabolic sub-systems and parabolic subgroups.
When~$\Phi$ is an affine root system of type $D_4^{\sss (1)}$,
three of the authors discovered in~\cite{DPP} that $Z_\Phi(\xx; u)$ and $Z_\Phi(\xx; u\xx^\dd)$ are related by a new type of functional equation that involves a  $3\times 3$ matrix $B(\xx;u)$;  here $\dd=\a_1+\a_2+\a_4+\a_4+2\a_5$
is the minimal positive imaginary root and the Dynkin diagram $D_4^{\sss (1)}$  is labelled as in Figure~\ref{affine-D4}.
\begin{figure}[!ht]
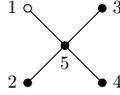

$
\dynkin [labels={1,2,5,3,4},extended]D{4}
$
\caption{$D_4^{\sss (1)}$ diagram labelling}\label{affine-D4}
\end{figure}
The matrix $B$ has an inverse with polynomial entries, which was given explicitly in~\cite{DPP}.
It turns out that the sum $B_{o, e}(\xx;u)$
of the entries in the third column of $B(\xx; u)$ determines the entire matrix $B$ and it can be expressed as follows
\[
B_{o, e}(\xx;u)=
\frac{u \xx^\dd}{x_5}\frac{\displaystyle\sum_{w\in W^5}\left.
\left[(1-ux_5)^{-1}\prod_{1\le i<j\le 4} (1-x_i x_j)^{-1} \right] \right|w  }{\DD_{\Phi^5}(\xx)}.
\]
Under the evaluation $x_1=0$, and ignoring the term ${u \xx^\dd}/{x_5}$, the above formula leads to the formula \eqref{eqD} corresponding to  the parabolic sub-system $\Phi^1$ of type $D_4$ and the node $i=5$
\[
(1-u^2\xx^{2\theta_1})Z_{\Phi^1}(x_2, \ldots, x_5) = \frac{\displaystyle\sum_{w\in W^{1,5}}\left.
\left[(1-ux_5)^{-1}\prod_{2\le i<j\le 4} (1-x_i x_j)^{-1} \right] \right|w  }{\DD_{\Phi^{1,5}}(\xx)}.
\]
Above, $\th_1=\a_2+\a_3+\a_4+2\a_5$ is the highest root of $\Phi^1$, $\Phi^{1,5}$ is the parabolic sub-system of $\Phi^1$ obtained by excluding the node $5$, and $W^{1,5}$ is its Weyl group.
The comparison of these two formulas suggests that   $B_{o, e}(\xx;u)$  is an \qq{affinization} of the finite zeta average of type $D_4$. We discovered  a similar phenomenon for
affine groups of type $A_r^{\sss (1)}$ with $r$ odd. Based on the treatment of the case $D_4^{\sss (1)}$,
Theorem~\ref{thmD} will play a role in deriving the extra functional equation for affine root systems in the ongoing work~\cite{DIPP}.  These facts prompted us to investigate the existence of such formulas for zeta averages associated to finite root systems, resulting in the discovery of  Theorem \ref{thmD}. The more fundamental Theorem \ref{thmA} was obtained in the process of proving  Theorem \ref{thmD}.

\subsection{}\label{sec1.9} To highlight  the main difficulty encountered in the proofs of Theorem \ref{thmA} and Theorem \ref{thmD}, let us point out  that, implicitly, both statements claim the existence of unexpected symmetries (in the form of extra functional equations) for certain objects. The residue in Theorem \ref{thmA} must satisfy functional equations that correspond to the simple roots in $\Phi_0$ that are not simple roots in $\Phi$, and the average over the parabolic subgroup $W^i$ in Theorem \ref{thmD} must satisfy the functional equation that corresponds to the excluded simple root. The existence of the extra symmetries, together with uniqueness results concerning rational functions with prescribed symmetries, are the main elements of both proofs.

For Theorem \ref{thmA}, the extra functional equation is proved by a detailed analysis of the Chinta-Gunnells action (Proposition \ref{prop_ssb}). We use the uniqueness result of \cite{CFG}*{Corollary 5.8}, \cite{Fr}*{Corollary 5.2}, describing $Z_\Phi(\xx; u)$ as the unique rational function invariant under $W$ with the property that $D_{\Phi}(\xx; u) Z_\Phi(\xx; u)$ is a polynomial with constant term $1$, where $D_{\Phi}(\xx; u)=\prod_{\a \in \Phi^+} (1-u^2\xx^{m_\a \a})$. However, in order to apply it, we must first show that the residue lies in the correct ambient space. This is accomplished by  revisiting,  in Section \ref{sec_numzeta}, the analysis from \cites{CFG, Fr} on the support of the numerator of a rational function invariant under the Chinta-Gunnells action.

A key ingredient in the proof of Theorem \ref{thmD} is the following uniqueness result. A rational function is uniquely determined by the residue at $x_i=1\slash u$, the invariance under the Chinta-Gunnells action of  $W^i$, and some properties of its polar divisor and the degree in $x_i$; we refer to Lemma \ref{lem_invar} for the precise conditions. The unique rational function whose residue is the one specified by Theorem \ref{thmA} is precisely  $Z_\Phi(\xx; u)$; it is remarkable that it is this precise specification of the residue that corresponds to a rational function with a larger group of functional equations. This characterization of $Z_\Phi(\xx; u)$ is different from the characterization of \cites{CFG,Fr} mentioned above.

We treat simply-laced and double-laced root systems on equal footing. However, we could have taken advantage of the following relationship between double-laced and simply-laced root systems. We denote by $\Phi^s\subseteq \Phi$ the root sub-system consisting of all short roots, and by $Q^s$ its root lattice. The inclusion of root lattices $Q^s\subset Q$ induces canonical morphisms $\F(Q^s)\subset \F(Q)$. We regard $Z_{\Phi^s}(\xx; u)$ as an element of $\F(Q)$ in this fashion. As it turns out  (see Proposition \ref{prop_double}), for $\Phi$ a double-laced root system, we have
\begin{equation}
Z_{\Phi}(\xx; u)=Z_{\Phi^s}(\xx; u).
\end{equation}
This is perhaps of independent interest.

\subsection{}\label{sec: potential}
Our results open a number of immediate questions. One set of questions is related to describing the residues, as well as formulas of type \eqref{eqD}, for the \emph{twisted}  quadratic Weyl group multiple Dirichlet series, constructed using the twisted Chinta-Gunnells action introduced in~\cite{CG1}. Such results would have implications for the description of the residues of  Eisenstein series on metaplectic $2$-covers of $p$-adic groups. A further set of questions is related to the extension of our results to the case of
higher order WMDS. Preliminary computations show that simple-minded generalizations of our Theorem \ref{thmA} and Theorem \ref{thmD} are not true.

In the case of affine root systems,  the corresponding version of Theorem \ref{thmA} plays an important technical role in the determination of the correction factor that must appear in the affine version of the local-to-global principle. It would be interesting to see if formulas  of type \eqref{eqD} hold in the affine case for an appropriate kernel function. If so, they would express the zeta average as a sum over a finite Weyl group, making the study of $Z_\Phi(\xx; u)$ more amenable.

\noindent\textbf{Acknowledgements.} Diaconu, Pa\c sol and Popa were partially supported by the CNCS-UEFISCDI grant PN-III-P4-ID-PCE-2020-2498. Ion was partially supported by the Simons Foundation grant 420882.

\section{The Chinta-Gunnells action}\label{sec_CG}

\subsection{}\label{sec2.1}
Let $\Phi$ be a finite, irreducible, reduced root system of rank $r$. We fix a basis $\Pi(\Phi)=\{\a_i\}_{1\le i\le r}$ and use $\Phi^\pm$ to refer to the corresponding sets of positive, and respectively negative, roots. The root sub-systems of short, respectively long, roots are denoted by  $\Phi^s$ and, respectively, $\Phi^\ell$. If $\Phi$ is simply-laced, we consider all roots to be \emph{short}. We extend this notation and convention to any subset of $\Phi$.

Let $Q=\bigoplus_{i=1}^r \Z \a_i$ be the root lattice of $\Phi$.  We denote by $W$ the Weyl group of $\Phi$.
For $\a\in\Phi$, let $\ss_\a$ denote the corresponding reflection. For simplicity,
we use $\ss_i$, $1\le i\le r$, to refer to the reflections corresponding to simple roots.

There is a unique  $W$-invariant inner product on  $V=Q\otimes_\Z\RR $ normalized such that the short roots have square length $2$. We use $\la\cdot ,\cdot \ra$ to denote this scalar product and  ${q}(\lam)=\frac 12 \la \lam,\lam \ra$ to refer to the associated quadratic form, which takes integral values with
this normalization. To each root $\a$ we associate the positive integer
\[
m_{\a}=\frac{2}{\gcd(2,q(\a) )} = \begin{cases} 2 & \text{ if  $q(\a)$ is odd} \\
1 & \text{ if  $q(\a)$ is even}.
        \end{cases}
\]
If $\Phi$ is simply-laced or of type $G_2$ we have $m_{\a}=2$ for all $\a\in \Phi$. If $\Phi$ is double-laced ($B_r$, $C_r$ and $F_4$), then $m_{\a}$ is $2$ if $\a$ is short, and $m_{\a}$ is $1$ if $\a$ is long. For simplicity, we denote $m_i=m_{\a_i}$.
We also consider the even sub-lattice of $Q$ defined as
$$Q_{\rm ev}=\{\lam \in Q : \la \lam, \a_i \ra \equiv 0\!\!\!\! \mod 2, \text{ for } 1\le i\le r \}.
$$
Note that $m_{\a}\a\in Q_{\rm ev}$ for $\a\in\Phi$, since the Cartan ratios ${2\la\a, \b \ra}/{\la\b, \b \ra }$  are integral for $\a,\b\in\Phi$.


\subsection{}\label{sec2.2}

For each $w$ in $W$ let $\ell(w)$ be the length of a reduced (i.e.
shortest) decomposition of $w$ in terms of simple reflections.
For $w$ in $W$ we have $\ell(w)=|\Phi(w)| $, where $\Phi(w):=\{\aa\in \Phi^+: w\aa\in\Phi^- \}$. There is a unique element of $W$ of maximal length, denoted by $w_\circ$. In this case, $\Phi(w_\circ)=\Phi^+$.

 If
$w=\ss_{i_\ell}\cdots \ss_{i_1}$ is a reduced decomposition, then
\be \label{e_phiw}
\Phi(w)=\{\a_{i_1}\prec\ss_{i_1}(\a_{i_2})\prec  \ldots  \prec
\ss_{i_1}\ss_{i_{2}}\cdots \ss_{i_{\ell-1}} (\a_{i_\ell})  \},
\ee
with the order $\prec$ dependent on the chosen reduced expression for $w$. We will also need the following well-known property of the set $\Phi(w)$ \cite{M3}*{(2.2.4)}.
\begin{lemma}\label{lem_phiw}
Assume  that $w,w^\prime\in W$ and $w^{-1}\a\in \Phi^+$ for all $\a\in \Phi(w^\prime)$.
Then, $\ell(w^\prime w)=\ell(w^\prime)+\ell(w)$ and
\[\Phi(w^\prime w)=\Phi(w)\cup w^{-1} \Phi(w^\prime).
\]
Reduced expressions of $w$, $w^\prime$, concatenate to a reduced expression of $w^\prime w$. Moreover, the order $\prec$ on $\Phi(w^\prime w)$  is
the concatenation of the order relations on $\Phi(w)$ and  $w^{-1} \Phi(w')$.
\end{lemma}
\noindent We also use the following notation
$$\Phi^s(w)=\{\b\in \Phi(w)\mid m_\b=2 \}\quad \text{and}\quad \ell_s(w)=| \Phi^s(w)|.$$
The order relation on $\Phi(w)$ induced by a reduced expression of $w$ restricts to an order relation on $\Phi^s(w)$. If $\Phi$ is simply-laced or
of type $G_2$, then $\Phi^s(w)=\Phi(w)$; when $\Phi$ is double-laced, then $\Phi^s(w)=\Phi(w)\cap \Phi^s$.


\subsection{}
Let $\F=\Q(u)$, with $u$ a formal parameter. The standard basis of $\F[Q]$,
the $\F$-group ring of $Q$, is denoted by $\{\xx^\lambda\}_{\lambda\in Q}$.
We regard  $\F[Q]$ as the ring of Laurent polynomials in the monomials $\xx^\lam$,
and we denote by $\F(Q)$ its field of fractions. Denote $x_i=\xx^{\a_i}$ and treat $\xx=(x_1,\dots,x_r)$ as a multivariable. For $\lambda\in Q$, let  $n_i(\lambda)\in \Z$ denote the coefficient of $\a_i$ in the expansion of $\lambda$ in the basis $\Pi(\Phi)$.  With this notation, if  $\lam \in Q$, then
$$\xx^\lam =\prod x_i^{n_i(\lambda)}.$$
The canonical Weyl group left action on $Q$ induces the following action of $W$ on $\F[Q]$ and $\F(Q)$
$$w \xx^\lam=\xx^{w^{-1} \lam},\quad w\in W, ~\lam\in Q.$$
The corresponding action on the multivariable $\xx$ is $(w\xx)_i =w x_i=w \xx^{\a_i}=\xx^{w^{-1}\a_i}$.

\subsection{}
Let $1\le i\le r$. The involution $\e_i: \F(Q)\to \F(Q)$ is defined  by
$$\e_i \xx^\lam=(-1)^{\la\lambda,\a_i\ra}\xx^{\lam},\quad \lam\in Q.$$
On the multivariable $\xx$ it acts by
$(\e_i \xx)_j:=\e_i x_j=\e_i \xx^\a_j=(-1)^{\la \a_j, \a_i\ra}x_j$.
If $m_i=1$ then $\a_i\in Q_{\rm ev}$ and $\e_i \xx= \xx$, so only the sign operators $\e_i$ with $m_i=2$ are non-trivial.

For $\mu\in Q$, we denote $\e^{\mu}=\prod \e_i^{n_i(\mu)}$. We have,
\be\label{e_sign}
\e^\mu \xx^\lam=(-1)^{\la \lam,\mu\ra } \xx^\lam, \quad\text{and} \quad
\e^\mu w \xx^\lam  =w \e^{w^{-1} \mu} \xx^\lam, \quad \text{for all }~ w\in W, \mu,\lambda\in Q.
\ee

\subsection{}\label{secpm}
For $f(\xx)=f(\xx;u)\in \F(Q)$, denote by $f^+_i$, $f^-_i$  its even and odd parts with respect to $\e_i$, namely
\[f^\pm_i(\xx)=\frac 12 (f(\xx)\pm f(\e_i\xx)).\]
We routinely omit the variable $u$ from the notation $f(\xx; u)$,
as it is fixed throughout.

\noindent Chinta and Gunnells~\cites{CG,CG1} define a \emph{right} action of $W$ on $\F(Q)$, which for simple reflections is described by\footnote{
This is the same action as the one defined in~\cite{CGP}*{eq. (7)},
with $n=2$ and  $v=u^2$.}
\[
f\CG \ss_i(\xx)=\begin{cases}\displaystyle \frac{1-u\slash x_i}{1-ux_i} f^+_i(\ss_i \xx)+\frac{1}{x_i} f^-_i(\ss_i \xx) &
\text{if } m_i=2 \\
         f(\ss_i \xx)      & \text{if } m_i=1.
               \end{cases}
\]

\noindent We also need the related action, denoted by $f\vert  w$,  which for simple reflections is described by
\[
f| \ss_i(\xx):=-x_i^{m_i} f\CG\ss_i(\xx).
\]

\subsection{}\label{secZCG}
Let
\begin{equation}\label{e_dphi}
\DD_{\Phi}(\xx)=\prod_{\a \in \Phi^+} (1-\xx^{m_\a \a})\quad \text{and} \quad  D_{\Phi}(\xx; u)=\prod_{\a \in \Phi^+} (1-u^2\xx^{m_\a \a}).
\end{equation}
We have $\DD_{\Phi}(\ss_i\xx)=-1/x_i^{m_i} \DD_{\Phi}(\xx)$, so the average
\be\label{zeta-avg}
Z_{\Phi}(\xx; u):=\frac{\sum_{w\in W} 1|w }{\DD_{\Phi}(\xx)}
\ee
satisfies $Z_{\Phi}=Z_{\Phi}\CG w$ for all $w\in W$.
It was shown in~\cite{CFG} that
\be\label{e_zphi}
Z_{\Phi}(\xx; u)=\frac{N_\Phi(\xx; u)}{D_{\Phi}(\xx; u)},
\ee
with $N_\Phi(\xx; u)$ a polynomial in $\xx$ with coefficients depending on $u$.
For example, if $\Phi$ is of type $A_1$ we have
\[Z_{A_1}(x)=\frac{1}{1-ux}\quad \text{and} \quad N_{A_1}(x)=1+ux.
\]

\begin{remark}\label{rem_zdecomp} The definition of $Z_\Phi(\xx; u)$ makes sense also when $\Phi$ is not irreducible. In fact, if $\Phi=\Phi'\oplus\Phi''$ is a direct sum
of irreducible root systems, then
\[Z_\Phi(\xx)=Z_{\Phi'}(\xx') \cdot Z_{\Phi''}(\xx''),
\]
with $\xx'$, respectively $\xx''$ , the multivariables that correspond to the simple roots in $\Phi'$, respectively $\Phi''$. We generally restrict ourselves to considering irreducible root systems $\Phi$, but
we will encounter reducible sub-systems as well.
 \end{remark}

\subsection{} The Chinta-Gunnells action can be also described as
\[f\CG \ss_i (\xx) = \begin{cases}f(\ss_i \xx) J(x_i, 0)+f(\ss_i \e_i \xx) J(x_i,1)
& \text{ if $m_i=2$}\\
f(\ss_i \xx)  & \text{ if $m_i=1$}
                       \end{cases}
\]
where
\[
J(x,\dd)=\frac{1}{2} \left( \frac{1-u/x}{1-ux} + \frac{(-1)^\dd}{x}\right), \quad \dd\in \{0,1\}.
\]
The action of a general element of $W$ can be expressed as follows. Recall that
$\ell_s(w)=|\Phi^s(w)|$, with $\Phi^s(w)$ defined in \S\ref{sec2.2}.
\begin{lemma}\label{lem_CG}
Let $w\in W$ and let $\ell_s=\ell_s(w)$.
For $\ddd=(\dd_\g)_{\g\in \Phi^s(w)}\in\{0,1\}^{\ell_s}$,
define $\e_\ddd=\e^{\sum_{\g\in\Phi^s(w) } \dd_\g \g}$.
We fix a reduced decomposition of $w$ and the corresponding order relation $\prec$ on $\Phi^s(w)$.
Then,
with the usual conventions on empty sums and products, we have
\be \label{e4}
f \CG w (\xx)=\sum_{\ddd\in \{0,1\}^{\ell_s}}
f\left(w \e_\ddd \xx \right)
\prod_{\b\in \Phi^s(w)}  J\left(
(-1)^{\la\b, \sum_{\g\prec\b}\dd_\g\g \ra}
\xx^{\b},\dd_\b \right).
\ee
\end{lemma}
\begin{proof}
For the purposes of this proof, let $J_2(x,\dd)=J(x,\dd)$ and $J_1(x,\dd)=\frac 12(1+(-1)^\dd)$.
Then, the Chinta-Gunnells action can be written uniformly as
\[f\CG \ss_i (\xx) = f(\ss_i \xx) J_{m_i}(x_i, 0)+f(\ss_i \e_i \xx) J_{m_i}(x_i,1).
\]
Using this formula, we prove by induction on $\ell(w)$ that
\be\label{e41}
f \CG w (\xx)=\sum_{\ddd\in \{0,1\}^{\ell}}
f\left(w \e_\ddd \xx \right)
\prod_{\b\in \Phi(w)}  J_{m_\b}\left(
(-1)^{\la\b, \sum_{\g\prec\b}\dd_\g\g \ra}
\xx^{\b},\dd_\b \right),
\ee
where $\ddd$ runs over $\ell$-tuplets $(\dd_\g)_{\g\in \Phi(w)}\in\{0,1\}^{\ell}$
 and $\e_\ddd=\e^{\sum_{\g\in\Phi(w) } \dd_\g \g}$.
If $\ell(w)=1$ then~\eqref{e41}
is clear, and if it holds for $w$, then, using Lemma~\ref{lem_phiw}, it is easy to check
that it holds for $w\ss_i$ if $\ell(w\ss_i)=\ell(w)+1$.

In \eqref{e41} it is clear that if $m_\b=1$, only the $\ddd$ with $\ddd_\b=0$ contribute in the product,
with a factor $J_1(x, 0)=1$. Formula~\eqref{e4} follows immediately.
\end{proof}

\subsection{}
As an immediate application, we have the following divisibility property of the numerator $N_\Phi(\xx; u)$
defined by~\eqref{e_zphi}.

\begin{lemma}\label{lem_div}
(i) If $m_i=2$, then $N_\Phi(\xx; u)$ is divisible by $1+ux_i$.

(ii) If $m_\a=1$, $\a\in \Phi^+$, then $N_\Phi(\xx; u)$ is divisible by $1-u^2\xx^\a$.
\end{lemma}
\begin{proof} Assume $m_i=2$. The function $1|w$ has a pole at $x_i=\pm 1\slash u$ if and only
$1\CG w$ has such a pole. By \eqref{e4}, this happens
 only if $\a_i\in \Phi(w)$, which means that $w$ has a reduced expression ending in $\ss_i$. Fixing such a reduced expression for $w$, the term in~\eqref{e4} corresponding to $\b=\a_i$ is $J(x_i, \dd_\b)$, which has no pole
at $x_i=-1\slash u$. Therefore $Z_\Phi(\xx; u)$ has no pole there, and the divisibility of $N_\Phi(\xx; u)$ follows.

Assume that $\a\in \Phi^+$ and $m_\a=1$. By~\eqref{e4}, the poles of $1|w$ are of the type
$\xx^\b=\pm 1\slash u$, for $\b\in \Phi(w)$ with $m_\b=2$. Therefore, $Z_\Phi(\xx; u)$ does not have a pole when $\xx^\a=1\slash u^2$.
Taking into account that $1-u^2 \xx^\a$ is a factor of $D_{\Phi}(\xx; u)$,
it follows that $N_\Phi(\xx; u)$ is divisible by this factor.
\end{proof}
\begin{convention}\label{conv}
We refer to $N_\Phi(\xx; u)$ as the numerator of $Z_\Phi(\xx; u)$, although it is not the numerator of the reduced
fraction expressing $Z_\Phi(\xx; u)$.
\end{convention}
The reason for considering $N_\Phi(\xx; u)$, rather than its quotient by its divisors in
Lemma~\ref{lem_div}, will become apparent in the next section.


\section{The numerator of the twisted Chinta-Gunnells average }\label{sec_numzeta}

\subsection{}
 While our main interest is the zeta average $Z_\Phi(\xx; u)$, in \S\ref{sec_proof} we shall later encounter  functions invariant under a \emph{twisted} Chinta-Gunnells action.
Such functions were studied in detail in \cites{CFG,Fr}, and since we need to consider a slightly more general setting we review here the results from loc. cit. that are needed.

\subsection{}

Let $\om_i$, $1\le i\le r$, denote the fundamental weights, and $\a_i^\vee={2\a_i}/{\la\a_i,\a_i \ra}$, $1\le i\le r$, the simple co-roots, so that $\la\om_i, \a_j^\vee\ra=\dd_{i,j}$. The weight lattice of $\Phi$ is $P=\bigoplus_{i=1}^r \Z \om_i$.  Let
$P^{+}=\sum_{i=1}^r \Z_{\ge 0}\om_i$, and $P^{++}=\sum_{i=1}^r \Z_{> 0}\om_i$, be the set of dominant weights, and respectively regular dominant weights.  With the notation
\be
\rho=\sum_{i=1}^r \om_i=\frac 12 \sum_{\a\in\Phi^+} \a,
\ee
we have $P^{++}=\rho+P^+$. On $P$ we have the natural partial order relation defined by
\be
\om'\le \om\quad  \text{if and only if}\quad  \om-\om'\in Q^+:=\sum_{i=1}^r\Z_{\ge 0}\a_i.
\ee

\subsection{}
We fix  $\ell_1, \ell_2, \ldots, \ell_r\in \Z$  (which will be called \emph{twist parameters}), and denote
\[\om=\sum_{i=1}^r \ell_i \om_i\in P,\quad\text{and}\quad  \theta=\rho+\om.
\]
Following~\cite{CFG}, we define a \emph{twisted action}\footnote{
After  the change of variables
$u \mapsto \sqrt q , x_i\mapsto x_i\slash \sqrt q$, the action $\CG_{\!\!\om}$ coincides with $|_\ell$ defined in~\cite{CFG}
with $\ell=(\ell_1, \ldots, \ell_r)$. }
by
\[
f \CG_{\!\!\om} \ss_i(\xx)=x_i^{\ell_i}\cdot
\begin{cases}\displaystyle\frac{1-u\slash x_i}{1-ux_i} f^+_i(\ss_i \xx)+\frac{1}{x_i} f^-_i(\ss_i \xx)
        & \text{\ \ if $m_i=2$, $\ell_i$ even,} \\
        \displaystyle \frac{1-u\slash x_i}{1-ux_i} f^-_i(\ss_i \xx)+\frac{1}{x_i} f^+_i(\ss_i \xx)
         & \text{\ \ if $m_i=2$, $\ell_i$ odd,}\\
         f(\ss_i \xx)  &  \text{\ \ if $m_i=1$}.
               \end{cases}
\]
The usual Chinta-Gunnells action corresponds to $\omega=0$.

\noindent We also consider the related action, denoted by $f\vert_{\om}  w$,  which for simple reflections is defined by
\[
f|_{\om} \ss_i(\xx):=-x_i^{m_i} f\CG_{\!\!\om}\ss_i(\xx).
\]

As in \eqref{zeta-avg}, we define the corresponding twisted zeta average. For $\om\in P^+$, it was shown in~\cite{CFG} that the denominator of the twisted zeta average divides $D_{\Phi}(\xx; u)$.
For $\om\in P$, the twisted zeta average may have additional poles when some $x_j=0$.
Therefore,  for any $\om\in P$, we define the $\F$-vector space
\be\label{N-space}
\NN_\om=\{N\in \F[Q] ~:~  \ f=N/D_{\Phi}(\xx; u) \text{ satisfies  } f=f\CG_{\!\!\om} w.
\text{ for all } w\in W\},
\ee
We emphasize that, unlike in \cite{CFG},  we allow numerators that are \emph{Laurent polynomials}.

\subsection{}

If $f(\xx)\in \F(Q)$ is  even with respect to all the sign operators $\e_i$ then,
for any $g(\xx)\in \F(Q)$ and $w\in W$, we have
$$[f(\xx)g(\xx)]|_\om w= f(w\xx)\cdot  g(\xx)|_\om w.$$
It follows that the characterizing property of $N\in \NN_\om$ is
\be\label{e_N}
N = \frac{1-u^2 x_i^{m_i}}{1-u^2\slash x_i^{m_i}} \cdot N\CG_{\!\!\om} \ss_i, \quad i=1,\ldots, r.
\ee
The support of a Laurent polynomial $N=\sum_{\lam \in Q} c_\lam \xx^\lam $ is $\supp (N)=\{\lam\in Q \mid c_\lam \ne 0\}.$ The equality~\eqref{e_N} is equivalent to the following linear system satisfied by the coefficients
\be\label{e_rec}
\begin{aligned}
    c_\mu-u c_\lam&=c_{\lam-\a_i}-u c_{\mu+\a_i}, & &\text{if } m_i=2 \text{ and } l_i+ \la\lam,\a_i^\vee \ra \text{ even,}\\
    c_\mu+u^2 c_\lam&=c_{\lam-2\a_i}+u^2 c_{\mu+2\a_i}, & &\text{if } m_i=2 \text{ and }  l_i+\la\lam,\a_i^\vee \ra \text{ odd,}\\
    c_\mu+u^2 c_\lam&=c_{\lam-\a_i}+u^2 c_{\mu+ \a_i}, & &\text{if } m_i=1,
    \end{aligned}
\ee
where $\mu=\ss_i (\lam) + \la\theta,\a_i^\vee \ra\a_i$. Remark that if $m_i=2$, then $q(\a_i)$ is odd, therefore
$\la\lam,\a_i \ra$ and $ \la\lam,\a_i^\vee \ra$ have the same parity.

\subsection{}
The order relation $\le$ on $P$, when restricted to a single $W$-orbit, has an equivalent description in terms of the Bruhat order on $W$. For an element $w\in W$ and $1\le i\le r$, we write $w<\ss_iw$\  if and only if $\ell(w)<\ell(\ss_iw)$. The transitive closure of this relation is called the (weak left) Bruhat order. For the basic properties of the Bruhat order we refer to \cite{Hum}*{\S5.9}. From Lemma \ref{lem_phiw} it follows that $w<\ss_i w$ if and only if $w^{-1}\a_i \in \Phi^+$.
For each weight $\eta\in P$ define $\eta_+$ to be the unique dominant element in $W\eta$, the orbit of $\eta$. Let $w_\eta\in W$ be the unique minimal length element such that $w_\eta(\eta_+)=\eta$.
\begin{Prop} \label{lm: order}
If $\ss_i\eta\neq \eta$, then $w_{\ss_i\eta}=\ss_i w_\eta$. Furthermore, the following are equivalent
\begin{enumerate}[label=(\alph*)]
\item $w_{\ss_i\eta}>w_\eta$;
\item  $w_\eta^{-1}\a_i \in \Phi^+$;
\item  $\la \eta,\a_i\ra>0$;
\item $\ss_i\eta< \eta$.
\end{enumerate}
In consequence, for any $\nu\in W\eta$,  we have $\nu\le \eta$ if and only if  $w_\nu\ge w_\eta$.
\end{Prop}
\begin{proof}
The first claim is \cite{Ion}*{Lemma 4.3}. The equivalence between (b) and (c) is proved in  \cite{Ion}*{Lemma 4.1}. The remaining implications are straightforward.
\end{proof}
We note that, in particular, $\eta_+$, and $w_\circ \eta_+$,
are  the  largest, and respectively smallest, elements of $W\eta$.
\subsection{}
Following~\cite{CFG},  for each dominant weight $\xi\in P^+$, we  introduce the set
\[
O_\xi:=\{ \theta-w\xi ~:~ w\in W \}.
\]
This is the usual $W$-orbit of $\xi$, reflected in the origin and translated by $\theta$.

We consider the relations~\eqref{e_rec}  for  elements
$\lam, \mu\in O_\xi$ such that
\[
\lam=\theta-\eta, \quad \mu= \theta-\ss_i \eta=\ss_i \lam +\a _i \la\theta ,\a_i^\vee \ra,
\text{ with }   \ss_i \eta \le\eta.
\]
The elements in the root lattice $Q$ appearing on the right-hand side of \eqref{e_rec}
 belong to orbits $O_\tau$ with $\tau>\xi$, by the following geometric lemma from \cite{CFG},
 whose proof we include for completeness.
\begin{lemma}\label{lem_geom}
Let $\lam=\theta-\eta $, $\mu=  \theta-\ss_i \eta$ for some
$\eta\in P$ such that   $\ss_i\eta\le\eta$. Denoting $\xi:=\eta_+$,
for any $m\ge 1$ we have
\begin{enumerate}
\item If $\mu+m\a_i\in O_\tau$, then $\tau>\xi$;
\item  If $ \lam-m\a_i\in O_\tau$, then $ \tau>\xi$.
\end{enumerate}
\end{lemma}

\begin{proof} We prove the statement for $\lam$, the argument for $\mu$ being similar.
Let $\tau\in P^+$ such that
$$\lam-m\a_i=\theta-w\tau.$$
Since $\eta_+=\xi$, we have $\lam=\theta-w_\eta\xi$, which gives
 $\xi=w_\eta^{-1}w\tau -m w_\eta^{-1} \a_i$. We now have two cases:
(i) if   $ \ss_i \eta <\eta$, then $w_\eta^{-1} \a_i>0$ by Lemma~\ref{lm: order};
(ii) if  $ \ss_i \eta =\eta$, then $\eta=\ss_i w_\eta \xi$ and the definition of $w_\eta$ implies
$\ell(\ss_i w_\eta)> \ell(w_\eta)$. Again we obtain $w_\eta^{-1} \a_i>0$, by Lemma~\ref{lem_phiw}.

In both cases it follows that $\xi<w_\eta^{-1}w\tau$. Since $\tau$ is dominant,
we have $w_\eta^{-1}w\tau\le \tau$, showing that $\tau>\xi$.
\end{proof}

\subsection{}
Using the above lemma, one can prove, along the same lines as in~\cite{CFG}, the following result. We emphasize that in our statement we do not assume that  the twisting parameters are nonnegative.
\begin{Prop}\label{prop_orb}
Let $\om\in P$, and $N\in\NN_\om$, $N\ne 0$. Let $\xi\in P^+$ be maximal
with the property that
$$O_\xi\cap \supp(N)\ne \emptyset.$$
Then $\xi$ is strongly dominant and $O_\xi\subseteq \supp(N)$.
\end{Prop}
\begin{proof}
Let $N=\sum c_\lam \xx^\lam$.  For a contradiction, assume that $\xi$ is not in  $P^{++}$, so
the parabolic subgroup $P_\xi$ generated by the simple reflections fixing it is nontrivial.
Let $\ss_i\in P_\xi$ and  $\lam=\theta-\xi$, $\mu=\theta-\ss_i\xi=\lam$. By Lemma~\ref{lem_geom},
the equations \eqref{e_rec} show that
$c_\lam$ is determined by some coefficients labeled by elements in sets $O_\tau$, with $\tau>\xi$.
Such coefficients must vanish by the maximality of $\xi$, so $c_\lam=0$.

Let now $\g\in O_\xi$, so that
$\g=\theta-\eta=\theta-w_\eta \xi$. Let $w_\eta=\ss_{i_\ell}\cdot \ldots\cdot \ss_{i_1}$ be a reduced
decomposition with $\ell=|w_\eta|$. By the minimality of $w_\eta$ and Lemma~\ref{lm: order},
it follows that
$\xi>\ss_{i_1} \xi>\ss_{i_2} \ss_{i_1} \xi>\cdots>\eta.$
Applying Lemma~\ref{lem_geom} a number of $\ell$ times shows that $c_\g=0$ for all $\g\in O_\xi$,
contradicting the hypothesis $O_\xi\cap \supp(N)\ne \emptyset.$

We conclude that $\xi\in P^{++}$, so the set $O_\xi$ contains $|W|$ elements. If $c_\lambda=0$, the argument in the previous paragraph shows that $O_\xi\cap \supp(N) = \emptyset$.
Therefore, $c_\lam\ne 0$, and the argument of the previous paragraph shows that  $c_\gamma \ne 0$ for all $\gamma\in O_\xi$. In conclusion, $O_\xi\subseteq \supp(N)$.
\end{proof}

\begin{remark}\label{rem_unique}
The proof of Proposition \ref{prop_orb} gives a slightly stronger result: every $N\in \NN_\om$ is uniquely determined by the coefficients $c_{\theta-\xi}$,  $\xi\in P^{++}$.
\end{remark}

\begin{Cor}\label{cor_notpol}
Let $0\neq N\in \NN_0$,  such that  $\supp(N)\not\subset Q^+$. Then there is $\lam<0$ such that $\lam\in \supp(N)$.
\end{Cor}
\begin{proof}
Take $\xi\in P^{++}$ maximal with the property that $O_\xi\cap \supp(N)\ne \emptyset$.
Since $\xi$ is strongly dominant, we have $\xi\ge \theta=\rho$.

We prove by contradiction that $\xi>\rho$ for at least one such maximal $\xi$.
Indeed, assume  that $\xi=\rho$ for all such maximal $\xi$.
It follows that $\rho\ge \tau$ for all
$\tau\in P^+$ such that $O_\tau\cap \supp(N)\ne \emptyset$. But then $\rho-w\tau\ge \rho-\tau\ge 0$, so
 $\supp(N)\subset Q^+$, contradicting the hypothesis.

Therefore, there exists such a maximal element $\xi\in P^{++}$
with $\xi>\rho$. Proposition~\ref{prop_orb} implies that $O_\xi\subset \supp(N)$,
and in particular,  $\lam=\rho-\xi<0$ is in $\supp(N)$.
\end{proof}

The following corollary is proved in~\cite{CFG} for $\om$ dominant.

\begin{Cor}\label{cor_supp} Let $\om \in P$, and assume $N\in \NN_\om$ has $\supp(N)\subset Q^+$.
Then,
\[
\supp(N)\subseteq \bigcup_{\theta\ge \xi\in P^+ } O_\xi.
\]
In particular, $\supp(N)$  is contained in the convex hull of the set $\{ \theta-w\theta\mid w\in W \}$.
\end{Cor}
\begin{proof}
By Proposition \ref{prop_orb}, for every maximal $\xi\in P^{++}$ with the property that $O_\xi\cap \supp(N)\ne \emptyset$, we have $O_\xi\subset \supp(N)$.
In consequence, $\theta-\xi\ge 0$. Since every $\tau\in P^+$ with $O_\tau\cap \supp(N)\ne\emptyset$
is smaller than such a maximal $\xi$, it follows that
 $\theta\ge \tau $ for all such $\tau$. This is  our first claim.

 The second claim
 follows from the first and the following fact: if $\theta\in P$, $\xi\in P^+$ and $\theta\ge \xi$,
 then $W\xi$ is contained in the convex hull of  $W\theta$~\cite{M3}*{\S2.6}.
\end{proof}

\subsection{}
Using that $\xi$ and $w_\circ\xi$ are the largest and, respectively, the smallest, elements of $W\xi$
for $\xi\in P^+$, we obtain the following information about the numerator $N_\Phi$ of the
untwisted zeta average, defined in~\eqref{e_zphi}.
\begin{Cor} \label{cor_degree}  For each $\lam\in\supp N_\Phi$ we have
$\lam\le \rho-w_\circ\rho =2\rho.
$
\end{Cor}
\begin{proof}Let $\lam\in\supp N_\Phi $. By Corollary~\ref{cor_supp} there is $\xi\le \rho$, $\xi\in P^+$, such
that
\[
\lam=\rho-w\xi\le \rho-w_\circ \xi\le \rho-w_\circ\rho.
\]
The second inequality holds because $\rho-\xi\in Q^+$, and $w_\circ$ maps $\Phi^+$ to $\Phi^-$.
\end{proof}

Finally, we will need the following result, proved in~\cite{CFG} for $\Phi$  simply-laced,
and generalized in~\cite{Fr} for arbitrary $\Phi$ (and covers of arbitrary degree).

\begin{Prop}[\cite{CFG}]\label{prop_unique} Let $Z=N/D_{\Phi}(\xx; u)$ be a rational function  invariant under the  Chinta-Gunnells action of $W$, and such that   $\supp(N)\subseteq Q^+$, and   $N(0)=1$.  Then, $Z=Z_\Phi(\xx; u)$.
\end{Prop}
\begin{proof}Since the proof is short, we include it here following~\cite{CFG}.
By Remark~\ref{rem_unique}, the coefficients $c_{\rho-\xi}$ for $\xi\in P^{++}$ uniquely determine
$N$. By Corollary~\ref{cor_supp}, the coefficients  $c_{\rho-\xi}$ are non-zero only for $\xi\le \rho$.
Since $\rho$ is the smallest strongly dominant weight, we must have $\xi=\rho$. We conclude that $c_0$ uniquely determines $N$, and therefore also uniquely determines $Z$. The zeta average $Z_\Phi(\xx; u)$
satisfies the conditions in the statement, so $Z=Z_\Phi$.
\end{proof}
\noindent Proposition \ref{prop_unique} was generalized in~\cite{Fr}.
It was shown there that, for $\om$ dominant, the dimension of the space $\NN_\om$ equals the number of
strongly dominant weights $\xi$ such that $\xi\le \theta=\om+\rho$, the maximal dimension
allowed by Remark~\ref{rem_unique}.


\section{Zeta averages of double-laced root systems}\label{sec_doublelaced}

\subsection{}
In this section we assume that the root system $\Phi$ is double-laced (i.e. of type $B_r$, $C_r$, $F_4$).
By Lemma~\ref{lem_div}, the zeta average $Z_\Phi(\xx; u)$ may have poles only at $\xx^\a=\pm1\slash u$,
with $\a\in \Phi^+$ a short root, as long roots $\a$ have $m_\a=1$. This suggests a possible role for
the (simply-laced) root sub-system $\Phi^s\subset \Phi$ consisting of short roots.
The positive short roots $(\Phi^s)^+=\Phi^s\cap \Phi^+$ induce a unique canonical basis $\Pi(\Phi^s)$ of $\Phi^s$, compatible with $(\Phi^s)^+$. The basis $\Pi(\Phi^s)$ contains $\Pi(\Phi)^s$, the short simple roots of $\Phi$. We will continue to call the elements of $\Pi(\Phi)$ simple roots, and we will refer to  the elements of $\Pi(\Phi^s)\setminus \Pi(\Phi)^s$ as $\Phi^s$-simple roots.

Let $Q^s$ denote the root lattice of $\Phi^s$. The inclusion of root lattices $Q^s\subset Q$ induces canonical morphisms $\F(Q^s)\subset \F(Q)$. The zeta average $Z_{\Phi^s}(\xx; u)$ defined using the basis $\Pi(\Phi^s)$
will be regarded  as an element of $\F(Q)$ via $\F(Q^s)\subset \F(Q)$. More explicitly,
$Z_{\Phi^s}$ is by definition
a function of variables $x_j$ for $\a_j$ a short simple root, and of new variables $x_\g$ for $\g$ a
$\Phi^s$-simple root. It is regarded as an element of $\F(Q)$, and denoted by $Z_{\Phi^s}(\xx; u)$, via the substitution $x_\g=\xx^\g$ for all $\Phi^s$-simple roots $\g$.

In this section, we show that, in fact,
$Z_\Phi(\xx; u)$ and $Z_{\Phi^s}(\xx; u)$ coincide. Although this result is not needed in the sequel,
it is a straightforward application
of the uniqueness property recalled in Proposition~\ref{prop_unique}, and we include it because it might be of independent interest.
This identity can be used to derive Theorem~\ref{thmA} for double-laced systems from
the same result for simply-laced systems,  but the proof of Theorem \ref{thmA} that we present is uniform.

\subsection{}\label{sec4.2}

The group generated by the simple reflections corresponding the elements of $\Pi(\Phi)^\ell$, the long simple roots of $\Phi$,
keeps  $(\Phi^s)^+$ stable. Therefore, this group permutes the elements of $\Pi(\Phi^s)$, and hence it can be identified with a subgroup of the group of automorphisms for the Dynkin diagram of $\Phi^s$. This subgroup coincides with the full group of diagram automorphisms of $\Phi^s$, except for the case of $\Phi=C_4$. The orbit of the unique short simple root that has a long neighbor in the diagram of $\Phi$ contains all the $\Phi^s$-simple roots; all the elements in this orbit are orthogonal. The explicit action of the simple long reflections of $W$  on the Dynkin diagrams of $\Phi^s$  is indicated in Table \ref{table2} below. We use the standard labelling of the Dynkin diagram of $\Phi$  (see  \cite{Bou} and \S\ref{sec_Phi0}). By convention, $D_2=A_1\times A_1$ and $D_3=A_3$.

\begin{table}[!ht]
\begin{center}
\begin{tabular}{l|l}
$\Phi$ & \multicolumn{1}{c}{$\Phi^s$} \\\hline
$B_r\ (r\ge 3)$ & 
\begin{tikzpicture}[baseline=-1]
    \node at (-0.7,0) {$A_1^{r}$:};
    \node[label={[label distance=-0.3cm]-90:$\scriptstyle{\a_r}$}] (1) at (0,0) {$\scriptstyle\bullet$};
    \node[label={[label distance=-0.3cm]-90:$\scriptstyle{\a_r+\a_{r-1}}$}] (2) at (1.5,0) {$\scriptstyle\bullet$};
    \node at (2.5,0) {$\cdots$};\node at (2.5,0.5) {$\cdots$};
    \node[label={[label distance=-0.3cm]-90:$\scriptstyle{\a_r+\ldots +\a_2} $}] (3) at (3.5,0) {$\scriptstyle\bullet$};
    \node[label={[label distance=-0.3cm]-90:$\scriptstyle{\a_r+\ldots +\a_1} $}] (4) at (5,0) {$\scriptstyle\bullet$};
    \draw[latex-latex](1) to[out=60,in=120] node[above]{$\scriptstyle{\ss_{r-1}}$} (2);
    \draw[latex-latex](3) to[out=60,in=120] node[above]{$\scriptstyle{\ss_{1}}$} (4);
    \node[fit=(current bounding box),inner ysep=0mm,inner xsep=0]{};
\end{tikzpicture}\\
$C_r\ (r\ge 2)$ & 
$D_r$: \ \ \dynkin[labels={\a_1,\a_2,\a_{r-2},\a_{r-1}, \a_{r-1}+\a_r},involutions={
5<right>[\ss_r]4}]{D}{**...***}\\
$F_4$ & 
$D_4$: \ \ \dynkin[labels={\a_3,\a_4,\a_2+\a_3,\a_1+\a_2+\a_3},
involutions={3<above>[\ss_2]1;4<right>[\ss_1]3}]D4\\
 \end{tabular}
\end{center}
\vspace{1em}
\caption{\label{table2} Diagram automorphisms.}
\end{table}

\begin{Prop}\label{prop_double}
Let $\Phi$ be a double-laced irreducible root system, and regard $Z_{\Phi^s}(\xx;u)\in \F(Q^s)\subset \F(Q)$.
Then,
$$Z_\Phi(\xx; u)=Z_{\Phi^s}(\xx;u).$$
More explicitly,
\[
Z_{B_r}(\xx; u)=\prod_{i=1}^r Z_{A_1} (x_i\cdot \ldots\cdot x_r; u),\quad  Z_{C_r}(\xx; u)=Z_{D_r}(x_1,\ldots, x_{r-1}, x_{r-1} x_r; u),
\]
\[
Z_{F_4}(\xx; u)=Z_{D_4}(x_3,x_4, x_2 x_3, x_{1} x_2 x_3; u).
\]
\end{Prop}
\begin{proof}By Lemma~\ref{lem_div}, we can write $Z_{\Phi^s}(\xx; u)= N(\xx; u)/D_{\Phi}(\xx; u)$
with a polynomial $N(\xx;u)$ such that $N(0,\ldots, 0)=1$.
By Proposition~\ref{prop_unique}, our claim follows once we show that
\be\label{e_Zinv}
Z_{\Phi^s}(\xx; u)=Z_{\Phi^s}(\xx; u)\CG \ss_i, \quad \text{for all}\ 1\le i\le r.
\ee
If $\a_i$ is a short root, then~\eqref{e_Zinv} follows from the invariance of $Z_{\Phi^s}$
under $\ss_i$.
If $\a_i$ is a long root, then $\ss_i$ acts on $\Phi^s$ by a diagram automorphism. Since a change of variables that corresponds to a diagram automorphism leaves the corresponding zeta average invariant, we obtain that  $$Z_{\Phi^s}(\xx; u)=Z_{\Phi^s}(\ss_i \xx; u)=Z_{\Phi^s}(\xx; u)\CG\ss_i.$$  This finishes the proof.
\end{proof}


\section{The orthogonal root system}\label{sec: roots}

\subsection{}
Throughout this section we assume that the root system $\Phi$ is not of type $G_2$. The case $\Phi=G_2$ is considered in Appendix~\ref{sec_G2}.   With this assumption, the roots $\a\in\Phi$ with $m_\a=2$
are precisely the short roots in $\Phi$. Recall our convention that for simply-laced root systems we consider $\Phi^s=\Phi$.

 We fix  $1\le i\le r$ such that $\a_i$ is \emph{short}. As stated in Theorem \ref{thmA}, the residue  of $Z_{\Phi}(\xx; u)$ at $x_i=1/u$ can be expressed in terms of the zeta average of the root sub-system
orthogonal to $\a_i$. We first describe some of the properties of such orthogonal root systems.

\subsection{} \label{sec_parabolic} For $i$ a node in the Dynkin diagram of $\Phi$, we denote by $\Phi^i\subset \Phi$ the maximal standard parabolic root sub-system obtained by excluding the node $i$, and we denote by $\Phi^{(i)}\subset \Phi$ the standard parabolic root sub-system obtained by excluding the node $i$ and its neighbors. We use $\Pi(\Phi^i)$ and $\Pi(\Phi^{(i)})$ to refer to the corresponding bases. The corresponding  parabolic subgroups of $W$ are denoted by  $W^i$ and $W^{(i)}$, respectively. We remark that  $W^{(i)}=\stab_{W^i}(\a_i)$, the parabolic subgroup of $W^i$, generated by the simple reflections that fix  $\a_i$. For $\a\in\Phi$, let  $n_i(\a)\in \Z$ denote the coefficient of $\a_i$ in the expansion of $\a$ in the basis $\Pi(\Phi)$.

\subsection{}
Consider now the orthogonal complement
\[
\Phi_0=\a_i^\perp:=\{\a\in\Phi : \la\a_i, \a\ra =0 \}.
\]
This is a root system containing $\Phi^{(i)}$. We denote by $W_0$ the Weyl group of $\Phi_0$, and by $\Pi(\Phi_0)$ the basis compatible with the subset of positive roots $\Phi^+_0:=\Phi^+\cap \Phi_0$. Of course,
$\Pi(\Phi_0)\cap \Pi(\Phi)=\Pi(\Phi^{(i)})$. Let
\[
\Pi_\new(\Phi_0)=\Pi(\Phi_0)\setminus \Pi(\Phi^{(i)}).
\]
The Dynkin diagram of $\Phi_0$ is obtained by
attaching  the elements $\b\in \Pi_\new(\Phi_0)$  to the Dynkin diagram of $\Phi^{(i)}$ according to the information provided by the inner products between $\b$ and the elements of $\Pi(\Phi^{(i)})$. Since the Weyl group acts simply transitively on the short roots of $\Phi$, the isomorphism type of $\Phi_0$ is independent of the node $i$ and is recorded in  Table \ref{table3}. The conventions that we use for the information displayed in Table \ref{table3} are the following: we assume that $r\ge 4$ for $D_r$, $r\ge 2$ for $C_r$, and $r \ge 3$ for $B_r$; for $C_2=B_2$, node $1$ corresponds to a short root. For the second line, the following conventions apply: $D_2=A_1\times A_1$; $D_3= A_3$;
$C_0=\emptyset$; $C_1=A_{1}^\ell$. The notation $A_1^s$, and respectively $A_1^\ell$, refers to a root system of type $A_1$ generated by a short, and respectively long, root of $\Phi$.

\begin{table}[ht]
\renewcommand{\arraystretch}{1.2}
\begin{center}
\begin{tabular}{c|cccccccc}
$\Phi$ & $A_r$       & $D_r$               & $E_6$ & $E_7$ & $E_8$
 & $B_r$ &  $C_r$               & $F_4$ \\ \hline
$\Phi_0$ & $A_{r-2}$ & $D_{r-2}\times A_1$ & $A_5$ & $D_6$ & $E_7$ &  $B_{r-1}$ &  $C_{r-2}\times A_{1}^s$ & $B_3$
 \end{tabular}
 \vspace{1em}
\end{center}
\caption{\label{table3}The Cartan type of the orthogonal root system}
\end{table}

\subsection{}\label{sec_Phi0} We provide an explicit description of the elements of $\Pi_\new(\Phi_0)$, and of the Dynkin diagram of $\Phi_0$, for $\Phi$ a classical root system, and $\Phi=F_4$. The same information in the case of exceptional simply-laced root systems is found in
Appendix~\ref{sec_exc} (and for $\Phi=G_2$ in Appendix~\ref{sec_G2}). For a connected subset
$\SS$ of nodes in the Dynkin diagram of $\Phi$, we denote by $\theta_\SS$ the longest root in the
corresponding root system (if $\Phi$ is simply-laced), and by $\theta^s_{\SS}$, respectively $\theta^\ell_{\SS}$ the dominant short, respectively long root (if $\Phi$ is double-laced).

\begin{remark}\label{rem_1} For the reader who prefers a more conceptual description
of the Dynkin diagram of $\Phi_0$, we summarize here the possible types of new simple roots that occur.
A short/long element $\b\in\Pi_\new(\Phi_0)$ is the dominant short/long root in a connected
subdiagram $D$ of the Dynkin diagram of $\Phi$, which is minimal (with respect to inclusion)
with the following properties
\begin{enumerate}[label=(\alph*)]
\item node $i$ belongs to $D$;
\item $\a_i$ is orthogonal to the dominant short/long root of $D$.
\end{enumerate}
The Cartan type of $D$, the index of the node $i$ in the standard labeling of $D$,
and the role played by the root~$\b$ in~$D$ are among the following:
\begin{enumerate}
 \item $A_3$ with the node~$i$ in position $2$ and $\b$ the highest root;
 \item $D_n$, $n\ge 4$, with the node $i$ in position $1$ and $\b$ the highest root;
 \item $C_n$, $n\ge 2$, with $i$ in position $1$ and $\b$ the dominant short root;
 \item $C_3$ with the node $i$ in position $2$ and $\b$ the dominant long root.
\end{enumerate}
\end{remark}
For double-laced root systems, for the definition of the kernel functions $K_{\a_i,\Phi}(\xx)$ in \S\ref{sec: kernel}, we need the set $\Pi_\new^*(\Phi_0)$ that consists of the elements of $\Pi_\new(\Phi_0)$ that arise, as explained above, from a diagram $D$ of type $A_3$. This set is empty, unless $\Phi=C_r$, $r\ge 4$, and $1<i<r-1$.

\subsubsection{$A_r,~ r\ge 1$:}\hspace{2mm} \dynkin[labels*={\a_1,\a_2,\a_r}]A{**...*}

The set $\Phi^+$ consists of $\a_j+\ldots +\a_k$ for $j\le k$. If $i=1$, or $i=r$, we have $\Phi_0=\Phi^{(i)}$. If $1<i<r$,
then  $\Pi_\new(\Phi_0)=\{\b=\a_i+\a_{i-1} +\a_{i+1}\}$,
and the Dynkin diagram of $\Phi_0$ is
\dynkin[labels*={\a_1,\a_{i-2},\b,\a_{i+2},\a_{r} }]A{*...***...*}

\subsubsection{$B_r,~ r\ge 3$:}\hspace{2mm} \dynkin[labels*={\a_1,\a_2,\a_{r-1},\a_{r}}]B{**...**}

The set $\Phi^+$ consists of  $\a_j+\ldots +\a_k$, $j\le k\le r$, and $\a_j+\ldots +\a_k+2(\a_{k+1}+\ldots+\a_r)$, $j\le k< r$.
In this case $\a_r$ is the only short root, and for $i=r$ the set
$\Pi_\new(\Phi_0)$ consists of $\b=\a_{r-1} +\a_{r}$. The
 Dynkin diagram of $\Phi_0$ is: \dynkin[labels*={\a_1,\a_2,\a_{r-2},\b }] B{**...**}.

\subsubsection{$C_r,~r\ge 2$:}\hspace{2mm} \dynkin[labels*={\a_1,\a_2,\a_{r-1},\a_r}]C{**...**}

The set $\Phi^+$ consists of  $\a_j+\ldots +\a_k$, $j\le k< r$, $2(\a_j+\ldots +\a_{r-1})+\a_r$, $j< r$, and $\a_j+\ldots +\a_k+2(\a_{k+1}+\ldots+\a_{r-1})+\a_r$, $j\le k< r$.
The dominant long and short roots are
\[\theta^\ell=2 \a_1+\ldots 2\a_{r-1}+\a_r,\quad \theta^s=\a_1+2 \a_2+\ldots +2\a_{r-1}+\a_r.\]

The description of  $\Pi_\new(\Phi_0)$ and the Dynkin diagram of $\Phi_0$ is included in Table \ref{table5}.
In small rank cases, the situation depicted in the table reduces to the following:
if $r=2$, $i=1$, then $\Phi_0$ is of type $A_{1}^s$ generated by $\b$; if
$r=3$, $i=1$, then $\Phi_0$ is of type $A_{1}^s\times A_{1}^\ell$ generated by $\b$ and $\a_3$;
and when $r=3$, $i=2$, then $\Phi_0$ is of type $A_{1}^s\times A_{1}^\ell$ generated by $\b$ and $\b'$.  

\begin{table}[!ht]
\renewcommand{\arraystretch}{1.2}
\begin{center}
\begin{tabular}{lcc}
Node $i$ & $\Pi_\new(\Phi_0)$  & $\Phi_0$ \\\hline
$i=1$ & $\b=\theta^s$ &  $\begin{dynkinDiagram}
[labels*={\a_3, \a_{r-1},\a_{r}}]
C{*...**}
\end{dynkinDiagram}
\begin{dynkinDiagram}[labels*={\b}]
A{*}
\end{dynkinDiagram}$\\\hline
$1<i< r-1$ & \makecell{ \vspace{-1em} \\ $\b=\theta_{\{i-1,i,i+1\}}$\\
$\b^\prime=\theta^s_{\{i,\ldots,r\}}$ \\ \vspace{-1em}}   & $\begin{dynkinDiagram}
[labels*={\a_1,\a_{i-2},\b, \a_{r-1},\a_r}]
C{*...**...**}
\end{dynkinDiagram}
\begin{dynkinDiagram}[labels*={\b^\prime}]
A{*}
\end{dynkinDiagram}$ \\\hline
$i=r-1$ & \makecell{\vspace{-1em} \\ $\b=\theta^s_{\{r-1,r\}}$ \\$\b'=\theta^\ell_{\{r-2,r-1,r\}}$ \\ \vspace{-1em}}   &$\begin{dynkinDiagram}
[labels*={\a_1,\a_{r-3}, \b'}]
C{*..**}
\end{dynkinDiagram}
\begin{dynkinDiagram}[labels*={\b}]
A{*}
\end{dynkinDiagram}$
\\\hline
\end{tabular}\vspace{1cm}
\end{center}\caption{\label{table5} The orthogonal root system in type $C_r$, $r\ge 2$}
\end{table}

\subsubsection{$D_r,~ r\ge 4$:}\hspace{2mm} \dynkin[labels*={\a_1,\a_2,\a_{r-2},\a_{r-1}, \a_{r}}] D{**...***}

The set $\Phi^+$ consists of any sum of distinct simple roots corresponding to a connected part of the Dynkin diagram, together with the roots
$ \a_j+\ldots+\a_{k-1}+2\a_{k}+\ldots +2 \a_{r-2}+\a_{r-1}+\a_r$,
$1 \le j< k \le r-2$. The longest root is
$\theta=\a_{1}+2 \a_2+\ldots 2\a_{r-2}+\a_{r-1}+\a_r$. The description of  $\Pi_\new(\Phi_0)$ and the Dynkin diagram of $\Phi_0$ is included in Table \ref{table4}.
We emphasize that when $i=r-3$, the node labeled by $\b$ is connected with nodes $\a_{r-4}$, $\a_{r-1}$, and $\a_{r}$.

\begin{table}[!ht]
\begin{center}
\begin{tabular}{lcc}
Node $i$ & $\Pi_\new(\Phi_0)$  & $\Phi_0$ \\\hline
$i=1$ & \makecell{\\ $\b=\theta$\\ \\ \vspace{-1em} }&   $\begin{dynkinDiagram}
[labels*={\a_3, \a_{r-2}, \a_{r-1},\a_{r}}]
D{*...***}
\end{dynkinDiagram}
\begin{dynkinDiagram}[labels*={\b}]
A{*}
\end{dynkinDiagram}$\\\hline
$1<i<r-2$ & \makecell{\\ $\b=\theta_{\{i-1,i,i+1\}}$\\
$\b'=\theta_{\{i,\ldots,r\}}$ \\ \\ \vspace{-2em}}&   $\begin{dynkinDiagram}
[labels*={\a_1,\a_{i-2},\b, , \a_{r-1},\a_r}]
D{*...**...***}
\end{dynkinDiagram}
\begin{dynkinDiagram}[labels*={\b'}]
A{*}
\end{dynkinDiagram}$ \\\hline
$i=r$ &\makecell{\\ $\b=\theta_{\{r-3,\ldots,r\}}$\\  \vspace{1em}}& $\begin{dynkinDiagram}
[labels*={\a_1,\a_{r-4},\b, \a_{r-3}}]
D{*..***}
\end{dynkinDiagram}
\begin{dynkinDiagram}[labels*={\a_{r-1}}]
A{*}
\end{dynkinDiagram}$
\\\hline
$i=r-2$ & \makecell{$\b=\theta_{\{r-3,r-2,r-1\}}$\\
$\b'=\theta_{\{r-3,r-2,r\}}$\\
$\b''=\theta_{\{r-1,r-2,r\}}$\\ \vspace{-1em}} & $
 \begin{dynkinDiagram}
[labels*={\a_1,\a_{r-4},\b, \b'}]
D{*..***}
\end{dynkinDiagram}
\begin{dynkinDiagram}[labels*={\b''}]
A{*}
\end{dynkinDiagram} $
 \\\hline
\end{tabular} \vspace{1em}
\end{center}\caption{\label{table4} The orthogonal root system in type $D_r$, $r\ge 4$}
\end{table}

\subsubsection{$F_4$:}\hspace{2mm} \dynkin[labels*={1,2,3,4}]F{****}

The dominant roots are
$\theta^\ell=2\a_1+3\a_2+4\a_3+2\a_4$ and $\theta^s=\a_1+2\a_2+3\a_3+2\a_4$. The description of  $\Pi_\new(\Phi_0)$ and the Dynkin diagram of $\Phi_0$ is included in Table \ref{table6}.
\begin{table}[!ht]
\renewcommand{\arraystretch}{1.2}
\begin{center}
\begin{tabular}{lcc}
Node $i$ & $\Pi_\new(\Phi_0)$  & $\Phi_0$ \\\hline
$i=3$ & \makecell{\\ \vspace{-2.5em}\\ $\b=\theta^s_{\{2,3\}}$ \\$\b'=\theta^\ell_{\{2,3,4\}}$ \\ \vspace{-1em} }
& $\begin{dynkinDiagram}[labels*={\b', \a_1, \b}]
B{***}
\end{dynkinDiagram}$\\\hline
$i=4$ & \makecell{$\b=\theta^s_{\{2,3,4\}}$\\  \vspace{-1em}} & $
\begin{dynkinDiagram}[labels*={\a_2, \a_1,\b}]
B{***}
\end{dynkinDiagram}$
\\\hline
\end{tabular}\vspace{1em}
\end{center}\caption{\label{table6} The orthogonal root system in type $F_4$}
\end{table}


\section{The residue as a zeta average}\label{sec_roots}

\subsection{}
We continue to work under the hypothesis that $\Phi$ is an irreducible root system not of type $G_2$, and $\a_i$ is a fixed short root. For $e\in\{-1, 0, 1 \}$, we denote
$$\Phi_e=\{\a\in\Phi : \la\a_i, \a\ra =e \},$$
so that $\Phi_0$ is the orthogonal root subsystem studied in the previous section.
Since $\a_i$ is short, for $\a\in \Phi$ we have $\la\a_i, \a\ra=\pm 1$, if and only if $\a$ is a short root and
$\a\notin\{\pm \a_i\}\cup \Phi_0$.
Therefore, $\Phi_{1}$, $\Phi_{-1}$ consist of short roots and we have a disjoint union
\[
\Phi=\{\pm \a_i\}\cup \Phi_0\cup \Phi_{-1}\cup \Phi_1 \cup (\Phi^\ell \setminus \Phi_0^\ell).
\]

\subsection{}
We consider the following modified version of $Z_{\Phi}(\xx; u)$, which is obtained by removing  from the denominator   of $Z_\Phi(\xx;u)$ the factors corresponding to roots in $\Phi_1^+$
\be
Z_{\Phi}^{[i]}(\xx; u)=Z_{\Phi}(\xx; u)\cdot \prod_{\a\in\Phi^+_1} (1-u^2\xx^{2\a}).
\ee
Let $Q_0$ denote the root lattice of $\Phi_0$. The inclusion of root lattices $Q_0\subset Q$ induces canonical morphisms $\F(Q_0)\subset \F(Q)$. The zeta average $Z_{\Phi_0}$ defined using the basis $\Pi(\Phi_0)$
will henceforth be regarded  as an element of $\F(Q)$ via $\F(Q_0)\subset \F(Q)$.
More explicitly, $Z_{\Phi_0}$ is by definition
a function of variables $x_j$ for  $\a_j\in \Pi(\Phi^{(i)})$ (with the notation in \S\ref{sec_parabolic}),
and of new variables $x_\b$ for
$\b\in \Pi_\new(\Phi_0)$; it is regarded as an element of $\F(Q)$, and denoted by $Z_{\Phi_0}(\xx; u)$,
via the substitution $x_\b=\xx^\b$ for all $\b\in \Pi_\new(\Phi_0)$.
With his notation, Theorem~\ref{thmA} can be restated as follows.

\begin{Thm} \label{thm_res}
Let $\Phi$ be an irreducible root system not of type $G_2$, and let $\a_i$ be a short simple root. Then,
\be
\res_{x_i=1\slash u} Z_{\Phi}^{[i]}(\xx; u)=
\left.Z_{\Phi_0}(\xx; u)\right|_{x_{i}=1\slash u}.
\ee
\end{Thm}
The proof, which is  technical,  will occupy the remainder of this section.

\subsection{} Before we start developing the technical elements that are needed for the proof, we give a succinct overview of the main argument.
Let
\be\label{e_Fdef}
F(\ux):=\res_{x_i=1\slash u} Z_{\Phi}^{[i]}(\xx; u),
\ee
where $\ux=(x_1,\ldots, x_{i-1}, x_{i+1}, \ldots, x_r)$. For  $\lam\in Q$, we denote  $\ux^\lam=\xx^{\lam}|_\evu.$
With this notation, Theorem~\ref{thm_res} states that $F(\ux)$ equals the evaluation of $Z_{\Phi_0}$ at
$x_\b=\ux^\b$, for all $\b\in\Pi_\new(\Phi_0)$.

The first step in the argument is to show that the evaluation at $\evu$ of the numerator
$N_\Phi$ of $Z_\Phi$ defined in~\eqref{e_zphi} factors as follows
$$
N_{\Phi}(\xx;u)|_\evu=2 \prod_{\a\in \Phi_{-1}^+\cup (\Phi^{\ell,+}\setminus \Phi_0^+)}  (1-u^2\ux^{m_\a \a}) \cdot N_0(\ux; u),
$$
for $N_0(\ux; u)$ a polynomial in $\ux$ such that  $N_0(\underline{0};u)=1$, and such that its degree
is explicitly bounded with respect to all variables. We use a general property of roots $\a\in\Phi^+$ such that $\la\a_i, \a\ra=-1$ that may be of independent interest (Lemma~\ref{lemma: transitive},) and
as a consequence, in Proposition \ref{prop_den} we show that
$$
F(\ux)=\frac{N_0(\ux; u)}{D_{\Phi_0}(\xx; u)|_\evu}~~.
$$

It is a direct consequence of its definition that $F(\ux)$ is fixed by the Chinta-Gunnells action of the simple reflections corresponding to $\Pi(\Phi^{(i)})$.  The second step in the argument is to show the remarkable fact that,
under the Chinta-Gunnells action of the reflection $\ss_\b\in W$, $\b\in\Pi_\new(\Phi_0)$, $F(\ux)$ transforms exactly as  $Z_{\Phi_0}(\xx; u)|_\evu $ 
does under the action of the \emph{simple} reflections associated to $\b$ in the Weyl group of $\Phi_0$.
This is accomplished in Proposition \ref{prop_ssb}, using an explicit description of the set $\Phi(\ss_\b)$
for $\b$ a dominant root for a root system $D$ as in Remark~\ref{rem_1}.

The third part of the argument is to show that $N_0(\ux; u)$ a polynomial in $\ux^\gamma$, $\gamma\in \Pi(\Phi_0)$. This is accomplished in Proposition \ref{e_polyn}, crucially relying on the results of Section \ref{sec_numzeta}, more specifically, on Proposition \ref{prop_orb} and Corollary \ref{cor_notpol}.

Finally, in \S \ref{sec_proof} we collect all these facts and use the uniqueness result of Proposition \ref{prop_unique} to conclude the argument.

\subsection{}\label{sec_denom} We begin with some preliminary results that are needed for the proof of Proposition \ref{prop_den}.
\begin{lemma}\label{lemma: transitive}
Let $\a \in \Phi^+$ such that $\la\a_i,\a\ra=-1$. Then, there exists a simple root $\a_j$ and $w\in W$ such that $\la\a_i,\a_j\ra=-1$   and
 \be\label{e_w}
w\a_i=\a_i, \ w \a_j =\a .
\ee
\end{lemma}
\begin{proof}
Assume first that $\Phi$ is simply-laced.
Let $\b=\tau \a$ with $\tau\in W^{(i)}$ be a smallest positive root (with respect to the order $\le$) in the orbit $W^{(i)} \a$.
Clearly $\la \b, \a_i\ra=\la\tau \a,\tau \a_i \ra=-1$.
Let $\a_j$ be such that $\ss_\b \a_j\in \Phi^-$. If $\b=\a_j$, then $w=\tau^{-1}$ has the required properties since $\la \b, \a_i\ra=-1$ implies $\la \a_j, \a_i\ra=-1$. If  $\la \b, \a_j\ra=1$, then  $\la \a_j, \a_i\ra=-1$, otherwise we would have $\ss_j\in W^{(i)}$ and $\ss_j \b=\b-\a_j<\b$, contradicting the minimality of $\b$.
Let $\g=\ss_j \b=\b-\a_j$, and define $w=\tau^{-1}\ss_\g$.
Note that $\la\g,\a_i\ra=0$ and $\la\g,\a_j\ra=-1$, so $w\a_i=\a_i$ and
\[w \a_j=\tau^{-1}(\a_j+\g)=\tau^{-1}\b=\a,
\]
proving~\eqref{e_w}.

Assume now $\Phi$ is double-laced. Observe that the hypothesis implies
that $\alpha\in \Phi^s$, so $\Phi$ is not of type $B_r$ (for which all short roots are mutually orthogonal).
In particular $\Phi^s$ is irreducible (see Table \ref{table2} in \S\ref{sec4.2}), and
we can use the simply-laced case just proved. We conclude that there exists a $\Phi^s$-simple
root $\alpha^\prime$ which is a neighbor of $\alpha_i$ in the diagram of $\Phi^s$, and $w^\prime$ in the Weyl group generated by reflections  in $\Phi^s$,
such that
$$
w^\prime \a_i=\a_i, \ w^\prime \a^\prime =\a .
$$
If $\a^\prime$ is a simple root in $\Phi$, we take $w=w^\prime$. If not,
then $\alpha_i$ does not have a long neighbor in the diagram of $\Phi$  (see Table ~\ref{table2}). In particular $\alpha_i$ is fixed by the action of reflections associated to simple long roots, and there exists an element $v$ in the subgroup generated by
these reflections such that $v\a^\prime=\a_j$, with $\a_j$ a short neighbor of $\a_i$. The element $w=w^\prime v^{-1}$ then satisfies the required properties.
\end{proof}

\subsection{} We use Lemma \ref{lemma: transitive} to prove the following result.
\begin{lemma}\label{lem_res1}
Let $\a \in \Phi^+$ such that $\la\a_i,\a\ra=-1$. Then, for either choice of signs,
 \[\res_{\substack{x_i=1\slash u\\ \xx^\a=\pm 1\slash u}} Z_{\Phi}(\xx; u)=0.
\]
\end{lemma}

\begin{proof}
With $w$ and $\a_j$ as in~\eqref{e_w}, we define:
\[v:= \ss_{\a_i+\a_j} w^{-1}.
\]
 We have $\ell(v)=\ell(w)+3$. With the reduced expression $\ss_{\a_i+\a_j}=\ss_i\ss_j\ss_i$,  the last three elements in $\Phi(v)$, as described in Lemma~\ref{lem_phiw}, are
\[ w \a_i=\a_i\prec w (\a_i+\a_j)=\a_i+\a \prec w \a_j=\a.
\]

We take the double residue at $\evu$, $\xx^\a=\pm 1\slash u$  in the functional equation
$Z_{\Phi}(\xx; u)=Z_{\Phi}(\xx; u)\CG v$, using~\eqref{e4} to express the right-hand side.
We have
\be\label{e_doubleres}
\res_{\substack{x_i=1\slash u\\ \xx^\a=\pm 1\slash u}} Z_\Phi(\xx; u) =
\sum_{\ddd\in \{0,1\}^{\ell_s} }
Z_\Phi(v \e_\ddd \xx)\Bigr|_{\substack{x_i=1\slash u\ \ \\ \xx^\a=\pm 1\slash u}} \cdot
\res_{\substack{x_i=1\slash u\\ \xx^\a=\pm 1\slash u}} \Pi_\ddd~~,
\ee
where $\ell_s=|\Phi^s(v)|$ and
\[\Pi_\ddd =\Pi_{v,\ddd}(\xx):= \prod_{\nu\in\Phi^s(v)}
J\left( (-1)^{\la \nu, \sum_{\mu\prec\nu}\dd_\mu \mu \ra} \xx^{\nu},\dd_\nu \right).
\]
The function
$Z_{\Phi}(v \e \xx)$ has no pole at $\evu$, $\xx^\a=\pm 1\slash u$ for any sign function $\e$,
due to the following fact.

{If $\g\in \Phi^s$, $\g=m\a_i+n\a$, for $m,n\in \Z$ with $|m+n|=1$,
then $\g\in\{\pm \a_i, \pm \a \}$. } Indeed, if $\g\ne \pm \a_i$, then $|\la \g, \a_i \ra|=|2m-n |\le 1$, since $\g$ and $\a_i$ have the same length.
It follows that $3|m|\le 2$, so $m=0$ and $|n|=1$. In consequence, $\g\in \{ \pm \a\}$, as claimed.

The poles of $Z_{\Phi}(v \e \xx)$ occur at $\xx^{v^{-1} \g}=\pm 1\slash u$, for $\g\in\Phi^{s,+}$. As explained above, a double pole for  $\evu$, $\xx^\a=\pm 1\slash u$, can only occur
for $v^{-1} \g \in \{ \a_i, \a\}$, which is impossible since $\a_i,\a\in\Phi(v)$.

For a similar reason, the product $\Pi_\ddd$  can have a  double pole
at $x_i=1\slash u$, $\xx^\a=\pm 1\slash u$, only in the terms corresponding to $\nu=\a_i$ and $\nu=\a$ since, according to the fact  above,  for   $\nu\in \Phi^s(v)\setminus\{\a_i, \a\}$ we have $u^2\xx^{2\nu}\ne 1$ under
the double evaluation.

Therefore we concentrate on the product of the last three terms in $\Pi_\ddd$, specifically
\be\label{e_J3}
J\left(\xx^{\a_i} (-1)^{\la\a_i, \nu' \ra }, \dd_{\a_i}\right)\cdot
J\left(\xx^{\a_i+\a} (-1)^{\la\a_i+\a, \nu' \ra+\dd_{\a_i} }, \dd_{\a_i+\a} \right)\cdot
J\left(\xx^{\a} (-1)^{\la\a, \nu' \ra +\dd_{\a_i}+\dd_{\a_i+\a}}, \dd_{\a}\right),
\ee
where $\nu'=\sum_{\nu\prec \a_i}\dd_\nu \nu$.
The double residue of the product~\eqref{e_J3} vanishes
unless
\be\label{e_cond}
(-1)^{\la\a_i, \nu' \ra}=1, \qquad \text{ and }\qquad
(-1)^{\la\a, \nu'\ra+\dd_{\a_i}+\dd_{\a_i+\a}}=\pm 1.
\ee
We group the terms in~\eqref{e_doubleres} in pairs corresponding to a tuples $\ddd$ and $\ddd^\prime$ that satisfy
satisfying~\eqref{e_cond}, coincide everywhere except for the last three components, and
\[\dd'_{\a_i}=1-\dd_{\a_i},\quad  \dd'_{\a_i+\a}=1-\dd_{\a_i+\a},\quad \dd'_{\a}=1-\dd_{\a}.
 \]
 Note that if $\ddd$ satisfies \eqref{e_cond} then also $\ddd'$ also satisfies~\eqref{e_cond} and $\e_{\ddd'}=\e_\ddd$. Therefore,  it is enough
to show that the double residue of $\Pi_\ddd+\Pi_{\ddd'}$ vanishes. More specifically, it is enough to show that the sum of
double residues of the products in~\eqref{e_J3} corresponding to $\ddd$ and $\ddd'$ vanishes. Now,
$\displaystyle\res_{x=1\slash u} J(x,\dd)=(1-u^2)\slash 2$ is  independent of $\dd$ and  the sum of double residues vanishes
thanks to the identity
\[
J(1\slash u^2,0)+J(-1\slash u^2,1)=0.\qedhere
\]
\end{proof}

\subsection{}
We are now ready to complete the first step in the proof of Theorem \ref{thm_res}.
\begin{Prop}\label{prop_den}We have
\[
F(\ux)=\frac{N_0(\ux; u)}
{D_{\Phi_0}(\xx; u)|_\evu},
\]
 with $N_0(\ux; u)$ a polynomial in $\ux$
 such that  $N_0(\underline{0};u)=1$ and any $\lam\in\supp(N_0(\ux; u))$ satisfies $\lam\le {\sum_{\a\in\Phi_0^+} \a }$.
\end{Prop}
\begin{proof}
 By definition, we have
\[
\res_\evu Z_{\Phi}^{[i]}(\xx; u)=\frac{N_{\Phi}(\xx; u)|_\evu}{2 \prod_{\a\in \Phi_{-1}^+\cup \Phi_0^+\cup (\Phi^{\ell,+}\setminus \Phi_0^+)} (1-u^2\ux^{m_\a \a}) },
\]
where $\Phi^{\ell,+}$ denotes the positive long roots.
By Lemma \ref{lem_res1} and Lemma \ref{lem_div}, we also have
\be\label{e_degN}
N_{\Phi}(\xx; u)|_\evu=2 \prod_{\a\in \Phi_{-1}^+\cup (\Phi^{\ell,+}\setminus \Phi_0^+)}  (1-u^2\ux^{m_\a \a}) \cdot N_0(\ux; u),
\ee
for a polynomial $N_0(\ux; u)$ (the factor 2 comes from the fact that $1+ux_i$ divides $N_{\Phi}(\xx; u)$ by Lemma~\ref{lem_div}).
 This gives the formula in the statement.

To analyze the support of $N_0(\ux; u)$, we introduce some notation.
For $\lam=\sum n_j \a_j\in Q$, denote $$\us(\lam)=\sum_{j\ne i} n_j \a_j\in Q, $$
and for a set of roots $A$, let $\us(A)=\sum_{\lam\in A}\us(\lam)$.
Since $\ss_i:\Phi_{-1}^+\rightarrow \Phi_1^+$ is an isomorphism
given by $\a\mapsto \a+\a_i$, we have $\us(\Phi_{1})=\us(\Phi_{-1})$.
By Corollary~\ref{cor_degree}, all the monomials $\lam\in\supp(N_{\Phi}(\xx; u)|_\evu)$
have
\[
\lam\le \us(\Phi^+)=\us(\Phi_0^+)+2\us(\Phi_{-1}^+)+\us(\Phi^{\ell,+}\setminus \Phi_0^+ ).
\]
Combining this with~\eqref{e_degN} finishes the proof.
\end{proof}
From the Proposition \ref{prop_den} and~\eqref{e_degN}, we deduce the following equivalent formulation of
Theorem~\ref{thm_res}.
\begin{Cor} Theorem~\ref{thm_res} is equivalent to
\[N_{\Phi}(\xx; u)|_\evu=2 \prod_{\a\in \Phi_{-1}^+\cup(\Phi^{\ell,+}\setminus \Phi_0^+)} (1-u^2\ux^{m_\a \a})\cdot N_{\Phi_0}(\xx; u)|_\evu.
\]
\end{Cor}

\subsection{} We continue by developing the technical elements that are necessary for the proof of Proposition \ref{prop_ssb}, which states
 that for every $\b\in\Pi_\new(\Phi_0)$,  the modified residue $F(\ux)$ defined in~\eqref{e_Fdef} transforms under the action of $\ss_\b\in W$ precisely as
 $Z_{\Phi_0}(\xx; u)|_\evu$ transforms under the action of the \emph{simple} reflection associated to $\b$ in $W_0$.
 This is surprising, as the reflection $\ss_\b$ is far from being a simple reflection in $W$.

By Remark~\ref{rem_1}, the elements $\b\in \Pi_\new(\Phi_0)$ are dominant roots
in subdiagrams of $\Phi$ of type $A_3$, $D_n$, or $C_n$. For the proof of Proposition \ref{prop_ssb},
we will need an explicit description of the set $\Phi(\ss_\b)$ in each  case, which is provided in the next lemma. We use the standard labelling
of the root systems from \S\ref{sec_Phi0}.

\begin{lemma} \label{lem_phissb}
\emph{(i)} Let $\b$ be the highest root in a root system $\Phi$ of type $D_n$, $n\ge 4$, or of type $A_3$. Let $\a_i=\a_1$ in the case of $D_n$, and $\a_i=\a_2$ in the case of $A_3$.
There is a reduced expression for $\ss_\b$
such that the set $\Phi(\ss_\b)$, ordered as in~\eqref{e_phiw}, is given by
\[
\{ \g_1\prec \g_1+\a_i\prec \ldots\prec \g_{t}\prec\g_{t}+\a_i\prec \b\prec
\g_{t+1}\prec \g_{t+1}+\a_i\prec \ldots\prec \g_{2t}\prec \g_{2t}+\a_i \}.
\]

\emph{(ii)} Let $\b$ be the highest short root in a root system $\Phi$ of type $C_n$,
$n\ge 2$, and let $\a_i=\a_1$. There is a reduced expression for $\ss_\b$
such that the set $\Phi(\ss_\b)$, ordered as in~\eqref{e_phiw}, is given by
\[
\{ \g_1\prec \g_1+\a_i\prec \ldots\prec \g_{t}\prec\g_{t}+\a_i\prec \b-\a_i\prec \b\prec \b+\a_i \prec
\g_{t+1}\prec \g_{t+1}+\a_i\prec \ldots\prec \g_{2t}\prec \g_{2t}+\a_i\}.
\]
In $\Phi(\ss_\b)$, the only long roots are $\b\pm \a_i$.

\emph{(iii)} Let $\b$ be the highest long root in a root system $\Phi$ of type $C_3$,
and let $\a_i=\a_2$. Then, $\ss_\b=\ss_1\ss_2 \ss_3 \ss_2\ss_1$ and
\[
 \Phi(\ss_\b)=\{\g_1\prec \g_1+\a_i \prec  \b \prec \g_2  \prec \g_2+\a_i\},
\]
where $\g_1=\a_1$ and $\g_2=\a_1+\a_2+\a_3$.

\emph{(iv)} In all cases above, we have $t=\operatorname{rank}(\Phi)-2$ and
\[
\{\g_1,\ldots, \g_{2t}\}=\Phi(\ss_\b) \cap \Phi_{-1}.
\]
In particular, $\g_j$ and $\ss_i \g_j=\g_j+\a_i$ are short roots. Furthermore,
$\la \g_j, \b^\vee \ra =1$, $1\le j\le  2t$.
\end{lemma}
\begin{proof}

(i) We prove the statement for $D_n$ by induction on $n$.
Assume $\Phi$ is of type $D_n$, $n\ge 4$, and let $\Phi'\subset \Phi$ be the root subsystem of type
$D_{n-1}$ obtained by removing
node $1$ from the Dynkin diagram of $\Phi$ (with the convention that $D_3=A_3$). Suppose the statement is
true for $\Phi'$, and we show that it is true for $\Phi$ as well.

Let $\b^\prime$ be the highest root in $\Phi^\prime$, so $\b=\a_1+\a_2+\b^\prime$. Note that $\b$ is orthogonal
to all simple roots except for $\a_2$, and $\la \beta, \a_2 \ra=1$. Since $\b=\ss_2\ss_1\b^\prime$, we have
$\ss_\b = \ss_2 \ss_1 \ss_{\b^\prime} \ss_1\ss_2.
$
From Lemma~\ref{lem_phiw}, it follows that
\be\label{e_phissb}
\Phi(\ss_\b)=\Phi(\ss_1\ss_2)\cup \ss_2\ss_1 \Phi(\ss_{\b^\prime}) \cup  \ss_2\ss_1\ss_{\b^\prime}\Phi(\ss_2\ss_1),
\ee
with the ordering $\prec$ on $\Phi(\ss_\b)$ being the concatenation of the ordering on the three sets on the right-hand side. Using that $\la \a_1,\b^\prime\ra=-1$, $\la \a_2,\b^\prime\ra =0$ we compute
\[\Phi(\ss_1\ss_2)=\{\a_2\prec \a_2+\a_1\}\quad \text{and} \quad \ss_2\ss_1\ss_{\b^\prime}\Phi(\ss_2\ss_1)=\{\b^\prime\prec \b^\prime+\a_1\}.
\]
By the induction hypothesis, the roots in $\Phi(\ss_{\b^\prime})$ come in consecutive pairs $\g\prec \g+\a_2$,
with $\b^\prime$ the central element in the set (which is of odd cardinality). We have $\la \g, \a_2\ra =-1$,
$\la \g, \a_1\ra =0$, so the corresponding elements in the set $\ss_2\ss_1 \Phi(\ss_{\b^\prime})$ are
\[
\ss_2\ss_1 (\g)=\g+\a_2 \prec \ss_2\ss_1 (\g+\a_2) =\g+\a_2+\a_1 \quad \text{and} \quad \ss_2\ss_1 (\b^\prime)=\b,
\]
finishing the induction step. Remark that $\la \a_1, \a_2\ra=\la \a_1, \b^\prime\ra=\la \a_1, \g+\a_2\ra=-1$.

The base case for induction is $D_3=A_3$, with $i=2$ the middle node.
We have $\b=\a_{1}+\a_2+\a_{3}$, $\ss_\b=\ss_{1} \ss_2 \ss_{3}\ss_2 \ss_{1}$, and formula~\eqref{e_phiw} gives
\[\Phi(\ss_\b)=\{\a_{1}\prec \a_{1}+\a_2 \prec \b\prec \a_{3}\prec \a_{3}+\a_2 \}.
\]

(ii) We again use induction on $n$, the base case being $n=2$, when $\b=\a_1+\a_2$, $\ss_\b=\ss_2\ss_1\ss_2$, and
$$
 \Phi(\ss_\b)=\{\b-\a_1\prec \b \prec \b+\a_1\}.
$$
Assume  now $\Phi$ is of type $C_n$, $n\ge 3$, and let $\Phi^\prime\subset \Phi$ be the root subsystem of type
$C_{n-1}$ obtained by removing node $1$ from the Dynkin diagram of $\Phi$. Assume that our claim is true for  $\Phi^\prime$.

Let $\b^\prime$ be the dominant short root in $\Phi^\prime$ so $\b=\a_1+\a_2+\b^\prime=\ss_2\ss_1 \b'$. Note that $\b$ is orthogonal
to all simple roots except for $\a_2$, and $\la \beta, \a_2 \ra=1$. As before we have
\[\ss_\b = \ss_2 \ss_1 \ss_{\b'} \ss_1\ss_2,
\]
and the remaining part of the argument proceeds as in part (i).

(iii) The claim follows by direct verification.

 (iv) In cases (i) and (ii), the claim follows by induction, and in (iii) by direct verification.
 \end{proof}

\subsection{} One consequence of Lemma \ref{lem_phissb} is the following identity.

\begin{lemma}\label{lem_DD}
For $\b\in\Pi_\new(\Phi_0)$, let $ \Phi(\ss_\b) \cap \Phi_{-1}=\{ \g_1, \ldots,  \g_{2t}\}$
as in Lemma~\ref{lem_phissb} (iv). Then,
\[
\prod_{\a\in \Phi_1^+} \frac{1-u^2\ux^{2\a}}{1-u^2\ux^{2\ss_\b \a}}
=\prod_{j=1}^{2t} k(\ux^{\g_j})^{-1},
\]
where $k(x)=\displaystyle \frac{1-u^2 x^{-2}}{1-x^2}$.
\end{lemma}
\begin{proof}
Since $\ss_\b\a_i=\a_i$, the reflection $\ss_\b$ keep $\Phi_e$ stable, for any  $e\in\{-1,0,1\}$.
Therefore $\ss_\b$ restricts to a bijection of $\Phi_1^+\setminus \Phi(\ss_\b)$, giving
\be\label{e_prod}
\prod_{\a\in \Phi_1^+} \frac{1-u^2\xx^{2\a}}{1-u^2\xx^{2\ss_\b \a}}
=\prod_{\g\in \Phi_1^+\cap \Phi(\ss_\b)}\frac{1-u^2\xx^{2\a}}{1-u^2\xx^{2\ss_\b \a}}
=\prod_{j=1}^{2t} \frac{1-u^2 \xx^{2(\g_j+\a_i)}}{1-u^2 \xx^{2(\g_j+\a_i-\b)}}.
\ee
Above, we used that $\Phi_{1}\cap \Phi(\ss_\b) =\{ \g_j+\a_i \mid j=1,\ldots, 2t\} $, according to Lemma~\ref{lem_phissb} (iv).
If $t=0$, we have $\Phi_1^+\cap \Phi(\ss_\b)=\emptyset$ and all the products above are equal to $1$.

We also have  $\Phi_{-1}\cap \Phi(\ss_\b) =\{ \g_j \mid 1\le j\le 2t\} $, and the
map $\g\mapsto - \ss_\b \ss_i \g=\b-\g-\a_i$ is a bijection of this set. It follows that,
for $\g\in \Phi_{-1}\cap \Phi(\ss_\b)$, we have
\[
\left.\frac{1-u^2 \xx^{2(\g+\a_i)}}{1-u^2 \xx^{2(\g+\a_i-\b)}}\right|_\evu = \frac{1-\ux^\g}{1-u^2 \ux^{-\g^\prime}}~,
\]
with $\g^\prime=\b-\g-\a_i\in \Phi_{-1}\cap \Phi(\ss_\b)  $. Taking the product of these fractions over
all $\g\in \Phi_{-1}\cap \Phi(\ss_\b)$, and comparing with~\eqref{e_prod} concludes the proof.
\end{proof}

\subsection{} We are now ready to complete the second step in the proof of Theorem \ref{thm_res}.

Remark that   $(\ss_\b \xx)_i=x_i$ and $(\e^\b\xx)_i=x_i$ for $\b\in\Pi_\new(\Phi_0)$, by~\eqref{e_sign}.
We let sign functions
$\e^\lam$ with $(\e^\lam \xx)_i=x_i$ act on the multivariable $\ux$ by restriction of their action on $\xx$,
and the reflection $\ss_\b$ by $(\ss_\b\ux)_j:=\ux^{\ss_\b\a_j}=\xx^{\ss_\b \a_j}|_\evu$ for $j\ne i$.

\begin{Prop}\label{prop_ssb} For $\b\in\Pi_\new(\Phi_0)$ we have
\[
F(\ux)=\begin{cases} F(\ss_\b \ux) J(\ux^\b,0) + F(\ss_\b\e^\b \ux) J(\ux^\b,1) & \text{ if $\b$ is short},\\
      F(\ss_\b \ux)  & \text{ if $\b$ is long}.
       \end{cases}
\]
Moreover, $F(\ux)=F(\e_i \ux)$.
\end{Prop}

\begin{proof}
Throughout the proof, we denote
\[R(\ux):=\res_{x_i=1\slash u} Z_{\Phi}(\xx; u).
\]
Taking residues in the functional equation $Z_{\Phi}(\xx; u)=Z_{\Phi}(\xx; u)\CG \ss_i$, we  obtain
\be\label{e_ressi}
R(\ux)=(1-u^2) \cdot (Z_{\Phi})^+_i (\ss_i\xx; u)|_{x_i=1\slash u}.
\ee
Therefore $R(\ux)$ is even with respect to $\e_i$, and we conclude that $F(\ux)=F(\e_i \ux)$ as well.

Since $\ss_\b \a_i=\a_i$, taking residues in the functional equation $Z_{\Phi}(\xx; u)=Z_{\Phi}(\xx; u)\CG \ss_\b$ will relate
$R(\ux)$ to a linear combination of $R(\ss_\b\e\ux)$ for certain sign functions $\e$, using
 formula~\eqref{e4} for $w=\ss_\b$. If the sign function
 $\displaystyle\e_\ddd=\e^{\sum_{\g\in \Phi(\ss_\b)} \dd_\g \g}$
in~\eqref{e4} changes the sign of $x_i$, then the residue of the corresponding term
in~\eqref{e4} vanishes, since $\displaystyle\res_{x_i=-1\slash u} Z_{\Phi}=0$.
Denote by $E$  the set of all $\e_\ddd$ with $(\e_\ddd\xx)_i=x_i$. We obtain
\be\label{e5}
R(\ux)=\sum_{\substack{\ddd\in \{0,1\}^{\ell_s} \\ \e_\ddd\in E }} R(\ss_\b \e_\ddd \ux)\cdot \Pi_\ddd ~~,
\ee
where
\be\label{e_ce}
\Pi_\ddd=\Pi_{\ss_\b,\ddd}(\ux):=\prod_{\g\in\Phi^s(\ss_\b)}
J\left( (-1)^{\la\g, \sum_{\a\prec\g}\dd_\a \a \ra} \ux^{\g},\dd_\g \right),
\ee
and $\Phi^s(\ss_\b)$,  $\ell_s=|\Phi^s(\ss_\b)|$ are as in Lemma~\ref{lem_CG}. Recall that
the order $\prec$ on  $\Phi^s(\ss_\b)$ is induced by the order on $\Phi(\ss_\b)$ in~\eqref{e_phiw}.
In all cases we have, by Lemma~\ref{lem_phissb},
\[
\Phi^s(\ss_\b)=
\{ \g_1\prec \g_1+\a_i\prec \ldots\prec \g_{t}\prec\g_{t}+\a_i\prec \b\prec
\g_{t+1}\prec \g_{t+1}+\a_i\prec \ldots\prec \g_{2t}\prec \g_{2t}+\a_i \},
\]
for some $t\ge 0$, with the central element $\b$ missing if $\b$ is a long root
as in  Lemma \ref{lem_phissb} (iii). We have $t=0$ only when $\b$ is
the dominant short root in a subsystem of type $C_2=B_2$, in which case $\Phi^s(\ss_\b)=\{\b\}$.

For a sign $\e_\ddd\in E$, the condition $(\e_\ddd\xx)_i=x_i$ translates, by~\eqref{e_sign}, to
\be\label{e_par}
\sum_{j=1}^{2t} (\dd_{\g_j}+\dd_{\g_j+\a_i})\equiv 0\!\!\! \mod 2~~,
\ee
the condition being automatically satisfied if $t=0$.
Assume $t>0$, and let $\g:=\g_{2t}$, so the last two elements in $\Phi^s(\ss_\b)$ are $\g$ and $\g+\a_i$.
For $\ddd\in E$, define a tuplet $\ddd^\prime$  such that $\ddd^\prime$ and $\ddd$ are identical, except on the last two positions, for which
\[
\dd^\prime_{\g}=1-\dd_{\g},\qquad \dd^\prime_{\g+\a_i}=1- \dd_{\g+\a_i}.
\]
Note that $\e_{\ddd^\prime}=\e_\ddd\e_i$, so $R(\ss_\b \e_\ddd \ux)=R(\ss_\b \e_{\ddd^\prime} \ux)$ in~\eqref{e5}
can be taken as common factor in front of the sum $\Pi_\ddd+\Pi_{\ddd^\prime}$. To compute this sum, remark that
the first $\ell_s-2$ factors in $\Pi_{\ddd}$ and $\Pi_{\ddd^\prime}$ are the same,
and the last two factors of $\Pi_\ddd$ are
\be\label{e_Jlast}
J( (-1)^{\la\g,\g' \ra } \ux^{\g} ,  \dd_\g)\cdot
J( (-1)^{\la\g,\g'\ra +\dd_{\g+\a_i}} \ux^{\g}\slash u,  \dd_{\g+\a_i}),
\ee
with $\g'=\sum_{\a<\g} \dd_\a \a$. Here, we have used condition~\eqref{e_par}
and Lemma~\ref{lem_phissb} (iv).

We now use the following identities
\be\label{e_Jident}
\begin{gathered}
J(x,0)J(x/u,0)+J(x,1) J(-x/u,1) = k(x),\\
 J(x,1)J(x/u,0)+J(x,0) J(-x/u,1)=0,
 \end{gathered}
 \ee
where
$$k(x):=\frac{1-u^2x^{-2}}{1-x^2}.$$
By~\eqref{e_Jident},
the sum of the product in~\eqref{e_Jlast} with the corresponding product for $\Pi_{\ddd'}$
equals $k(\ux^\g)$ or 0, depending on whether $\dd_\g$ and $\dd_{\g+\a_i}$
are the same or not, respectively. We conclude that the sum $\Pi_\ddd+\Pi_{\ddd'}$ vanishes,
unless $\dd_\g=\dd_{\g+\a_i}$, when it equals $k(\ux^{\g_{2s}})$ times a product involving
only on the first $\ell_s-2$ elements of $\ddd$. When $\dd_{\g}=\dd_{\g+\a_i}$, we also have
\[
\sum_{j=1}^{2t-1} (\dd_{\g_j}+\dd_{\g_j+\a_i})\equiv 0\!\!\!\mod 2.
\]
Repeating the same reasoning with $\g=\g_{2t-1}$ and smaller indices, we conclude that only
the tuplets $\ddd$ having $\dd_{\g_j}=\dd_{\g_j+\a_i}$, for all $1\le j\le 2t$, contribute
non-trivially to the sum in~\eqref{e5}. We obtain
\[
R(\ux)= \prod_{j=1}^{2t} k(\ux^{\g_{j}})\cdot\begin{cases}
              R(\ss_\b \ux) J(\ux^\b,0) + R(\ss_\b\e^\b \ux) J(\ux^\b,1) & \text{ if $\b$ is short},\\
      R(\ss_\b \ux)  & \text{ if $\b$ is long.}
                                        \end{cases}
\]
We also have
$$F(\ux)=R(\ux)\cdot \prod_{\a\in \Phi_1^+} (1-u^2 \ux^{2\a}),$$
and the identity in Lemma~\ref{lem_DD} concludes the argument.
\end{proof}

As a consequence of Proposition~\ref{prop_ssb}, we have the following complement to Lemma~\ref{lem_res1}.
\begin{lemma} \label{lem_res2}  For  $\b\in\Pi_\new(\Phi_0)$, we have
\[\res_{\substack{ x_i=1\slash u \\
 \xx^{\b}=-1\slash u }} Z_{\Phi}(\xx; u)= 0.
\]
\end{lemma}
\begin{Cor}\label{cor_div}
The numerator $N_0(\ux; u)$ from Proposition~\ref{prop_den} is divisible by $1+u \ux^\g$, for any short root
$\g\in \Pi(\Phi_0)$.
\end{Cor}

\subsection{} The third step in the proof of Theorem \ref{thm_res} is the following.

\begin{Prop}\label{e_polyn}
 $N_0(\ux; u)$ is a polynomial in $\ux^\g$, $\g\in \Pi(\Phi_0)$.
\end{Prop}

\begin{proof}
We distinguish two cases.

\textit{Case I}: the root system $\Phi$ is not of type $A_r$. In this case, the number
of roots $\b\in \Pi_\new(\Phi_0)$ equals the number of neighbors of the node $i$ in the Dynkin diagram of $\Phi$,
as it can be seen from the tables in \S\ref{sec_Phi0} and  Appendix~\ref{sec_exc}.
By Proposition~\ref{prop_ssb} the function
$N_0(\ux; u)$ is even under the sign function $\e_i$, which changes the sign of
$x_j$ precisely for $\a_j$ such that $\la\a_i, \a_j\ra=-1$ (necessarily, $\a_j$ is a short root). Using that the cardinality of
$\Pi_\new(\Phi_0)$ equals the number of neighbors of the node $i$ in the Dynkin diagram of $\Phi$,
it follows that in each monomial appearing in $N_0(\ux; u)$ we can make a substitution
\[\prod_{\la \a_j,\a_i\ra<0} x_j^{a_j} =m \cdot \prod_{\b\in \Pi_\new(\Phi_0)} (\ux^\b )^{c_\b},
\]
where $a_j\ge 0$ and $m$ is a Laurent monomial in $u$ and variables $x_k$ with $\la \a_k,\a_i\ra=0$.
Taking into account that $\displaystyle \sum_{\la\a_j,\a_i\ra=-1} a_j$ is even,
a verification of the cases in \S\ref{sec_Phi0} and in Appendix~\ref{sec_exc} shows that all
exponents $c_\b$ on the right-hand side are integral, and at least one is positive if one of the
$a_j$ is positive on the left-hand side. Therefore, after the substitution above, $N_0(\ux; u)$ becomes a Laurent polynomial in the variables
$\ux^\g$ for $\g\in \Pi(\Phi_0)$, such that each monomial that contains some negative exponents also
contains a factor $(\ux^\b )^c$ with $c>0$, for some $\b\in \Pi_\new(\Phi_0)$.
By Corollary~\ref{cor_notpol}, this is possible only
if $N_0(\ux; u)$ is a polynomial in the variables
$\ux^\g$ after the substitution above.

\textit{Case II}: the root system $\Phi$ is of type $A_r$. If $i=1$, then $\Phi_0$ is the
root system of type $A_{r-2}$ with simple roots $\a_k$, $k\ge 3$. The bound on degree in
Proposition~\ref{prop_den} shows that $N_0(\ux; u)$ does not depend on~$x_2$, which is our claim.

If $i\ne 1,r$, then $\b=\a_{i-1}+\a_i+\a_{i+1}$ is the unique root in
$\Pi_\new(\Phi_0)$. We want to show that $N_0(\ux; u)$ is a
polynomial in $u\ux^\b=x_{i-1}x_{i+1}$ and $x_k$, $k\not\in \{i-1,i,i+1\}$, so we decompose
\[
F(\ux)=\sum_{a>0} x_{i-1}^{2a} f_a(\ux) +f_0(\ux)+\sum_{a>0} x_{i+1}^{2a} g_{a}(\ux)
\]
where $f_a(\ux), g_a(\ux)$ are of the form $P_a(\ux)/D_{\Phi_0}(\xx)|_\evu$ with $P_a(\ux)$ a polynomial in
$\ux^\g$, $\g\in\Pi(\Phi_0)$. The exponents of $x_{i\pm 1}$ are even in this expression because
$F(\ux)$ is even with respect to $\e_i$.

We claim that $f_a(\ux)=g_a(\ux)=0$, for all $a\ne 0$. By symmetry we concentrate on $f_a(\ux)$.
The decomposition above is preserved by the actions $\CG \ss_k$ for $\la \a_k, \a_i\ra=0$,  and $\CG \ss_\b$. It follows that, for $a>0$, the function $f_a(\ux)$ is the specialization
at $x_i=1\slash u$ of a function invariant under the twisted action of the Weyl group of $\Phi_0$,
for some twisting parameter $\om$ in the weight lattice of $\Phi_0$, as in Section~\ref{sec_numzeta}.
We now identify $f_a(\ux)$ with this invariant function.

For a contradiction, assume that  $f_a(\ux)\ne 0$, and write $f_a(\ux)=P_a(\ux)/D_{\Phi_0}(\xx)|_\evu$
as above. Proposition~\ref{prop_orb} applied to $\Phi_0$, implies that  there is a strongly dominant
weight $\xi$ such that $O_\xi\subset \supp(P_a(\ux))$. The bound on degree in Proposition~\ref{prop_den},
implies that $0\le\lam<2\rho_0$ for $\lam\in \supp(P_a(\ux))$, with $\rho_0$ and $w_\circ$ being the half-sum of positive roots in $\Phi_0$ and the longest element  in the Weyl group of $\Phi_0$.
Setting $\theta=\om+\rho_0$ as in Section~\ref{sec_numzeta}, we have
\[0\le\theta-\xi\le\theta-w_0\xi<2\rho_0=\rho_0-w_\circ\rho_0.
\]
It follows that $\xi-\rho_0<w_\circ (\xi-\rho_0)$,
which is impossible since
$\xi-\rho_0\in Q_0^+$ (non-negative integral linear combinations of elements in $\Phi_0^+$), and $w_\circ$ maps $\Phi_0^+$ onto $\Phi_0^-$.
The contradiction shows that $f_a(\ux)=0$. Therefore, $F(\ux)=f_0(\ux)$, which is precisely our claim.
\end{proof}

\subsection{Proof of Theorem~\ref{thm_res}}\label{sec_proof}
We are now ready to assemble all the results in this section to prove Theorem~\ref{thm_res}.
By Proposition~\ref{prop_den} and Proposition \ref{e_polyn}, we have
\[
F(\ux)=\frac{N_0(\ux;u)}{D_{\Phi_0}(\xx; u)|_\evu}~~,
\]
where  $N_0(\ux; u)$ is a polynomial in $\ux^\g$, $\g\in \Pi(\Phi_0)$. Proposition~\ref{prop_ssb} shows that $F(\ux)$
has the same transformation properties as $Z_{\Phi_0}(\xx; u)|_\evu$.
If $\Phi_0$ is irreducible,
Proposition~\ref{prop_unique} and the fact that $N_0(\underline{0}; u)=1$ finishes
the proof of~Theorem~\ref{thm_res} in this case.

If $\Phi_0$ is reducible (which is the case when $\Phi$ is of type $C_n$, $n\ge 3$, or $D_n$, $n\ge 4$), the argument in the previous paragraph
has to be slightly adjusted. If $\Phi$ is not of type $D_4$, we have $\Phi_0=\Phi_0^\prime\cup \{\pm \g\}$
with $\Phi_0^\prime$ irreducible and $\g\in\Pi(\Phi_0)$ orthogonal to $\Phi_0^\prime$.
From Corollary~\ref{cor_div}, we have   $N_0(\ux)=(1+u \ux^{\g}) N_0^\prime(\ux)$ for some polynomial
$N_0^\prime(\ux)$. We obtain that
\[
F(\ux)=\frac{1}{1-u \ux^{\g}} F^\prime(\ux), \quad
F^\prime(\ux):=\frac{N_0'(\ux)}{D_{\Phi_0^\prime}(\xx)|_\evu},
\]
and $F^\prime(\ux)$ is invariant under the Weyl group of the irreducible component $\Phi_0^\prime$. Moreover,
$N_0^\prime(\ux)$ satisfies the conclusion of Proposition \ref{e_polyn} for $\Phi_0$ replaced with $\Phi_0^\prime$. Therefore, we can
apply Proposition~\ref{prop_unique} as before to conclude that $F^\prime(\ux)=Z_{\Phi_0^\prime}(\xx; u)|_\evu $. In consequence,
we have
$$F(\ux)=Z_{\Phi_0'}(\xx; u)|_\evu \cdot  Z_{A_1}(\ux^\g; u)=Z_{\Phi_0}(\xx; u)|_\evu. $$

Finally, if $\Phi$ is of type $D_4$, then $\Phi_0$ is isomorphic to the direct sum of three root systems of type   $A_1$, and
a similar argument applies. Therefore, the proof of  Theorem~\ref{thm_res} is concluded.

\section{Parabolic subgroup averages}\label{sec_main}

\subsection{}\label{sec_cascade}
This section is dedicated to the proof of Theorem \ref{thmD}. We continue to work under the hypothesis that $\Phi$ is an irreducible root system not of type $G_2$, and $\a_i$ is a fixed short root.
We first describe the kernel function that appears in the statement.
The description involves a finite directed graph $\mathcal{K}_{\Phi}(\a_i)$ with vertices labeled by positive
roots, akin to Kostant's cascade construction \cite{K}.
The directed graph is the Hasse diagram (the graphical representation of the cover relations) of a finite partial order relation on the set labeling the vertices. If  $\b,\g\in \mathcal{K}_\Phi(\a_i)$ and $\b$ is immediately  followed by $\g$ in the partial order (i.e. there is a directed edge from $\b$ to $\g$), we write $\b\lessdot \g$. We will routinely interchange between these two equivalent descriptions of $\mathcal{K}_{\Phi}(\a_i)$ (directed graph and partial order relation).
The partial order will match the order $\le$ restricted to the set of positive roots that label the vertices of $\mathcal{K}_\Phi(\a_i)$.


Associated to  $\mathcal{K}_\Phi(\a_i)$, there is an auxiliary copy $\mathcal{F}_\Phi(\a_i)$ of the same graph,
whose vertices are  irreducible root sub-systems (with corresponding bases)
of $\Phi$. If $\g\lessdot \g'$ is a directed edge in $\mathcal{K}_\Phi(\a_i)$, then the corresponding
vertices $\Psi$, $\Psi'$ in  $\mathcal{F}_\Phi(\a_i)$ are irreducible root systems with $\Psi'\subset \Psi$,
such that $\g\in \Pi(\Psi)$ and  $\g'\in \Pi(\Psi')$.
The bases for the root systems in $\mathcal{F}_\Phi(\a_i)$ are inherited from the basis $\Pi(\Phi)$,
 and they will not be included in the notation.

The two directed graphs are constructed recursively. The minimal element of
$\mathcal{K}_\Phi(\a_i)$ is $\a_i$, and the corresponding vertex in $\mathcal{F}_\Phi(\a_i)$
is $\Phi$, with basis $\Pi(\Phi)$. Given a vertex labelled $\b$ in
$\mathcal{K}_\Phi(\a_i)$ and the corresponding irreducible root system $\Psi$ in $\mathcal{F}_\Phi(\a_i)$,
with basis $\Pi(\Psi)$, the vertices $\g$ such that $\b\lessdot \g$ in $\mathcal{K}_\Phi(\a_i)$
and their corresponding root systems in $\mathcal{F}_\Phi(\a_i)$  are constructed as follows. If $\b$ is a long
root, then it is a terminal vertex in $\mathcal{K}_\Phi(\a_i)$. Otherwise,
let  $\b^\perp\subset \Psi$ be the orthogonal sub-system  that consists of roots orthogonal to  $\b$.
The basis $\Pi(\Psi)$ induces a basis $\Pi(\b^\perp)$, and we
denote by $\Pi_\new(\b^\perp)$ the set of elements of  $\Pi(\b^\perp)$ that are not in $\Pi(\Psi)$.
Then the vertices $\g$ such that $\b\lessdot \g$ in $\mathcal{K}_\Phi(\a_i)$ are precisely the elements in $\Pi_\new(\b^\perp)$; the corresponding root system in $\mathcal{F}_\Phi(\a_i)$ is the irreducible component
of $\b^\perp$ that contains $\g$. Naturally, if $\Pi_\new(\b^\perp)$ is empty, then $\b$ is a terminal vertex in
$\mathcal{K}_\Phi(\a_i)$.

The description of $\Pi_\new(\b^\perp)$
given in Section \ref{sec: roots} applies, so it is easy to construct the two graphs in all classifications.
In particular, one checks that the graph  $\mathcal{F}_\Phi(\a_i)$ is well-defined, namely the
root system associated with a given vertex only depends on the corresponding root of $\mathcal{K}_\Phi(\a_i)$.

 For certain nodes $i$, the directed graph $\mathcal{K}_\Phi(\a_i)$ is a rooted tree   isomorphic to the so-called Kostant cascade, a decreasing rooted tree of strongly orthogonal roots defined in~\cites{J, K}. Each vertex in the Kostant cascade is the highest root of an associated irreducible root system. The root vertex in the Kostant cascade is the highest root in $\Phi$. For a fixed $\b$ vertex, the vertices immediately lower in the tree order are the highest roots of the irreducible components of the root sub-system $\b^\perp$.

\begin{Exp} We give two examples for which  $\mathcal{K}_\Phi(\a_i)$ is a rooted tree isomorphic to the cascade of roots in~\cite{J}*{Table~III}.  For $\Phi$ of type $A_r$ and $r=2i-1$, then   $\mathcal{K}_\Phi(\a_i)$ is
a chain  (that is, a directed tree with one terminal vertex)
$$ \b_1\lessdot \b_2\lessdot \ldots \lessdot \b_i,$$
with $\b_j=\a_{i-j+1}+\ldots+ \a_{i+j-1}$.
For $\Phi$ of type $D_r$ with $r=2i$ even, the directed graph  $\mathcal{K}_\Phi(\a_i)$ is pictured
in Figure \ref{fig2}. For these examples, the kernel functions $K_{\Phi,\a_i}(\xx)$  defined in \S\ref{sec: kernel} below, are explicitly indicated in \S\ref{sec: thmD}.
 \end{Exp}
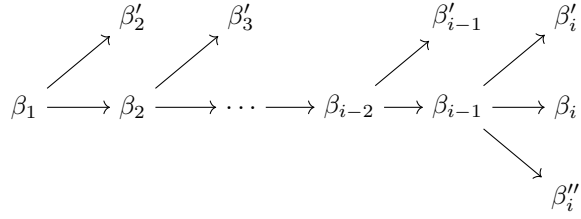
\begin{figure}[!ht]
\begin{tikzpicture}[scale=0.8,
level 1/.style={level distance=18mm},
a/.style={edge from parent/.style={draw, ->, solid} }]
\node {$\b_1$}[grow=right]
 child[missing]
 child [a]{ node{$\b_2$}
            child [missing]
            child[a]{ node{$\ldots$}
                        child [missing]
                        child [a]{ node{$\b_{i-2}$}
                                    child [missing]
                                    child [a]{ node{$\b_{i-1}$}
                                            child [a]{ node{$\b_{i}''$}}
                                            child [a]{ node{$\b_{i}$}}
                                            child [a]{ node{$\b_{i}'$}}
                                            }
                                    child [a]{ node{$\b_{i-1}'$} }
                                  }
                        child [missing]
                    }
            child [a]{ node{$\b_{3}'$} }
  		 }
  child [a]{ node{$\b_2'$}};
\end{tikzpicture}\caption{\label{fig2} The cascade  $\mathcal{K}_{D_{2i}}(\a_i)$}
\end{figure}

\subsection{}  The kernel function $K_{\Phi,\a_i}(\xx)$ in
Theorem~\ref{thmD} is defined only when the directed graph $\mathcal{K}_\Phi(\a_i)$ has a special structure,
as described in the next lemma.

\begin{lemma}\label{lem_cond}
Let $\Phi$ be an irreducible root system not of type $G_2$ and $\a_i$ a simple short root. The following  are equivalent
\begin{enumerate}
\item All the roots  $\b\in\mathcal{K}_\Phi(\a_i)$ have $n_i(\b)=1$.
\item The node $i$ is one of the admissible nodes in Table \ref{table1}.
\end{enumerate}
Furthermore, if these conditions are satisfied, then removing from $\mathcal{K}_\Phi(\a_i)$
the terminal vertices yields a chain.
\end{lemma}
In particular, $\mathcal{K}_\Phi(\a_i)$ is a rooted tree under the assumptions of the lemma.
\begin{proof} The equivalence is verified using the recursive construction
of $\mathcal{K}_\Phi(\a_i)$, and the information in the tables in \S\ref{sec_Phi0} and Appendix~\ref{sec_exc}.
The last statement also follows from a case by case analysis.
\end{proof}
\begin{remark}\label{rem_theta}
If the conditions in Lemma \ref{lem_cond} are satisfied,
then the highest root $\theta\in\Phi$ has $n_i(\theta)\le 2$. This condition is satisfied automatically except in the four exceptional root systems.
\end{remark}

\subsection{}  \label{sec: kernel}
Let $i$ be one of the admissible nodes in Table \ref{table1}.
By the lemma, removing from $\mathcal{K}_\Phi(\a_i)$ the terminal vertices yields
a (possibly empty) chain
$$\a_i=\b_1\lessdot \b_2\lessdot \ldots \lessdot \b_N;$$
let  $\Psi_1\supset\Psi_2\supset\ldots\supset \Psi_N$ be the corresponding root
systems in $\mathcal{F}_\Phi(\a_i)$.
This chain structure is used to define the kernel function $K_{\Phi,\a_i}(\xx)$, and ultimately makes possible the induction argument in the proof of Theorem~\ref{thm_main}. We recursively define the kernel function
 $K_{\Phi,\a_i}(\xx)$ as follows.
 \begin{itemize}
  \item  When $\Phi$ is simply-laced, we define
 \be\label{e_ker}
K_{\Phi,\a_i}(\xx) =\prod_{j=1}^{N} \prod_{\g\in \Pi_\new(\b_j^\perp)} \frac{1}{1-\xx^{\g-\b_j}}~~,
\ee
so that we have for $1\le j\le N$:
\[
K_{\Psi_j, \b_j}(\xx)=K_{\Psi_{j+1}, \b_{j+1}}(\xx)\cdot
 \prod_{\g\in \Pi_\new(\b_j^\perp)} \frac{1}{1-\xx^{\g-\b_j}},
 \]
 setting $K_{\Psi_{N+1}, \b_{N+1}}(\xx)=1$.
 \item When $\Phi$ is double-laced, we define
\be\label{e_ker_double}
K_{\Phi,\a_i}(\xx) =\prod_{j=1}^{N} \prod_{\g\in \Pi^*_\new(\b_j^\perp)} \frac{1}{1-\xx^{\g-\b_j}}~~.
\ee
Recall that $\Pi_\new^*(\b^\perp)$ consists of the elements in $\Pi_\new(\b^\perp) $
which are highest roots in a subdiagram of type $A_3$.
Therefore $K_{\Phi,\a_i}(\xx)=1$, unless $\Phi$ is of type $C_r$ and $1<i<r-1$, when
the explicit formula for $K_{\Phi, \a_i}(\xx)$ is given in \S\ref{sec: thmD}.
 \end{itemize}

Remark that, according to the conditions in Lemma \ref{lem_cond}, we have  $n_i(\b)=1$ for all $\b\in \mathcal{K}_\Phi(\a_i)$. In consequence,  $K_{\Phi,\a_i}(\xx)$ is independent of $x_i$ and there are no extra poles involving $x_i$ in the right-hand side of the formula in Theorem~\ref{thm_main} below.

\subsection{} We consider the following modified version of $Z_{\Phi}(\xx; u)$, obtained by removing the poles $x^\a=\pm 1\slash u$ with $n_i(\a)\ge 2$, as follows
\be
Z_{\Phi}^{(i)}(\xx; u)=Z_{\Phi}(\xx; u)\cdot
\prod_{\substack{\a\in\Phi^s\\ n_i(\a)\ge 2 }} (1-u^2\xx^{2\a}), \quad \quad
Z_{\Phi_0}^{(i)}(\xx; u)=Z_{\Phi_0}(\xx; u)\cdot
\prod_{\substack{\a\in\Phi_0^s\\ n_i(\a)\ge 2 }} (1-u^2\xx^{2\a}).
\ee
This is motivated by the observation that the  right-hand side of formula~\eqref{e_main} below
has poles involving $x_i$ only at $x^\a=\pm 1\slash u$ for $\a\in\Phi^+$ with $n_i(\a)=1$.  With his notation, Theorem \ref{thmD} can be restated as follows.

\begin{Thm}\label{thm_main} Let $\Phi$ be an irreducible root system not of type $G_2$, and let
$i$ one of the admissible nodes specified in Table \ref{table1}. We have,
\be\label{e_main}
Z_{\Phi}^{(i)}(\xx; u)=\frac{\displaystyle\sum_{w\in W^i}\left.\frac{1}{1-ux_i}K_{\Phi,\a_i}(\xx)\right|w }{\DD_{\Phi^i}(\xx)}~~.
\ee
\end{Thm}
\begin{remark}\label{rem_nitheta}
 The theorem is sharp, in the sense that the identity in the theorem does not hold as stated
for nodes $i$ not listed  in Table \ref{table1}, with $K_{\Phi,\a_i}$ defined as  above, using a longest chain
in the directed graph obtained from $\mathcal{K}_\Phi(\a_i)$ by removing its terminal vertices.
We verified this numerically for  root systems $\Phi$
of small rank ($r\le 10$), by evaluating some of the variables to random numbers.

However there are similar formulas if one allows
for more general kernel functions; we give here an example for $\Phi$ of type $D_6$ and
$i=4$. The graph $\mathcal{K}_\Phi(\a_i)$ is given in Figure~\ref{fig21},
where $\b_1=\a_4$, $\Pi_\new(\b_1^\perp)=\{\b_2$, $\b_2'$, $\b_2''\}$,
and $\theta$ is the highest root in~$\Phi$.
Using a computer, we verified that formula~\eqref{e_main} holds with
\[K_{\Phi,\a_i}(\xx)=(1+u\xx^\theta|_\evu)\cdot \prod_{\g\in \Pi_\new(\b_1^\perp)} \frac{1}{1-\xx^{\g-\b_1}}.
\]
One can prove this along the same lines as Theorem~\ref{thm_main}, but we leave a more thorough investigation
of the cases not covered in Theorem~\ref{thm_main} for future work.
\begin{figure}[!ht]
\begin{tikzpicture}[scale=0.8]
\tikzset{edge/.style = {->}}
\node (a) at (0,0) {$\b_1$};
\node (b) at (2,0) {$\b_2''$};
\node (c) at (2,1) {$\b_2$};
\node (d) at (2,-1) {$\b_2'$};
\node (e) at (4,0) {$\theta$};
\draw[edge] (a) to (b);
\draw[edge] (a) to (c);
\draw[edge] (a) to (d);
\draw[edge] (c) to (e);
\draw[edge] (d) to (e);
\end{tikzpicture}\caption{\label{fig21} The directed graph  $\mathcal{K}_{D_{6}}(\a_4)$}
\end{figure}
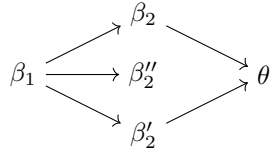
\end{remark}

The proof of Theorem \ref{thm_main} will  occupy the remainder of this section. The structure of the argument is the following. What  makes the argument possible, is a basic uniqueness result for rational functions with prescribed poles and invariance properties. This is stated as Lemma \ref{lem_invar}. We show in Proposition \ref{prop_res1} that the uniqueness result ultimately reduces the proof of Theorem \ref{thm_main} to the equality of the residues at $\evu$  of both sides of \eqref{e_main}. The residue of  $Z_{\Phi}^{(i)}(\xx; u)$ is computed in Proposition \ref{thm_res1}, which is essentially a reformulation of Theorem~\ref{thm_res}. Finally, the equality of the residues on both sides of \eqref{e_main} follows from Theorem~\ref{thm_res1}, by induction on the rank of $\Phi$.

\subsection{}\label{sec_reduction}
The next result shows that a rational function is determined uniquely by its residues at $x_i=\pm 1\slash u$,
provided it is invariant under the Chinta-Gunnells action of the maximal parabolic subgroup $W^i$
and satisfies some easily verified conditions. For later use,
we formulate the lemma allowing for potentially more general parabolic subgroups in place of $W^i$.
For a rational function $f(\xx)$, we denote by $\deg_{x_i}f(\xx)$
the degree of its numerator minus the degree of its denominator with respect to the variable $x_i$.

\begin{lemma}\label{lem_invar} Let $\Phi$ be a  simply-laced root system and fix $\a_i$, a simple root. Let $W^\prime$ be a parabolic subgroup corresponding to a subdiagram of the Dynkin
diagram of $\Phi$ such that $\ss_i\not\in W^\prime$. Let $f(\xx)$ be a rational function that satisfies the following properties
\begin{enumerate}[label=(\alph*)]
 \item $f(\xx)=f(\xx){\CG} w$ for all $w\in W^\prime$;
 \item $f(\xx)$ has only simple poles as a function of the variable $x_i$,
 and all poles involving $x_i$ occur among $\xx^{w\a_i}=\pm 1\slash u$, for $w\in W^\prime$;
\item We have $\deg_{x_i}f(\xx) <0 $.
\end{enumerate}
Then,  $f(\xx)$ is uniquely determined by the two residues $\displaystyle\res_{x_i=\pm 1\slash u} f(\xx)$.
\end{lemma}
\begin{proof}
 Using the fact that $\deg_{x_i}f(\xx) <0 $, we can decompose $f(\xx)$ into simple fractions
 \[f(\xx)=\sum_{w\in W^\prime/W^{\prime\prime}} \frac{f_w(\xx)}{1-u^2 \xx^{2w \a_i}}~,
 \]
 where $W^{\prime\prime}=\stab_{W^\prime} \a_i$, the sum is over a system of representatives
 for the coset space $W^\prime\slash W^{\prime\prime}$ and  $$f_w(\xx)=a_w(\ux)+b_w(\ux) x_i,$$ for $a_w(\ux),~b_w(\ux)$ polynomials in $\ux$.
 Since
 the denominators are even with respect to all the signs $\e_i$, it follows that for $w^\prime\in W^\prime$ we have
 \[f(\xx)=f(\xx)\CG w^\prime= \sum_{w\in W^\prime/W^{\prime\prime}}\frac{f_w(\xx){\CG} w^\prime}{1-u^2 \xx^{2w^{\prime-1}w \a_i}}
 \]
 The functions $f_w(\xx){\CG} w^\prime$ are also of the type $a(\ux)+b(\ux) x_i$. The uniqueness
 of the decomposition into simple fractions implies that $f_w(\xx){\CG} w^\prime=f_{w^{\prime-1}w}(\xx)$, so
 all the coefficients $f_w(\xx)$ are determined by $f_{I}(\xx)$, with $I$ representing the trivial coset.
But $f_{I}(\xx)$ is determined by $\displaystyle\res_{x_i=\pm 1\slash u} f(\xx)$. Therefore, $f(\xx)$ is uniquely determined by the residues $\displaystyle\res_{x_i=\pm 1\slash u} f(\xx)$.
\end{proof}

\subsection{}
We illustrate the use of Lemma \ref{lem_invar} on three examples. The first illustrates in a simple case
the proof of Theorem~\ref{thm_main}, while the other two will be needed later in the proof.
For a list $I$ of indices, we denote by $\Phi^I$ the parabolic root sub-system of $\Phi$
obtained by removing the nodes in $I$, and by $W^I$ the corresponding parabolic sub-group of $W$.

\begin{Exp}\label{exp: case1}
 Let $\Phi$ be the root system of type $A_2$. Then,
\[
 Z_\Phi(\xx; u)=
\frac{\displaystyle \sum_{w\in W^{2}}\left.\frac{1}{(1-ux_2)}\frac{1}{(1-x_1x_3)} \right|w }{\DD_{\Phi^{2}}(\xx)} ~~.
\]
This is a particular example of the equality in Theorem \ref{thm_main}. To verify it, use  Lemma~\ref{lem_invar} for $\Phi$, the node  $i=2$,
and $W^\prime=W^{2}$, to show that it is enough to prove the corresponding equality of residues at $x_2=1\slash u$. The equality of the residues is equivalent to
\[
\res_{x_2=1\slash u} Z^{[2]}_\Phi(\xx; u)=\frac{1}{(1-x_1x_3)},
\]
which is precisely the claim of Theorem \ref{thm_res} for this case.
\end{Exp}
\begin{Exp}\label{exp: case2}
 Let $\Phi$ be the root system of type $A_3$,  and let $\theta$ denote its highest root. Then,
\[
(1-u^2\xx^{\theta})\cdot Z_\Phi(\xx; u)=
\frac{\displaystyle \sum_{w\in W^{1,3}}\left.\frac{1}{(1-ux_1)(1-ux_3)} \right|w }{\DD_{\Phi^{1,3}}(\xx)} ~~.
\]
Indeed,  Lemma~\ref{lem_invar} for $\Phi$, the node  $i=3$,
and $W^\prime=W^{1,3}$, is used to show that it is enough to prove the corresponding equality for the residues at $x_3=1\slash u$. This equality of the residues follows from the application of Theorem \ref{thm_res}.
\end{Exp}
\begin{Exp}\label{exp: case3}
Let $\Phi$ be the root system of type $A_5$,  and let $\theta$ denote its highest root. Then,
\[
(1-u^2\xx^{\theta})(1-u^2\xx^{\theta-\a_1})\cdot Z_\Phi(\xx; u)=
\frac{\displaystyle \sum_{w\in W^{2, 5}}\left.\frac{1}{(1-ux_2)(1-x_1 x_3)(1-ux_5)} \right|w }{\DD_{\Phi^{2, 5}}(\xx)} ~~.
\]
To see this, we apply Lemma~\ref{lem_invar} for $A_5$, the node  $i=5$, and $W^\prime=W^{2,5}$,  to conclude that it is enough to prove the corresponding equality for the residues at $x_5=1\slash u$. The equality of residues
reduces precisely to the equality considered above in Example \ref{exp: case1}.
\end{Exp}

\subsection{} \label{sec: UST}
We use Lemma \ref{lem_invar} to show that Theorem \ref{thm_main} reduces to proving the equality of the residues at $\evu$ of both sides of \eqref{e_main}. Before presenting the argument we need some technical preparation.

As before, fix an index $i$ such that
$\a_i$ is short. By Remark~\ref{rem_theta}, we  can restrict to nodes $i$ such that $n_i(\theta)\le 2$, for $\theta$ the highest root in $\Phi$, since this
assumption is implied by the hypothesis of Theorem~\ref{thm_main}. Denote
\be\label{e_ust}
U=\{\a\in \Phi^+: n_i(\a)=0 \},\quad S=\{\a\in \Phi^+: n_i(\a)=1 \},  \quad T=\{\a\in \Phi^+: n_i(\a)=2 \}.
\ee
Since $n_i(\theta)\le 2$, we have $\Phi^+=U\cup S \cup T$.  For $e\in\{0,\pm 1,\pm 2 \}$,  denote
\[
A_e=\{\a \in A: \la \a_i, \a\ra =e \},
\]
where $A$ is any subset of $\Phi$. By $A^s$, respectively $A^\ell$ we denote the short,
respectively long roots in the set~$A$.
\begin{lemma} \label{lem_setstruct} Let $\a_i$ be a short root such that $n_i(\theta)\le 2$. We have
\begin{enumerate}
 \item  $U_1=U_2=\emptyset $, $\ S_{-2}=\emptyset $, $\ S_2=\{\a_i\} $, $\ T_{-1}=T_{-2}=\emptyset$;
 \item  $U_{-1}=(\Phi^{i,+}\setminus \Phi^{(i),+})^s $, $\ U_{-2}=(\Phi^{i,+}\setminus \Phi^{(i),+})^\ell $,
 $\ U_{0}=\Phi^{(i),+} $;
\item
 $S_{1}=\{\a_i+\a: \a\in U_{-1} \} $, $\ S_{-1}=\{\a-\a_i: \a \in T_{1} \} $, $\ T_{2}=\{\a+2\a_i: \a \in U_{-2} \}$.
 \end{enumerate}
\end{lemma}
\begin{proof}
We use the fact that if $\b, \b^\prime\in \Phi$ with $\la \b, \b' \ra=-1$
then $\b+\b^\prime\in \Phi$, and if $\la \b, \b^\prime \ra=1$, then $\b-\b^\prime\in \Phi$.
Part (i) immediately follows, taking into account that there are no roots
$\a$ with $n_i(\a)>2$.

Clearly $\Phi^{(i),+}\subset U_0$, and since $n_i(\b)>0$ for
$\b\in \Pi_\new(\Phi_0)$, the other inclusion also holds. We have
$U=\Phi^{i,+}$ and, from part (i),  we have a disjoint union
$U=U_0\cup U_{-1}\cup U_{-2}$. Since $U_{-1}$, respectively $U_{-2}$ are the short,
respectively long roots in $U\setminus U_0$, the proof of part (ii) is finished.

The reflection $\ss_i$ gives bijections $U_{-1}\simeq S_1 $, $S_{-1}\simeq T_1$, $U_{-2}\simeq T_2$,
proving  part (iii).
\end{proof}

\subsection{}
Some subsets of the sets $U$, $S$, $T$ above are orbits under the parabolic groups $W^i$ or $W^{(i)}$.
\begin{lemma}\label{lem_setstruct2} Let $\a_i$ be a short root. Then,
\begin{enumerate}
\item  $W^i\a_i=S^s$;
\item $U_{-2}=W^{(i)} (\g-\a_i)$, where $\g\in \Pi_\new(\Phi_0)$ is the unique element such that $\g-\a_i\in (\Phi^{i})^\ell$;
\item If $n_i(\theta^s)=2$, then $W^i\theta^s=T^s$, where $\theta^s$ is the dominant short root in $\Phi$.
\end{enumerate}
\end{lemma}
\begin{proof}
(i) Since $n_i(w\a_i)=n_i(\a_i)$ for $w\in W^i$, the inclusion $W^i\a_i \subseteq S^s$ is clear. For the reverse inclusion, let $\a\in S^s$, and let $\b\in W^i\a$ of minimal height. We show that $\b\in W^i\a_i$. If $\b$ is a simple root, then
$\b=\a_i$ as $n_i(\b)=1$, and we are done. If $\b$ is not simple, let $\a_j$ simple such that $\ss_\b \a_j\in \Phi^-$. Therefore, $\la \b, \a_j \ra>0$, and if $j\neq i$, it follows that $\b>\ss_j\b\in W^i\b=W^i\a$, contradicting the minimality of $\b$. In consequence,  $j=i$, so $\la\b,\a_i\ra=1$ (since $\b$ is a short root),
and $\ss_i\b=\b-\a_i$ has $n_i(\b-\a_i)=0$. But, in this situation, $\ss_{\b-\a_i}\in W^i$ and $\ss_{\b-\a_i}\a_i=\b$, showing that $\b\in W^i\a_i$. In conclusion, $W^i\a_i=S^s$, finishing the proof of (i).

(ii) Let $\a\in U_{-2}$. Since $\a+2\a_i\in T_2$ is also a root, it follows that $\a+\a_i\in S_0^s$. The root  $\a+\a_i$ is short because
both $\a, \a+2\a_i$ are long. We consider two cases.

If $\Phi$ is of type $B_r$ or $F_4$, the root system $\Phi_0$ is irreducible, and there is a unique short root
$\g\in \Pi_\new(\Phi_0)$. The set $S_0^s$ consists of those roots $\g$ in $\Phi_0$ having $n_{\g}=1$
and, by part (i), we deduce $S_0^s=W^{(i)} \g$. Therefore,
$\a+\a_i=w\g$ for some $w\in W^{(i)}$, so $\g=w^{-1}\a+\a_i\in U_{-2}+\a_i$, as $W^{(i)}$ permutes both the
sets $U_{-2}$ and $S_0^s$. It follows that $U_{-2}=W^{(i)} (\g-\a_i) $.

If $\Phi$ is of type $C_r$, the root system $\Phi_0$ has an irreducible component of type $A_1$ generated by a short
root $\g\in \Pi_\new(\Phi_0)$. If $i=1$, or $i=r-1$, we have that $S_0^s=\{\g\}$; otherwise, part (i) implies that
$S_0^s=W^{(i)}\b \cup\{\g\}$, for $\b\in \Pi_\new^*(\Phi_0)$.
Since $\g$ satisfies $\g-\a_i\in (\Phi^i)^\ell$, but $\b$ does not, it follows
that $U_{-2}=W^{(i)}(\g-\a_i)=\{\g-\a_i\}$ is a set with one element.  This proves  part (ii).

(iii) The inclusion  $W^i\theta^s\subseteq T^s$ is clear.  The reverse inclusion follows if we show that if $\b\in T^s$ has the largest height in its $W^i$-orbit, then $\b$ is dominant. This is indeed the case. For $j\neq i$, we have $\la\beta,\a_j\ra\ge 0$, otherwise $\sigma_j(\beta)\in W^i\beta$ has larger height than $\beta$. Also, $\la\beta,\a_i\ra\ge 0$, otherwise  $n_i(\sigma_i(\beta))>2=n_i(\theta^s).$
\end{proof}

\subsection{}
As a consequence of Lemmas \ref{lem_setstruct} and \ref{lem_setstruct2}, we obtain
the following explicit description of the evaluation at $\evu$ of a zeta average
in a special situation. This formula will be used in the proof of Theorem \ref{thm_main}.
To ease notation, we write $\DD_{\Phi^i}^\ell(\xx)$ for $\DD_{({\Phi^i})^\ell}(\xx)$,
and $\DD_{\Phi^{(i)}}^\ell(\xx)$ for the $\DD_{({\Phi^{(i)}})^\ell}(\xx)$.

\begin{lemma}\label{lemma: special-double}
Assume $\Phi$ is double-laced, and $\a_i$ is a simple short root such that  $n_i(\theta)\le 2$. Let
 $\g$ be the unique root in $\Pi_\new(\Phi_0)\setminus \Pi_\new^*(\Phi_0)$,  and let $\Phi_0^\prime$ be the irreducible
component of $\Phi_0$ which contains $\g$. Then,
\[
\left. Z_{\Phi_0^\prime}(\xx; u)\right|_\evu=\frac{\DD^\ell_{\Phi^{(i)}}(\xx)}{\DD^\ell_{\Phi^i}(\xx)}~~.
\]
\end{lemma}
\begin{proof} By Lemma~\ref{lem_setstruct} (ii), we have $U_{-2}=(\Phi^{i,+}\setminus \Phi^{(i),+})^\ell $.
It follows that
\[
\frac{\DD^\ell_{\Phi^i}(\xx)}{\DD^\ell_{\Phi^{(i)}}(\xx)}=\prod_{\a\in U_{-2}}(1-\xx^\a)~~=\prod_{\a\in \a_i+U_{-2}}(1-u \ux^\a).
\]
By Lemma~\ref{lem_setstruct2} (ii), we have $U_{-2}+\a_i=W^{(i)}\g$, as the unique
$\g\in \Pi_\new(\Phi_0)\setminus \Pi_\new^*(\Phi_0)$  satisfies $\g-\a_i\in U_{-2}$.

If $\Phi$ is of type $C_r$ with $i\le r-1$, then $\Phi_0^\prime=\{\pm \g\}$ is of type $A_1$,
and $W^{(i)}\g=\{\g\}$. It follows that
\[\left.Z_{\Phi_0^\prime}(\xx; u)\right|_\evu= \frac{1}{1-u\ux^{\g}}=\frac{\DD^\ell_{\Phi^{(i)}}(\xx)}{\DD^\ell_{\Phi^i}(\xx)}.
\]

If $\Phi$ is of type $F_4$ with $i=4$, or  of type $B_r$ with $i=r$, then $\Phi_0=\Phi_0^\prime$ is irreducible
of type $B_3$ or, respectively, $B_{r-1}$, with $W^{(i)}$ generated by the simple reflections associated to long roots.
Proposition~\ref{prop_double} implies that
\[
\left.Z_{\Phi_0^\prime}(\xx; u)\right|_\evu
=\prod_{\a\in W^{(i)}\g} \frac{1}{1-u \ux^\a}=\frac{\DD^\ell_{\Phi^{(i)}}(\xx)}{\DD^\ell_{\Phi^i}(\xx)}.\qedhere
\]
\end{proof}

\subsection{} We are now ready to show that Theorem~\ref{thm_main} reduces to proving an equality of the residues at $\evu$.

\begin{Prop} \label{prop_res1} Let $\Phi$ be an irreducible root system not of type $G_2$, and let
$i$ be one of the admissible nodes specified in Table \ref{table1}. Then the identity \eqref{e_main}
in Theorem~\ref{thm_main} is equivalent to
\be\label{e_res}
\frac{\DD_{\Phi^i}(\xx)}{\DD_{\Phi^{(i)}}(\xx)} \res_{\evu}  Z_{\Phi}^{(i)}=
\frac{\displaystyle\sum_{w\in W^{(i)}} K_{\Phi,\a_i}(\xx)|w }{\DD_{\Phi^{(i)}}(\xx)} ~~.
\ee
\end{Prop}
\begin{proof}
We show that both sides of~\eqref{e_main} satisfy the assumptions of Lemma \ref{lem_invar}. We use the notation in \S\ref{sec: UST}. Since $n_i(\theta)\le 2$, and $W^i$ permutes the elements of $T^s$, we have
that $Z_{\Phi}^{(i)}(\xx)$  is invariant under the Chinta-Gunnells actions of $W^i$. The right-hand side of~\eqref{e_main} is also invariant under $W^i$.

The poles involving $x_i$ of the right-hand side of~\eqref{e_main} are precisely $\xx^{\a}=\pm 1/u$
for $\a$ in the orbit $W^i\a_i$, while the poles of the
left-hand side occur at $\xx^{\a}=\pm 1/u$ for $\a\in S^s$. By Lemma~\ref{lem_setstruct} (iv), we have
$W^i  \a_i=S^s$. Therefore condition (b) in Lemma~\ref{lem_invar} is satisfied.

The degree in $x_i$ of the right-hand side of~\eqref{e_main}
is clearly negative and, by Corollary~\ref{cor_degree}, we have
\[\deg_{x_i} Z_\Phi^{(i)}=4|T^s|-n_i(2\rho)+|S^\ell|+2|T^\ell|=
4|T^s|-|S|-2|T|+|S^\ell|+2|T^\ell|= 2|T^s|-|S^s|.
\]
 If $n_i(\theta^s)=1$, the inequality $|S^s|>2|T^s|$ is trivial (as $T^s=\emptyset$). If $n_i(\theta^s)=2$, by Lemma~\ref{lem_setstruct2}, the same inequality reduces to
\[\frac{|W^i|}{|W^{(i)}|} > 2 \frac{|W^i|}{|\stab_{\theta^s} W^{i} |}.
\]
This inequality can be directly verified in all cases in Table \ref{table1}. For
example, we include here the verification for $D_r$ and $2i\le r+1$, $i\ne 1$. In this case, we have $\stab_{W^{i}} \theta^s=W^{2,i}$, the Weyl group of  the parabolic root sub-system obtained by excluding the nodes $2$ and $i$ from the Dynkin  diagram
of $\Phi$. The inequality above is equivalent to
\[|W_{A_1}\times W_{A_{i-3}} \times W_{D_{r-i}}|>2 |  W_{A_{i-2}} \times W_{D_{r-i-1}}|.
\]
Since $|W_{A_r}|=(r+1)!$,
$|W_{D_r}|=2^{r-1} r!$, the inequality reduces to $3i<2r+1$, which is satisfied in the
range  $2i\le r+1$.

The residues of the right-hand side of~\eqref{e_main} at $x_i=\pm 1\slash u$
involve only the terms  in the sum that correspond to $w\in W^{(i)}=\stab_{W^i}\a_i$, and the residues at
$x_i=-1\slash u$ of both sides clearly vanish. Formula~\eqref{e_res} expresses
the equality of the residues at $x_i=1\slash u$ of the two sides of~\eqref{e_main}, and out conclusion follows from  Lemma~\ref{lem_invar}.
\end{proof}

\subsection{}

Before presenting the proof of Theorem \ref{thm_main}, we need to restate Theorem~\ref{thm_res} in terms of
the residue of $Z_\Phi^{(i)}$.
we collect one preliminary residue computation. To ease notation, we write $\DD_{\Phi^i}^s(\xx)$ for $\DD_{({\Phi^i})^s}(\xx)$, and $\DD_{\Phi^{(i)}}^s(\xx)$ for the $\DD_{({\Phi^{(i)}})^s}(\xx)$.
\begin{Prop}\label{thm_res1}
Let $\Phi$ be an irreducible root system not of type $G_2$, and $\a_i$ be a short simple root for which $n_i(\theta)\le 2$. We have,
\[
\frac{\DD_{\Phi^i}^s(\xx)}{\DD_{\Phi^{(i)}}^s(\xx)} \res_{\evu}  Z_{\Phi}^{(i)}(\xx; u)=Z_{\Phi_0}^{(i)}(\xx; u)|_{\evu}
\]

\end{Prop}
 \begin{proof}
Lemma \ref{lem_setstruct2} implies that $\Phi_1^+=S_1\cup T_1$.
From the description of $S_1$ and $U_{-1}$ in Lemma \ref{lem_setstruct2} it follows that
\[
\res_{\evu}  \frac{Z_{\Phi}^{[i]}(\xx; u)}{Z_{\Phi}^{(i)}(\xx; u)}
=\frac{\displaystyle\prod_{\a\in S_1 \cup T_1} (1-u^2\ux^{2\a})}{\displaystyle \prod_{\a\in T_0^s\cup T_1 } (1-u^2\ux^{2\a})}
=\frac{\DD_{\Phi^i}^s(\xx)}{\DD_{\Phi^{(i)}}^s(\xx)}
\frac{1}{\displaystyle\prod_{\a\in T_0^s}(1-u^2\ux^{2\a})} ~~.
\]
The conclusion follows from Theorem~\ref{thm_res}.
\end{proof}

\subsection{Proof of Theorem~\ref{thm_main}}\label{sec_induction}
We are now ready to assemble all the results in this section to prove Theorem~\ref{thm_main}.
By Proposition~\ref{prop_res1} and Theorem~\ref{thm_res1}, the identity~\eqref{e_main} reduces to
\be\label{e_ind}
\left. Z_{\Phi_0}^{(i)}(\xx; u)\right|_\evu=
\frac{\displaystyle \sum_{w\in W^{(i)}} K_{\Phi,\a_i}(\xx)|w }{\DD_{\Phi^{(i)}}(\xx)} \cdot
\frac{\DD^\ell_{\Phi^{(i)}}(\xx)}{\DD^\ell_{\Phi^i}(\xx)},
\ee
where $\DD^\ell_{\Phi}(\xx):=\DD_{\Phi^\ell}(\xx)=\prod_{\a\in(\Phi^\ell)^+} (1-u^2\xx^\a)$, as in Lemma \ref{lemma: special-double}.
We prove~\eqref{e_main} by induction on the rank of $\Phi$, by showing that the
identity~\eqref{e_ind} is of the same type but for a smaller rank root system. For the base cases, \eqref{e_ind}  is verified directly.

Throughout the proof, it is useful to refer to the tables in \S\ref{sec_Phi0}
and in Appendix~\ref{sec_exc} for the structure of~$\Phi_0$. Note that
when $\Phi_0$ is reducible, we have an orthogonal root system decomposition
\[\Phi_0=\Psi\oplus \Psi^\prime ,\]
with $\Psi$ irreducible and $\Psi'$ of type $A_1$, except when $\Phi$ is of type $D_4$, in which case $\Psi$ is of type $A_1\times A_1$.
Therefore, if $\Phi_0$ is reducible, $Z_{\Phi_0}(\xx; u)$ factors as
\be \label{e_fact}Z_{\Phi_0}(\xx; u)= Z_{\Psi}(\xx; u)\cdot Z_{\Psi'}(\xx; u).
\ee

\subsubsection{Simply-laced root systems} In this case, the fraction involving long roots in~\eqref{e_ind}
is not present. We have three cases.

(i) { $\Phi$ is of type $A_r$.} For $i=1$, or $i=r$, we have $\Phi_0=\Phi^{(i)}$,
and $K_{\Phi,\a_i}(\xx)=1$, so~\eqref{e_ind} holds by the definition of $Z_{\Phi^{(i)}}$.
If $1<i<r$, we have $\Pi_\new(\Phi_0)=\{\b\}$,
\be\label{e_KAr}
K_{\Phi,\a_i}(\xx)=\frac{1}{1-u\ux^{\b}}K_{\Phi_0,\b}(\xx),
\ee
and the root system  $\Phi^{(i)}$ is the parabolic sub-system of $\Phi_0$ obtained by removing the node $\b$ from its  Dynkin diagram. Therefore
formula~\eqref{e_ind} follows by induction on $r$,  the base cases being $i=1$ or $i=r$.

(ii) { $\Phi$ is of type $D_r$, $r\ge 4$.} For $i=1$, we have
\[
\Psi=\Phi^{(i)},\quad \Psi^\prime=\{ \pm \b\},\quad
K_{\Phi,\a_i}(\xx)=\frac{1}{1-u\ux^{\b}}=Z_{\Psi^\prime}(\xx)|_\evu.
\]
Therefore, \eqref{e_ind} follows from the definition of $Z_\Psi(\xx; u)$ and the factorization~\eqref{e_fact}.

For $i=r$ (and, similarly, for $i=r-1$), we have $\Psi^\prime=\{\pm \a_{r-1}\}$, and $\Psi$ is of type $D_{r-2}$, with $\b$
playing the role of the node $r-2$. Therefore~\eqref{e_ind} follows by induction on $r$,
with the base case being $r=4$, $i=1$.

For $1<i<r-2$, we have  $\Psi^\prime=\{\pm \b^\prime\}$ and $\Psi$ is of type $D_{r-2}$, with $\b$
playing the role of the node $i-1$. Since $0\le r+1-2i=r-2+1-2(i-1)$, and
\be\label{e_KDr}
K_{\Phi,\a_i}(\xx)=\frac{1}{1-u\ux^{\b}}K_{\Phi_0,\b}(\xx)\cdot Z_{\Psi'}(\xx; u)|_\evu,
\ee
formula~\eqref{e_ind} follows again by induction on $r$.
The base cases are $i=1$, which was already proved, and $r=4, 5$, and $i=r-2$.

For $r=4$, $i=2$, we have that $W^{(i)}$ is trivial, $\Phi_0^+=\{\pm\b\}\oplus \{\pm\b'\}\oplus\{\pm\b''\}$,
and~\eqref{e_ind} follows from the definition of $K_{\Phi,\a_i}(\xx)$.
The case of $D_5$ with $i=3$ is different from those encountered so far, since $\Psi$ contains two
roots $\b,\b'\in\Pi_\new(\Phi_0)$. In this case, $\Psi$ is of type $A_3$ and $\b$, $\b^\prime$ play the role of nodes $1$ and $3$.
Therefore, formula~\eqref{e_ind} reduces to the identity discussed in Example \ref{exp: case2}.

(iii) { $\Phi$ is of type $E_r$, $6\le r\le 8$.} If $i$ is an extremal node, then $\Pi_\new(\Phi_0)=\{\b\}$.
The root system $\Phi_0$ is irreducible, and $\Phi^{(i)}$ is the parabolic sub-system of $\Phi_0$
obtained by removing node $\b$ from its Dynkin diagram. By definition,
$K_{\Phi,\a_i}$ satisfies~\eqref{e_KAr}, and~\eqref{e_ind} follows by induction.

The only case when the node $i$ is not extremal is  for $E_6$ and $i=3$. In this case, $\Phi_0$ is of type $A_5$ and $\b$, $\b^\prime\in\Pi_\new(\Phi_0)$ play the role of nodes $2$ and $5$. Therefore, formula~\eqref{e_ind} reduces to the
identity proved in Example \ref{exp: case3}. This concludes the proof of Theorem~\ref{thm_main} for simply-laced root systems.

\subsubsection{Double-laced root systems} We distinguish two cases.

(i) { $\Phi$ is of type $B_r$, $r \ge 3$, or $F_4$.} In this case $i=r$, or $i=4$, respectively.
Then, $\Phi^{(i)}$ consists of long roots only, and $K_{\Phi,\a_i}(\xx)=1$. The first fraction in~\eqref{e_ind} equals $1$, by the Weyl denominator formula for $A_{r-2}$ and, respectively, for $A_2$.
Formula~\eqref{e_ind} is then precisely the formula proved in  Lemma \ref{lemma: special-double}.

(ii){ $\Phi$ is of type $C_r$, $r \ge 2$}. In this case, $\Psi$ is of type $C_{r-2}$, and $\Psi^\prime=\Phi_0^\prime=\{\pm \g\}$ with the notation of Lemma \ref{lemma: special-double}.  If $i=1$, we have $K_{\Phi,\a_i}(\xx)=1$, and formula~\eqref{e_ind} follows from Lemma \ref{lemma: special-double} and the definition of $Z_\Psi(\xx; u)$. If $i>1$, then $2i\le r$ implies that $i<r-1$. Therefore, $\Psi$ contains a root $\b\in \Pi_\new^*(\Phi_0)$ that plays the role of node $i-1$ in its Dynkin diagram.
Formula~\eqref{e_KAr} holds for $K_{\Phi,\a_i}(\xx)$, and~\eqref{e_ind} follows by induction and the use of  Lemma \ref{lemma: special-double}, the base case being $i=1$. This completes the proof of Theorem~\ref{thm_main}.

\appendix
\section{Exceptional simply-laced root systems}\label{sec_exc}

We consider the root systems of type $E_6$, $E_7$, and $E_8$, and
describe the orthogonal complement $\Phi_0$ to a simple root $\a_i$, as defined in Section~\ref{sec: roots}. We will use the notation set-up at the beginning of \S\ref{sec_Phi0}.
The description of  $\Pi_\new(\Phi_0)$ and the Dynkin diagram of $\Phi_0$ can be found in the relevant table below. We mark in boldface the indices satisfying the conditions in Lemma~\ref{lem_cond}.
Using the information in the tables, one can verify in these cases the equivalence of conditions (i) and (ii) in Lemma~\ref{lem_cond}, as well as Remark~\ref{rem_theta}.
We first recall the standard labeling of the Dynkin diagrams and the formula for the longest root $\theta$:
\begin{itemize}
 \item $E_6$: \dynkin[labels={\a_1, \a_2,\a_3,\a_4,\a_5,\a_6 }]E6, \ $\theta=\a_1+2\a_2+2\a_3+3\a_4+2\a_5+\a_6$;
 \item $E_7$: \dynkin[labels={\a_1, \a_2,\a_3,\a_4,\a_5,\a_6,\a_7 }]E7, \ $\theta=2\a_1+2\a_2+3\a_3+4\a_4+3\a_5+2\a_6+\a_7$;
 \item $E_8$: \dynkin[labels={\a_1, \a_2,\a_3,\a_4,\a_5,\a_6,\a_7,\a_8 }]E8 , \  $\theta=2\a_1+3\a_2+4\a_3+6\a_4+5\a_5+4\a_6+3\a_7+2\a_8$.
\end{itemize}

\begin{table}[!ht]
\begin{center}\small
\begin{tabular}{lcc}
 Node $i$ & $\Pi_\new(\Phi_0)$  &  $\Phi_0$ \\\hline
$\mathbf{i=1}$ & $\b=\theta_{\{1,\ldots,5\}}$ &  \dynkin[labels={\a_2,\a_4,\a_5,\a_6,\b,\a_7}]D6\\ \hline
$\mathbf{i=2}$ & $\b=\theta_{\{2,\ldots,5\}}$ & \dynkin[labels={\a_3,\a_1,\b,\a_6,\a_5,\a_7}]D6\\ \hline
$i=3$ & \makecell{$\b=\theta_{\{1,3,4\}}$\\
$  \b'=\theta_{\{2,\ldots,5\}}  $} & \dynkin[labels={\a_2,\b,\a_5,\a_6,\a_7,\b'}]D6\\ \hline
$i=4$ & \makecell{$\b=\theta_{\{3,4,5\}}$\\
$  \b'=\theta_{\{2,3,4\}}  $\\
$  \b''=\theta_{\{2,4,5\}}  $} & \dynkin[labels={\b',\a_1,\b,\a_6,\a_7,\b''}]D6\\ \hline
$i=5$ & \makecell{$\b=\theta_{\{4,5,6\}}$\\
$  \b'=\theta_{\{2,\ldots,5\}}  $} & \dynkin[labels={\b',\a_1,\a_3,\b,\a_2,\a_7}]D6\\ \hline
$i=6$ &
   \makecell{ $\b=\theta_{\{5,6,7\}}$ \\
   $\b'=\theta_{\{2,\ldots,6\}}$  }&   \dynkin[labels={\b',\a_1,\a_3,\a_4,\b,\a_2}]D6  \\ \hline
$\mathbf{i=7}$ & $\b=\theta_{\{2,\ldots,7\}}$ &  \dynkin[labels={\b, \a_1,\a_3,\a_4,\a_5,\a_2}]D6\\ \hline
\end{tabular}\vspace{1em}
\end{center}\caption{\label{table8} The orthogonal root system in type $E_7$}
\end{table}
\newpage
\begin{table}[!ht]\small
\begin{center}
\begin{tabular}{lcc}
Node $i$ & $\Pi_\new(\Phi_0)$  & $\Phi_0$ \\\hline
$\mathbf{i=1}$ & $\b=\theta_{\{1,\ldots,5\}}$ &  \dynkin[labels={\a_2,\a_4,\a_5,\a_6,\b}]A5 \\\hline
$\mathbf{i=2}$ & $\b=\theta_{\{2,\ldots,5\}}$ & \dynkin[labels={\a_3,\a_1,\b,\a_6,\a_5}]A5   \\\hline
$\mathbf{i=3}$ & \makecell{$\b=\theta_{\{1,3,4\}}$ \\
    $\b'=\theta_{\{1,\ldots,5\}}$} & \dynkin[labels={\a_2,\b,\a_5,\a_6,\b'}]A5 \\\hline
$i=4$ & \makecell{$\b=\theta_{\{3,4,5\}}$ \\
    $\b'=\theta_{\{2,3,4\}}$\\
    $\b''=\theta_{\{2,4,5\}}$
    } & \dynkin[labels={\b',\a_1,\b,\a_6,\b''}]A5 \\\hline
\end{tabular}\vspace{1em}
\end{center}\caption{\label{table7} The orthogonal root system in type $E_6$}
\end{table}

\begin{table}[!ht]\small
\begin{center}
\begin{tabular}{lcc}
Node $i$ & $\Pi_\new(\Phi_0)$  & $\Phi_0$ \\\hline
$\mathbf{i=1}$ & $\b=\theta_{\{1,\ldots,5\}}$ &  \dynkin[labels={\a_8,\b,\a_7,\a_6,\a_5,\a_4,\a_2}]E7\\ \hline
$i=2$ & $\b=\theta_{\{2,\ldots,5\}}$ & \dynkin[labels={\a_8,\a_5,\a_7,\a_6,\b,\a_1,\a_3}]E7\\ \hline
$i=3$ & \makecell{$\b=\theta_{\{1,3,4\}}$\\
$  \b'=\theta_{\{2,\ldots,5\}}  $} & \dynkin[labels={\a_8,\b',\a_7,\a_6,\a_5,\b,\a_2}]E7\\ \hline
$i=4$ & \makecell{$\b=\theta_{\{3,4,5\}}$\\
$  \b'=\theta_{\{2,3,4\}}  $\\
$  \b''=\theta_{\{2,4,5\}}  $} & \dynkin[labels={\a_8,\b'',\a_7,\a_6,\b,\a_1,\b'}]E7\\ \hline
$i=5$ & \makecell{$\b=\theta_{\{4,5,6\}}$\\
$  \b'=\theta_{\{2,\ldots,5\}}  $} & \dynkin[labels={\a_8,\a_2,\a_7,\b,\a_3,\a_1,\b'}]E7\\ \hline
$i=6$ &
   \makecell{ $\b=\theta_{\{5,6,7\}}$ \\
   $\b'=\theta_{\{2,\ldots,6\}}$  }&   \dynkin[labels={\a_8,\a_2,\b,\a_4,\a_3,\a_1,\b'}]E7 \\ \hline
$i=7$ &
   \makecell{ $\b=\theta_{\{6,7,8\}}$ \\
   $\b'=\theta_{\{2,\ldots,7\}}$  }&   \dynkin[labels={\b,\a_2,\a_5,\a_4,\a_3,\a_1,\b'}]E7 \\ \hline
$\mathbf{i=8}$ & $\b=\theta_{\{2,\ldots,8\}}$ &  \dynkin[labels={\a_6,\a_2,\a_5,\a_4,\a_3,\a_1,\b}]E7\\ \hline
\end{tabular}\vspace{1em}
\end{center}\caption{\label{table9} The orthogonal root system in type $E_8$}
\end{table}

\newpage
\section{The root system of type \texorpdfstring{$G_2$}{G2}}\label{sec_G2}

\subsection{} Throughout this section we assume that the root system $\Phi$ is of type $G_2$: \hspace{2mm} \dynkin[reverse arrows, labels={\a_1,\a_2}]G2

 The dominant short root is $\theta_s=2\a_1+\a_2$ and the dominant long root is $\theta_\ell=3\a_1+2\a_2$.
 As in Section~\ref{sec: roots}, we fix a node $i$ and consider the orthogonal complement $\Phi_0=\a_i^\perp:=\{\a\in\Phi : \la\a_i, \a\ra =0 \}.$
 \begin{table}[!ht]
\begin{center}
\begin{tabular}{l|c|c}
Node $i$ & $\Pi(\Phi_0)$  & $\Phi_0$ \\\hline
$i=1$ & $\b=\theta_\ell$ &  $A_{1}^\ell$ \\\hline
$i=2$ & $\b=\theta_s$ & $A_{1}^s $  \\\hline
\end{tabular}\vspace{1em}
\end{center}\caption{\label{table10} The orthogonal root system in type $G_2$}
\end{table}
The Dynkin diagram of $\Phi_0$ is of rank one, with basis  $\Pi(\Phi_0)=\{\b\}$ as described in Table \ref{table10}.

We define
$$
Z_{\Phi}^{[i]}(\xx; u)=Z_{\Phi}(\xx; u)\cdot \prod_{\substack{\a\in\Phi^+_{>0}\\m_{\a}=2 }} (1-u^2\xx^{2\a}),
$$
where $\Phi^+_{>0}=\{\a\in \Phi^+\mid \la\a,\a_i\ra>0\}$. The condition $m_{\a}=2$ is superfluous,
as it holds for all $\a\in\Phi$, but we include it since this definition is consistent with the one
for  simply-laced and double-laced root systems.

\subsection{}

The zeta average $Z_{\Phi_0}(\xx; u)$ defined using the basis $\Pi(\Phi_0)$
will henceforth be regarded  as an element of $\F(Q)$ via $\F(Q_0)\subset \F(Q)$. Since $\Phi_0=\{\pm \b\}$ is of type $A_1$ in both cases, we have
$Z_{\Phi_0}(\xx; u)=1\slash(1-u\xx^\b).$

\begin{Thm} \label{thm_res_triple}
We have
\[
\res_{x_1=1\slash u} Z_{\Phi}^{[1]}(\xx;u)=
\left.Z_{\Phi_0}(\xx;u^3)\right|_{\substack{x_{1}=1\slash u }}\quad \text{and} \quad
\res_{x_2=1\slash u} Z_{\Phi}^{[2]}(\xx;u)=
\left.Z_{\Phi_0}(\xx;u)\right|_{x_{2}=1\slash u}.
\]
\end{Thm}
Remark the extra change of variable in the case $i=1$, which is a singular feature of the $G_2$ case.

 We adopt the notation in Section \ref{sec: roots} with respect to the parabolic sub-system $\Phi^i$ and its Weyl group $W^i$. Also,  for $\a\in\Phi$,  $n_i(\a)\in \Z$ denotes the coefficient of $\a_i$ in the expansion of $\a$ in the basis $\Pi(\Phi)$.
Let
\[Z_{\Phi}^{(i)}(\xx; u)=Z_{\Phi}(\xx; u)\cdot
\prod_{\substack{\a\in\Phi\\ n_i(\a)\ge 2 }} (1-u^2\xx^{2\a}).
\]
To state the analogue of Theorem~\ref{thm_main}, remark that  $\a_2$ is the only simple root $\a_i$ for which
$n_i(\b)=1$ (see Lemma~\ref{lem_cond}). This is is also the only simple root  for which $n_i(\theta_\ell)\le 2$ (see Remark~\ref{rem_nitheta}).
The product in the definition of $Z_{\Phi}^{(2)}(\xx; u)$ contains only one term, for $\a=\theta_\ell$.
\begin{Thm}\label{thm_main_triple}We have
\[
Z_{\Phi}^{(2)}(\xx; u)=\frac{\displaystyle\sum_{w\in W^2}\left.\frac{1}{1-ux_2}\frac{1}{1-u\xx^{\theta_s}} \right|w }{\DD_{\Phi^2}(\xx; u)}~~.
\]
\end{Thm}
The presence of two terms involving $u$ in the average above is explained by the fact that the
roots $\a$ with $n_2(\a)=1$ form two orbits under the group $W^2=\la\ss_1\ra$: one orbit consisting of long roots with
representative $\a_2$, and one orbit consisting of short roots with representative $\theta_s$.

Theorems \ref{thm_res_triple} and Theorem \ref{thm_main_triple} can be proved along the same lines as Theorems~\ref{thm_res} and~\ref{thm_main}.
However, they can be  verified directly, using the explicit formula
\[Z_{\Phi}(x_1,x_2; u)=\frac{u^5  x_1^7  x_2^4 - u^3 x_1^6  x_2^3  -u^3   x_1^4  x_2^3 + u^2 x_1^4  x_2^2  + u^3 x_1^3  x_2^2  -  u^2 x_1^3  x_2  - u^2  x_1  x_2 + 1    }{D_{\Phi}(\xx; u)/\left[(1+u x_1)(1+u x_2)(1+u\xx^{\theta_s})\right]},
\]
where $D_{\Phi}(\xx; u)=\prod_{\a\in\Phi^+}(1-u^2\xx^{2\a})$.

\section{Proof of Theorem~\ref{thmC}}\label{appC}

\subsection{}
In this appendix we give a proof of Theorem~\ref{thmC}. We let $\mathbb{K}=\Q(\sqrt{-1})$ and $\Phi$
an irreducible root system not of type $G_2$.
The argument given here applies with obvious modifications to give an
alternative proof of Theorem~\ref{thmB} over $\F_q(T)$ with $q\equiv 1 \pmod 4$.
What simplifies the argument,
and guides our choice of number field and the congruence condition in Theorem~\ref{thmB},
is the fact that the quadratic reciprocity law takes the simple shape $\qres{a}{b}=\qres{b}{a}$
under these assumptions, for $a,b$ coprime ideals of odd norm in $\Q(\sqrt{-1})$, or coprime
monic polynomials in $\F_q(T)$ with  $q\equiv 1 \pmod 4$. We emphasize that these assumptions are made
only to simplify the arguments, and similar results hold over arbitrary number fields. However in general 
one needs to consider MDS twisted by characters, as introduced in~\cite{CG}, and the statements are more 
involved. 

The idea of the proof is straightforward: we show that both sides of~\eqref{eqC} are
multiple Dirichlet series with the same $p$-part, 
and they satisfy the same twisted multiplicativity property. 
First, in Lemma~\ref{lem_resglobal} we derive a formula for the residue of
$\ZZ_\Phi(\bs)$ as an MDS in $s_j$ for $j\ne i$. Using this formula,
we show that both sides of~\eqref{eqC} have the same $p$-part; it is here that we crucially use
Theorem~\ref{thmA}, which is the main difficulty in the argument. Using again the formula
in Lemma~\ref{lem_resglobal}, we show that both sides of~\eqref{eqC} satisfy the 
same twisted multiplicativity, inherited from the root system $\Phi_0$.

\subsection{} We recall the definition of the MDS $\ZZ_\Phi(\bs)$, following~\cite{CG}. We have
$$
\ZZ_\Phi(\bs) = \sum_{} \frac{H(m_1,\dots,m_r)}{|m_1|^{s_1}\cdot\ldots\cdot |m_r|^{s_r}},
$$
where the sum is over integers $m_j$ in $\mathbb{K}$ of odd norm, modulo units, and the norms are the norms
of the principal ideals generated by $m_j$. In what follows we use the language of ideals, and we regard 
the $m_j$ as integral ideals  in $\mathbb{K}$ of odd norm. 
The coefficients $H$ satisfy the following properties, which uniquely determine $\ZZ_\Phi$.
\begin{itemize}
 \item Twisted multiplicativity: if the ideals $\prod m_j$ and $\prod m_j'$ are coprime, then\footnote{
 Here we assume that $\Phi$ is not of type $G_2$; for $G_2$ the condition in the product would be
 $\la \a_k , \a_j\ra<0$. Recall also that the Weyl invariant pairing is normalized as in \S\ref{sec2.1}. }
\be H(m_1 m_1', \dots, m_r m_r')=H(m_1, \dots, m_r)H(m_1', \dots, m_r')\cdot
\prod_{\substack{k<j\\ \la \a_k , \a_j\ra=-1 }}
 \qres{m_k}{m_j'} \qres{m_k'}{m_j};
\ee
 \item Determination of $p$-part: for a prime $p$ and $\lam=\sum n_j \a_j \in Q^+$, we have
\be\label{e_ppart}
H(p^{n_1}, \dots, p^{n_r})=a_\lam(|p|^{-1 \slash 2}),
\ee
where $a_\lam(u)$ are the coefficients of the zeta average $Z_\Phi(\xx;u)=\sum_\lam a_\lam(u) \xx^\lam$
defined in \S \ref{secZCG}.
\end{itemize}
The analytic properties of $\ZZ_\Phi(\bs)$ have been established in~\cite{CG}. In particular,
it has meromorphic continuation to $\CC^r$ and satisfies a group of functional 
equations isomorphic to the Weyl group of $\Phi$. 


\subsection{}
In this subsection, we prove the following formula for the residue of the MDS over $\mathbb{K}$.
\begin{lemma}\label{lem_resglobal}
Let $\mathbb{K}=\mathbb{Q}(\sqrt{-1})$. Let $\Phi$ be an irreducible root system not of type $G_2$,
and let $\a_i$ be a short simple root. Then,
\[\res_{s_i=1\slash 2} \ZZ_\Phi(\bs)=\frac{\pi}{8}
\sum_{ \substack{m_j, j\ne i\\ \prod_{\la\a_j,\a_i\ra=-1}m_j=\square } }
\frac{1}{\prod_{j\ne i} |m_j|^{s_j}}
\prod_{p|\prod_{j\ne i} m_j} (1-|p|^{-1})
\sum_{\substack{m\\ p|m\Rightarrow p|\prod_{j\ne i} m_j} } \frac{H(m_1,\ldots, m,\ldots, m_r) }{|m|^{1\slash 2}},
\]
where the ideal $m$ is on position $i$ in the argument of $H$, and the sums are over integral ideals of odd norm. The series converges for $\Re s_j$ large enough for $j\ne i$. 
\end{lemma}
\noindent By essentially the same argument, the same residue formula, but without the factor $\pi\slash 8$, holds over $\mathbb{K}=\F_q(T)$ for $q\equiv 1\!\! \pmod 4$. 
\begin{proof}
One sums first over $m_i$, keeping $m_j$ fixed for $j\ne i$, as in~\cite{CG}*{\S5}. 
From this sum one extracts a Dirichlet series with quadratic character, whose 
residue is 0 unless the character is trivial. In the latter case, we use the 
Dirichlet class number formula to compute the residue
\[\res_{s=1} \zeta_{\mathbb{K}}^{(2 \prod_{j\ne i} m_j)}(s)=
\frac{\pi}{4} \prod_{p\vert 2\prod_{j\ne i} m_j} (1-|p|^{-1}),
\]
where $\zeta_{\mathbb{K}}^{(c)}$ is the Dedekind zeta function of $\mathbb{K}$ with the Euler factors at the primes
dividing $c=2\prod_{j\ne i} m_j$ removed. The conclusion immediately follows.
\end{proof}

\subsection{}

Using the formula in Lemma~\ref{lem_resglobal}, we now show that the $p$-parts of
both sides in~\eqref{eqC} match.
We denote by $\mathcal{L}(\us)$ the series in Lemma~\ref{lem_resglobal} as an MDS in the multivariable
$\us=(s_1, \dots,s_{i-1},s_{i+1}, \dots s_r)$:
\be\label{e_MDS}
\mathcal{L}(\us):=\sum_{\substack{\um\\ \prod_{\la\a_j,\a_i\ra=-1}m_j=\square}}\frac{H'(\um)}{\prod_{j\ne i} |m_j|^{s_j}},
\ee
with $H'(\um)$ as resulting from Lemma~\ref{lem_resglobal} and  $\um=(m_1, \dots,m_{i-1}, m_{i+1}, \dots , m_r)$.
We leave aside for now the question as to what root system is the MDS $\mathcal{L}(\us)$ attached to.

To compute its $p$-part for $p$ a prime of odd norm,
let $m_j=p^{k_j}$ for $j\ne i$, $m=p^{k_i}$ and
make the change of variables $x_j=|p|^{-s_j}$ for $j\ne i$, $x_i=|p|^{-1/2}$. 
Because of the $p$-part property~\eqref{e_ppart}, we also denote 
$u=|p|^{-1/2}$. 
Let  $L_p(\ux;u)$ be the $p$-part of $\mathcal{L}(\us)$, after the substitutions above,
where $\ux$ denotes, as before, the multivariable $(x_1, \ldots, x_{i-1},x_{i+1},\ldots, x_r)$. 
Let  $R_p(\ux;u)$ denote the $p$-part of the right-hand side of~\eqref{eqC} (without the factor
$\pi\slash 8$), after the same  substitutions.

\begin{lemma}\label{lem_pparts}
With the notation above, we have
\[L_p(\ux;u)=R_p(\ux;u). \]
\end{lemma}

\begin{proof}
By definition, we have
\[\begin{aligned}
L_p(\ux;u)&=
\sum_{\substack{\lam\in Q^+ \\ \la\lam , \a_i\ra\ \mathrm{even}}}
(1-u^2) a_\lam(u) \xx^\lam|_{x_i=u}  \\
&=(1-u^2) (Z_{\Phi})^+_i (\xx;u)|_{x_i=u}, 
\end{aligned}
\]
where the notation $f^{+}_i$ for a function $f(\xx;u)$ is introduced in  \S \ref{secpm}.

On the other hand, the $p$-part of $\ZZ_{\Phi_0}(\bs)|_{s_i=1\slash 2}$ is, by definition and after the substitutions above,
$Z_{\Phi_0}(\xx;u)|_{x_i=u}$, and we can write
\[
R_p(\ux;u)= \prod_{\substack{\a\in \Phi^+\\ \la\a, \a_i\ra=1}}
\frac{1}{1-\xx^{2\a}|_{x_i=u}}  \cdot  Z_{\Phi_0}(\xx;u)|_{x_i=u}.
\]
We now use Theorem~\ref{thmA} to express the term $Z_{\Phi_0}(\xx;u)|_{x_i=u}$. 
To apply the theorem, we need to replace 
the evaluation of $x_i$ at~$u$ with an evaluation at $1\slash u$. 
We use repeatedly the following simple change of variables formula: if $f(\xx;u)$ is any function such that
the evaluations below are well-defined, then
\be\label{eFact}
f(\ss_i \xx;u)\big|_{\substack{x_i=1\slash u\\ x_j\mapsto u^{-\la\a_j,\a_i \ra}x_j } }=f(\xx;u)\big|_{x_i=u },
\ee
where here and below the substitution  $x_j\mapsto u^{-\la\a_j,\a_i \ra}x_j$
takes place for all $j\ne i$. It follows that
\[
Z_{\Phi_0}(\xx;u)|_{x_i=u}=Z_{\Phi_0}(\ss_i\xx; u)|_{\substack{x_i=1\slash u\\  x_j\mapsto u^{-\la\a_j,\a_i \ra}x_j}}
=Z_{\Phi_0}(\xx; u)|_{\substack{x_i=1\slash u\\  x_j\mapsto u^{-\la\a_j,\a_i \ra}x_j}},
\]
where the second equality uses the fact that $Z_{\Phi_0}(\xx; u)\in \F(Q_0)\subset \F(Q)$.
By~\eqref{eFact}, we also have
\[
\prod_{\la\a, \a_i\ra=1} (1-\xx^{2\a})|_{x_i=u}=\prod_{\la\a, \a_i\ra=1} (1-u^2 \xx^{2\ss_i \a})|_{x_i=u}
=\prod_{\la\a, \a_i\ra=1} (1-u^2 \xx^{2\a})|_{\substack{x_i=1\slash u\\  x_j\mapsto u^{-\la\a_j,\a_i \ra}x_j} }.
\]
Using this identity and applying Theorem~\ref{thmA}, we obtain the first equality below
\[\begin{aligned}
R_p(\ux;u)&=\res_{x_i=1\slash u} Z_{\Phi}(\xx;u)|_{ x_j\mapsto u^{-\la\a_j,\a_i \ra}x_j }\\
&= (1-u^2) (Z_{\Phi})^+_i(\ss_i\xx;u)|_{\substack{x_i=1\slash u\\  x_j\mapsto u^{-\la\a_j,\a_i \ra}x_j} } \\
&=(1-u^2)(Z_{\Phi})^+_i (\xx;u)|_{x_i=u} .
\end{aligned}
\]
The second equality follows from~\eqref{e_ressi}, while in the third we use again~\eqref{eFact}.
Comparing with the formula for $L_p(\ux;u)$ above, we conclude that $R_p(\ux; u)=L_p(\ux; u)$.
\end{proof}

\subsection{}
Using Lemma~\ref{lem_resglobal}, the identity~\eqref{eqC} becomes 
\be\label{e_twisted}
\prod_{{\substack{\a\in\Phi^+ \\ \la\a,\a_i \ra=1}}} 
\zeta_{\mathbb{K}}^{(2)}({2{\bs_\a}})^{-1}{\vert_{s_{i}=1\slash 2}} \cdot \mathcal{L}(\us)
=\ZZ_{\Phi_0}(\bs)|_{s_i=1\slash 2},
\ee
with $\mathcal{L}(\us)$ defined in~\eqref{e_MDS}.
 In the previous subsection we have shown that the $p$-parts of both sides match,
and now we show that both sides satisfy the same twisted multiplicativity property. 
The product of zeta functions on the left does not affect the twisted multiplicativity, 
so we concentrate on the coefficients $H'(\um)$ of $\mathcal{L}(\us)$.

Let $\um$, $\um'$ be tuples as in the summation defining~$\mathcal{L}(\us)$, and $m, m'$ ideals  of odd norm
such that $m\prod_{j\ne i} m_j$ and $m' \prod_{j\ne i} m_j'$ are coprime. Using twisted
multiplicativity for $H(m_1 m_1, \dots, mm', \dots, m_r m_r')$ under the condition that
$\prod_{\la\a_j,\a_i\ra=-1} m_j$ and $\prod_{\la\a_j,\a_i\ra=-1} m_j'$ are squares, 
one checks that the residue symbols involving $m,m'$  multiply to 1, 
so the formula for $H'$ gives
\be\label{e_Hprime}
H'(\um\cdot \um')=H'(\um)\cdot H'(\um') \cdot 
\prod_{\substack{k<j\\ \la \a_k , \a_j\ra=-1 }}
 \qres{m_k}{m_j'} \qres{m_k'}{m_j}.
\ee

To illustrate the last part of the argument in a concrete situation, and to simplify the notation, let us assume that $\Phi=A_r$, and $1<i<r$. In this case,  
$$\Pi(\Phi_0)=\{\a_j \mid |j- i|>1  \} \cup \{\b\}, \quad \text{with $\b=\a_{i-1}+\a_i+\a_{i+1}$.}  $$ 
Denote by $H''(\um)$  the coefficients of the left-hand side of~\eqref{e_twisted} when 
written as a MDS. They satisfy the same twisted multiplicativity as $H'(\um)$, and the
the matching of $p$-parts of both sides in~\eqref{e_twisted} shows that $H''(\um)=0$ unless 
$m_{i-1}=m_{i+1}$. Property~\eqref{e_Hprime} then reduces to the twisted multiplicativity satisfied 
by the coefficients of~$\ZZ_{\Phi_0}(\bs)|_{s_i=1\slash 2}$ with respect to the system $\Phi_0$.  
Together with Lemma~\ref{lem_pparts}, this finishes the proof of~\eqref{eqC} in this particular case. 

The general case is entirely similar, but it requires heavier notation, so we leave the verification 
to the interested reader. 


\end{document}